\documentclass{amsart}
\usepackage{amsmath,amssymb,stmaryrd}
\usepackage{graphicx}
\usepackage{amsrefs}
\def \rr {\mathbb{R}}
\def \nn {\mathbb{N}}
\def \rn {\rr^n}

\def \huno {H^1_0(\Omega)}
\def \dundeux {D^{1,2}(\Omega)}
\def \crit {2^{\star}}
\def \ue {u_\epsilon}

\def \eps {\epsilon}

\def \crits {2^\star(s)}
\newtheorem{theorem}{Theorem}[section]
\newtheorem{rem}[theorem]{Remark}
\newtheorem{defi}{Definition}
\newtheorem{proposition}[theorem]{Proposition}
\newtheorem{coro}[theorem]{Corollary}

\def \rr {\mathbb{R}}
\def \nn {\mathbb{N}}
\def \rn {\mathbb{R}^n}
\def \rnp {\mathbb{R}^n_+}

\def \ue {u_\eps}

\def \eps {\varepsilon}
\def \ap {\alpha_+(\gamma)}
\def \am {\alpha_-(\gamma)}
\def \bp {\beta_+(\gamma)}
\def \bm {\beta_-(\gamma)}
\def \aps {\alpha_+}
\def \ams {\alpha_-}
\def \dundeuxr {D^{1,2}(\rn)}
\def \dundeuxrnp {D^{1,2}(\rnp)}
\def \dundeux {D^{1,2}(\Omega)}
\def \eucl {\hbox{Eucl}}

\def \cone {{\mathcal C}}
\def \dundeuxc {D^{1,2}(\mathcal C)}

\title[Hardy-Sobolev inequalities]{Sobolev inequalities for the Hardy-Schr\"odinger operator: Extremals and critical dimensions}
\author{Nassif Ghoussoub}
\address{Nassif Ghoussoub, Department of Mathematics, 1984 Mathematics Road,The University of British Columbia, BC, Canada V6T 1Z2}
\email{nassif@math.ubc.ca}

\author{Fr\'ed\'eric Robert}
\address{Fr\'ed\'eric Robert, Institut \'Elie Cartan, Universit\'e de Lorraine, BP 70239, F-54506 Vand{\oe}uvre-l\`es-Nancy, France}
\email{frederic.robert@univ-lorraine.fr}

\thanks{This work was carried out while N. Ghoussoub was visiting l'Institut \'Elie Cartan, Universit\'e de Lorraine. He  was partially supported by a research grant from the Natural Science and Engineering Research Council of Canada (NSERC)}

\begin{document}
\begin{abstract} In this expository paper, we consider the Hardy-Schr\"odinger operator $L_\gamma:=-\Delta -\frac{\gamma}{|x|^2}$ on a smooth domain $\Omega$ of $\rn$ with $0\in \overline \Omega$, and describe how the location of the singularity $0$, be it in the interior of $\Omega$ or on its boundary, affects its analytical properties. We compare the two settings by considering the optimal Hardy, Sobolev, and the Caffarelli-Kohn-Nirenberg  inequalities. The latter can be stated as:  \[
\hbox{$C\left(\int_{\Omega}\frac{u^{2^*(s)}}{|x|^s}dx\right)^{\frac{2}{2^*(s)}}\leq \int_{\Omega} |\nabla u|^2dx-\gamma \int_{\Omega}\frac{u^2}{|x|^2}dx$\quad  for all $u\in H^1_0(\Omega)$,}
\]
where $\gamma <\frac{n^2}{4}$, $s\in [0,2)$ and $\crits:=\frac{2(n-s)}{n-2}$.
We address questions regarding the explicit values of the optimal constant $C:=\mu_{\gamma, s}(\Omega)$, as well as the existence of non-trivial extremals attached to these inequalities. Scale invariance properties often lead to situations where the best constants $\mu_{\gamma, s}(\Omega)$ do not depend on the domain, and hence they are not attainable. We consider two different approaches for ``breaking the homogeneity" of the problem, and restoring compactness. 

 One approach was initiated by Brezis-Nirenberg, when $\gamma=0$ and $s=0$, and by Janelli, when $\gamma>0$ and $s=0$. It is suitable for the case where the singularity $0$ is in the interior of $\Omega$, and consists of considering lower order perturbations of the critical nonlinearity. The other approach was initiated by Ghoussoub-Kang  for $\gamma=0$, $s>0$,  and by C.S. Lin et al. and Ghoussoub-Robert, when $\gamma \neq 0, s\geq 0$. It consists of considering domains, where the singularity $0$ is on the boundary.  

Both of these approaches are rich in structure and in challenging problems. If $0\in \Omega$, then a negative linear perturbation suffices for higher dimensions, while a positive {\it ``Hardy-singular interior mass"} theorem for the operator $L_\gamma$
is required in lower dimensions. If the singularity $0$ belongs to the boundary $\partial \Omega$, then the local geometry around $0$ (convexity and mean curvature) plays a crucial role in high dimensions, while a positive {\it ``Hardy-singular boundary mass"} theorem is needed for the lower dimensions. Each case leads to a distinct notion of critical dimension for the operator $L_\gamma$.
  
\end{abstract}

\maketitle
\newpage 
\tableofcontents
\part{Introduction and overview}
Given a domain $\Omega$ in $\rn$ ($n\geq 3$), we discuss issues of existence of extremals for the following general Sobolev inequality associated with the Hardy-Schr\"odinger operator $L_\gamma=-\Delta -\frac{\gamma}{|x|^2}$, where $\gamma \in \rr$, $s\in [0, 2]$, and $\crits:=\frac{2(n-s)}{n-2}$.
\begin{equation}\label{general}
\hbox{$C\left(\int_{\Omega}\frac{u^{\crits}}{|x|^s}dx\right)^{\frac{2}{\crits}}\leq \int_{\Omega} |\nabla u|^2dx-\gamma \int_{\Omega}\frac{u^2}{|x|^2}dx$\quad  for all $u\in D^{1,2}(\Omega)$,}
\end{equation}
where $D^{1,2}(\Omega)$ is the completion of $C^{\infty}_0(\Omega)$ for the norm  $\|u\|^2=\int_{\Omega}|\nabla u|^2 dx.$  If $\Omega$ is bounded we shall sometimes write $H^1_0(\Omega)$ instead of $D^{1,2}(\Omega)$.\\
Note that when $s=2$ and $\gamma=0$, this is the celebrated Hardy inequality. If $s=0$ and $\gamma=0$, it is the Sobolev inequality, while in their full generalities, i.e., when $s\in [0, 2]$ and $\gamma \in (-\infty , \frac{(n-2)^2}{4})$, they contain -- after a suitable change of functions -- %
the Caffarelli-Kohn-Nirenberg inequalities \cite{ckn}. The latter state that there is a constant $C:=C(a,b,n)>0$ such that,  
\begin{equation} \label{CKN}
\hbox{$\left(\int_{\rn}|x|^{-bq}|u|^q \right)^{\frac{2}{q}}\leq C\int_{\rn}|x|^{-2a}|\nabla u|^2 dx $ \quad for all $u\in C^\infty_0(\rn)$,}
\end{equation}
where
\begin{equation}\label{cond1}
-\infty<a<\frac{n-2}{2}, \ \ 0 \leq b-a\leq 1, \ \ {\rm and}\ \ q=\frac{2n}{n-2+2(b-a)}.
\end{equation}
We shall survey here the state of the art regarding the associated best constants, namely
\begin{equation} 
\mu_{\gamma,s}(\Omega):=\inf\left\{J^\Omega_{\gamma, s}(u); u\in \dundeux\setminus \{0\}\right\},
\end{equation}
where
\begin{equation} 
J^\Omega_{\gamma, s}(u):=\frac{\int_{\Omega} |\nabla u|^2-\gamma \int_{\Omega}\frac{u^2}{|x|^2}dx}{(\int_{\Omega}\frac{u^{2^*(s)}}{|x|^s}dx)^{\frac{2}{2^*(s)}}}.
\end{equation}
We consider the following questions: 
\begin{itemize}
\item How do the best constants $\mu_{\gamma, s}(\Omega)$ depend on $\Omega$, and when one can evaluate their explicit values? 
\item What geometric/topological, local/global conditions on the domain $\Omega$ guarantee the existence (or non-existence) of extremals for $\mu_{\gamma, s}(\Omega)$, that is a function  $u_\Omega$ in $\huno$ such that $J^\Omega_{\gamma, s}(u_{\Omega})=\mu_{\gamma, s}(\Omega)$? 
\item What is the role of the dimension of the ambiant space?
\end{itemize}
Note that such an extremal -- in the case where $\mu_{\gamma, s}(\Omega)>0$ -- would yield a  solution for the corresponding Euler-Lagrange equations,
\begin{eqnarray} \label{one}
\left\{ \begin{array}{llll}
-\Delta u-\gamma \frac{u}{|x|^2}&=&\frac{u^{2^*(s)-1}}{|x|^s} \ \ &\text{on } \Omega\\
\hfill u&>&0 &\text{on } \Omega\\
\hfill u&=&0 &\text{on }\partial \Omega.
\end{array} \right.
\end{eqnarray}
Elliptic problems with singular potential arise in quantum mechanics, astrophysics, as well as in Riemannian geometry, in particular  
in the study of the scalar curvature problem on the sphere $S^n$. Indeed, if the latter is equipped with its standard metric whose scalar curvature is singular at the north and south poles, then by considering its stereographic projection of $\rn$, the problem of finding a conformal metric with prescribed scalar curvature $K(x)$ leads to finding solutions of the form $-\Delta u-\gamma \frac{u}{|x|^2}=K(x)u^{2^*-1}$ on $\rn$. The latter is a simplified version of the nonlinear Wheeler-DeWitt equation, which appears in quantum cosmology
(see \cites{BB,BE,LZ,SmetsTAMS} and the references cited therein).

We shall always assume throughout this paper that $0\in\overline{\Omega}$. The case when the singularity $0\not\in\overline{\Omega}$ is not interesting for $s>0$.
 Indeed, in this case $L^{\crits}(\Omega, |x|^{-s})=L^{\crits}(\Omega)$ and the embedding $\huno\hookrightarrow L^{\crits }(\Omega)$ is compact since $1\leq \crit (s)<\frac{2n}{n-2}$. Therefore, the standard minimization methods work and there are extremals for $\mu_{\gamma,s}(\Omega)$. However, finding the explicit value of $\mu_{\gamma,s}(\Omega)$ is almost impossible in general.

\medskip Assuming now that $0\in\overline{\Omega}$, the first difficulty in these problems is due to the fact that  $\crits $ is critical from the viewpoint of the Sobolev embeddings, in such a way that if $\Omega$ is bounded, then $\huno$ is embedded in the weighted space $L^p(\Omega, |x|^{-s})$ for $1\leq p\leq\crits $, and the embedding is compact if and only if  $p<\crits$.
This lack of compactness defeats the classical minimization strategy to get extremals for $\mu_{\gamma, s}(\Omega)$. In fact, when $s=0$ and $\gamma=0$, this is the setting of the critical case in the classical Sobolev inequalities, which started this whole line of inquiry, due to its connection with the Yamabe problem on compact Riemannian manifolds \cite{aubin}, \cite{LeeParker}. 

\medskip Another complicating feature of the problem is that the  terms $\frac{u}{|x|^2}$ and $\frac{u^{2^*(s)-1}}{|x|^s}$ are critical, in the sense that they have the same homogeneity as the Laplacian. Moreover, the Hardy potential does not belong to the Kato class. 
The best constant in the Sobolev inequality on $\rn$ is  
\begin{equation}
\mu_{0,0}(\rn)=\inf\left\{\frac{\int_{\rn} |\nabla u|^2\, dx}{(\int_{\rn} |u|^{2^*})^{2/2^*}}; u\in D^{1,2}(\rn)\setminus \{0\}\right\}, 
\end{equation}
where $2^*=2^*(0)=\frac{2n}{n-2}$. 
It is attained, and has been computed to be 
\begin{equation}
\mu_{0,0}(\rn) 
=\frac{n(n-2)\omega_n^{2/n}}{4},
\end{equation}
where $\omega_n$ is the volume of the standard $n-$sphere of $\rr^{n+1}$.  Actually, a function $u\in D^{1,2}(\rn)\setminus\{0\}$ is an extremal for $\mu_{0,0}(\rn)$ if and only if there exist $x_0\in\rn$, $\lambda\in\rr\setminus\{0\}$ and $\epsilon>0$ such that
\begin{equation}\label{ext:00}
u_{\lambda, x_0}(x)=\lambda \left(\frac{\epsilon}{\epsilon^2+|x-x_0|^2}\right)^{\frac{n-2}{2}}\hbox{ for all }x\in\rn.
\end{equation}
These results are due to Rodemich \cite{rodemich}, Aubin \cite{aubin} and Talenti \cite{Tal}. We also refer to Lieb \cite{Lieb} and Lions \cites{lions1, lions2} for other nice points of view.\par

However, for general open subsets of $\rn$, one can show by translating, scaling and cutting off $u_{\lambda, x_0}$ that
$\mu_{0,0}(\Omega)=\mu_{0,0}(\rn)$
for all $\Omega$ open subset of $\rn$, which means that if there is an extremal for $\mu_{0,0}(\Omega)$, then it is also an extremal for $\mu_{0,0}(\rn)$ and has to be in the form of \eqref{ext:00}, which is impossible if $\Omega$ is bounded.

The above case has no singularities, which only appear when either $\gamma \neq 0$ or $s>0$. 
But even in this case, we get the same phenomenon as soon as the singularity belongs to the interior of the domain, that is $\mu_{\gamma, s}(\Omega)=\mu_{\gamma, s}(\rn)$, which again means that $\mu_{\gamma, s}(\Omega)$ is not attained unless $\Omega$ is essentially equal to $\rn$. 

It is well known that if $0$ is in the interior of $\Omega$, then the best constant in the Hardy inequality, 
$$\gamma_H(\Omega):=\mu_{0, 2}(\Omega)=\inf\left\{\frac{\int_{\Omega}|\nabla u|^2\, dx}{\int_{\Omega}\frac{u^2}{|x|^2}\, dx}; \, u\in D^{1,2}(\Omega)\setminus\{0\}\right\}, $$
does not depend on the domain $\Omega \subset \rn$, is never achieved, and is always equal to 
\begin{equation}
\mu_{0,2}(\Omega)=\mu_{0,2}(\rn)=\frac{(n-2)^2}{4}.
\end{equation}  
Also, if $0<s<2$,  the constant $\mu_{0,s}(\rn)$ is again explicit, and the extremals are also known (see Ghoussoub-Yuan \cite{GY}, Lieb \cite{Lieb}, Catrina-Wang \cite{CW}). More precisely, 
\begin{equation}
\mu_{0,s}(\rn)=(n-2)(n-s)\left(\frac{\omega_{n-1}}{2-s}\cdot\frac{\Gamma^2(\frac{n-s}{2-s})}{\Gamma(\frac{2n-2s}{2-s})}\right)^{\frac{2-s}{n-s}},
\end{equation}
and a function $u\in D^{1,2}(\rn)\setminus\{0\}$ is an extremal for $\mu_{0, s}(\rn)$ if and only if there exists $\lambda\in\rr\setminus\{0\}$ and $\epsilon>0$ such that $u=\lambda\cdot u_\epsilon$, where  
\begin{equation}
u_\epsilon(x):=\left(\frac{\epsilon^{\frac{2-s}{2}}}{\epsilon^{2-s}+|x|^{2-s}}\right)^{\frac{n-2}{2-s}}.
\end{equation}
 Here, it is important to note the following asymptotics for $u_\epsilon$ when $\epsilon\to 0$:
$$\lim_{\epsilon\to 0}u_\epsilon(0)=+\infty\hbox{ and }\lim_{\epsilon\to 0}u_\epsilon(x)=0\hbox{ for all }x\neq 0.$$
In other words, the function $u_\epsilon$ concentrates at $0$ when $\epsilon\to 0$.\par

When dealing with an open subset $\Omega$ of $\rn$, then clearly $\mu_{0,s}(\Omega)\geq\mu_{0,s}(\rn).$ On the other hand, if $0\in\Omega$,
and $\eta\in C^\infty_c(\Omega)$ is such that $\eta(x)\equiv 1$ in a neighborhood of $0$. Then $\eta u_\epsilon\in C^\infty_c(\Omega)$, and $$J^\Omega_{0, s}(\eta u_\epsilon)=\mu_{0,s}(\rn)+o(1)\hbox{ where $\lim_{\epsilon\to 0}o(1)=0$.}$$
 It then follows that  if  $0\in \Omega$, then 
 $$\mu_{0,s}(\Omega)=\mu_{0,s}(\rn),$$
  and again, there is no extremal for $\mu_{0,s}(\Omega)$ unless $\Omega$ is $\rn$ up to a set of capacity $0$. 
 
 The situation remains unchanged even when $\gamma >0$. One can still compute explicitly $\mu_{\gamma, s}(\rn)$. Indeed, if $n\geq 3$, $0\leq s <2$ and $0<\gamma < \frac{(n-2)^2}{4}$, 
the corresponding best constant is then
\begin{equation}
\mu_{\gamma, s}(\rn)=[(n-2)^2-4\gamma]^{\frac{1}{2^*(s)}+\frac{1}{2}} D_s, 
\end{equation}
where
\begin{equation*}
\quad D_s = \left[ \frac{2\pi^{n/2}}{\Gamma (n/2)}\right]^{\frac{2-s}{n-s} }
\left(\frac{2^*(s)}2\right)^{\frac{2}{2^*(s)}}
\left[\frac{\Gamma (\frac{n-s}{2-s})\Gamma(\frac{n+2-2s}{2-s})}
{\Gamma (\frac{2(n-s)}{2-s})}\right]^{\frac{2-s}{n-s} }.
\end{equation*}
See for example Beckner \cite{beck} or Dolbeault et al. \cite{DELT}. 
The extremals for $\mu_{\gamma, s}(\rn)$ are then given for $\eps>0$, by the functions $u_\eps (x)=\eps^{-\frac{(n-2)}{2}}U (\frac{x}{\eps})$, where 
\begin{equation}\label{J}
U(x):=\frac{1}
{\left(|x|^{\frac{(2-s)\beta_-(\gamma)}{n-2}}+|x|^{\frac{(2-s)\beta_+(\gamma)}{n-2}}\right)^{\frac{n-2}{2-s}}}
\qquad \hbox{ for }x\in\rn\setminus\{0\},
\end{equation}
and
\begin{equation}
\beta_{\pm}(\gamma):=\frac{n-2}{2}\pm\sqrt{\frac{(n-2)^2}{4}-\gamma}.
\end{equation}
Keep in mind that the radial function $x\mapsto |x|^{-\beta}$ is a solution of  $(-\Delta -\frac{\gamma}{|x|^2})|x|^{-\beta}=0$ on $\rn\setminus\{0\}$ if and only if $\beta\in \{\beta_-(\gamma),\beta_+(\gamma)\}$. 
Again, if $0\in \Omega$, we have $$
\mu_{\gamma,s}(\Omega)=\mu_{\gamma,s}(\rn),
$$ and as above, there is no extremal for $\mu_{\gamma,s}(\Omega)$ if, for example, $\Omega$ is bounded. 

Now, in order to remedy the lack of compactness in this Euclidean setting, one can consider the subcritial case, by replacing $2^*(s)$ by a power $p$ with $2<p<2^*(s)$. This direction, however, does not present any new idea or difficulty.  In this paper, we shall describe two --more subtle-- approaches for ``breaking the homogeneity" of the problem, and restoring compactness:
\begin{itemize}
\item One was initiated by Brezis-Nirenberg \cite{bn} when $\gamma=0$ and considered by  Ghoussoub-Yuan \cite{GY}, Janelli \cite{Jan}, Kang-Peng \cites{KP1, KP2, KP3} and many others \cites{CHa, CHP, CP} when $\gamma >0$. It consists of considering lower order perturbations of the critical case. 

\item The other approach was initiated by Ghoussoub-Kang \cite{gk} and developed by Ghoussoub-Robert \cites{gr1, gr2, gr3} when $s>0$ and $\gamma=0$, and by C.S. Lin et al. \cites{HLW, LW1, LW2, LW3} and Ghoussoub-Robert  \cite{gr4} when $\gamma \neq 0$. It consists of considering domains, where the singularity $0$ is on the boundary.  
\end{itemize} 
Both of these approaches are rich in structure and in challenging problems. They both invoke the geometry of the domain (locally and globally), and introduce new critical dimensions to the problem. They also differ in many ways.

\section{Linearly perturbed borderline variational problems with an interior singularity}

The linear perturbation approach consists of considering equations of the form
 \begin{eqnarray} \label{four}
\left\{ \begin{array}{llll}
-\Delta u-\gamma \frac{u}{|x|^2}&=&\frac{u^{2^*(s)-1}}{|x|^s} +\lambda |u|^{q-1}u \ \ &\text{on } \Omega\\
\hfill u&=&0 &\text{on }\partial \Omega, 
\end{array} \right.
\end{eqnarray}
where $1\leq q <2^*(s)$ and $\lambda>0$ is small enough. For simplicity, we only discuss
that case for $q=1$. One then considers the quantity
\begin{equation} \label{general.constant}
\mu_{\gamma, s, \lambda}(\Omega):=\inf\left\{\frac{\int_{\Omega} |\nabla u|^2\, dx-\gamma \int_{\Omega}\frac{u^2}{|x|^2}dx-\lambda \int_{\Omega} |u|^2\, dx }{(\int_{\Omega}\frac{u^{2^*(s)}}{|x|^s}dx)^{\frac{2}{2^*(s)}}};\, u\in D^{1,2}(\Omega)\right\},
\end{equation}
and use the fact that compactness is restored as long as 
\begin{equation}
\mu_{\gamma, s, \lambda}(\Omega) <\mu_{\gamma, s}(\rn).
\end{equation} 
This extremely important observation is due to Trudinger \cite{Tru}, when $s=\gamma=\lambda=0$, in the case of Riemannian manifolds, where the geometry plays the crucial role. He was actually trying to salvage Yamabe's  proof of  his own conjecture. This kind of condition
is now standard while dealing with borderline variational problems. See also Aubin \cite{aubin}, Br\'ezis-Nirenberg \cite{bn}. The condition limits the energy level of minimizing sequences, prevents the creation of ``bubbles" and hence insures compactness. We give below an idea of the proof based on Struwe's decomposition of non-convergent minimizing sequences. 

The idea of restoring compactness on Euclidean domains by considering linear perturbations was pioneered by Brezis-Nirenberg \cite{bn}. They studied  the case where $\gamma=0$, $s=0$ and $0<\lambda <\lambda_1(\Omega)$, the latter being the first eigenvalue of the Laplacian on $H^1_0(\Omega)$, that is the equation 
 \begin{eqnarray} \label{BZ}
\left\{ \begin{array}{llll}
-\Delta u- \lambda u&=&|u|^{2^*-1}u \ \ &\text{on } \Omega\\
\hfill u&>&0 &\text{on } \Omega\\
\hfill u&=&0 &\text{on }\partial \Omega. 
\end{array} \right.
\end{eqnarray}
They showed existence of extremals for $n\geq 4$. 
The case $n=3$ is special and involves a ``positive mass" condition introduced by Druet \cite{d2}, and inspired by the work of Shoen \cite{schoen1} on the Yamabe problem.  
The bottom line is that --at least for $\gamma=0$-- the geometry of $\Omega$ need not be taken into account in dimension $n\geq 4$, while in dimension $n=3$, the existence depends heavily on $\Omega$, since the mass condition does. We shall elaborate further on this theme.

The paper of Brezis-Nirenberg \cite{bn}  generated lots of activities. Combined with the contribution of Druet \cite{d2}, it contains most of the ingredients relevant to the case when $0\in \Omega$, including the case when the Laplacian is replaced by the Hardy-Schr\"odinger  operator $L_\gamma$ that we discuss below.

Following Janelli \cite{Jan}, who dealt with the case $s=0$,  many others   \cites{RW, KP1, KP2, KP3, CHa, CHP, CP} showed what amounts to the following. 

\begin{theorem} Assume $\Omega$ is a smooth bounded domain in $\rn$ such that $0\in \Omega$. If $n\geq 4$, $s\geq 0$, 
$0\leq \gamma\leq\frac{(n-2)^2}{4}-1,$ and $0<\lambda <\lambda_1(L_\gamma)$,  
then  $\mu_{\gamma, s, \lambda}(\Omega)$ is attained. 
\end{theorem}
The proof again consists of testing the functional on minimizing sequences of the form $\eta U_\epsilon$, where $U_\epsilon$ is an extremal for $\mu_{\gamma, s}(\rn)$ 
and $\eta\in C^\infty_c(\Omega)$ is a cut-off function equal to $1$ in a neigbourhood of $0$, and showing that $\mu_{\gamma,s, \lambda}(\Omega)<\mu_{\gamma, s}(\rn)$. 

Janelli and others had partial results for the remaining interval that is when $\frac{(n-2)^2}{4}-1\leq \gamma <\frac{(n-2)^2}{4}$, a gap that we proceeded to fill recently in \cite{gr5}. In order to complete the picture, it was first important to know for which parameters $\gamma$ and $s$, the best constant $\mu_{\gamma, s}(\rn)$ is attained. 

\begin{proposition} Assume $\gamma <\frac{(n-2)^2}{4}$. Then, the best constant $\mu_{\gamma, s}(\rn)$ is attained if either $s>0$ or if $\{s=0$ and $\gamma \geq 0\}$. On the other hand, if $s=0$ and $\gamma <0$, then $\mu_{\gamma, s}(\rn)$ is not attained. \end{proposition}

A proof for general cones is given in section 5. Note that (\ref{J}) gives explicit extremals for $\mu_{\gamma, s}(\rn)$ under the conditions $n\geq 3$, $0\leq s <2$ and $0\leq \gamma < \frac{(n-2)^2}{4}$.  

The next step was to define a notion of {\it Hardy interior mass} associated to the operator $-\Delta -\frac{\gamma}{|x|^2} -\lambda$ on a bounded domain $\Omega$ in $\rn$ containing  
$0$.

\begin{proposition} \label{interior.mass} {\rm (Ghoussoub-Robert \cite{gr5})} Assume $0\in \Omega$, where $\Omega$ is a  smooth bounded domain $\Omega$ in $\rn$ ($n\geq 3$). Suppose $a$ is a $C^2$-potential on $\Omega$ so that the operator $-\Delta-\frac{\gamma}{|x|^2}+a(x)$ is coercive. 
\begin{enumerate} 
\item There exists then  $H\in C^\infty(\overline{\Omega}\setminus\{0\})$ such that 
$$(E)\qquad\qquad \left\{\begin{array}{ll}
\Delta H-\frac{\gamma}{|x|^2}H+a(x)H=0 &\hbox{ in } \Omega\setminus\{0\}\\
H>0&\hbox{ in } \Omega\setminus\{0\}\\
H=0&\hbox{ on }\partial\Omega.
\end{array}\right.$$
These solutions are unique up to a positive multiplicative constant, and there exists $c>0$ such that 
$H(x)\simeq_{x\to 0}\frac{c}{|x|^{\beta_+(\gamma)}}.$
\item If either $a$ is sufficiently small around $0$ or if $\frac{(n-2)^2}{4}-1<\gamma<\frac{(n-2)^2}{4}$, then for any solution $H\in C^\infty(\overline{\Omega}\setminus\{0\})$ of $(E)$, there exist $c_1>0$ and $c_2\in\rr$ such that 
$$H(x)=\frac{c_1}{|x|^{\bp}}+\frac{c_2}{|x|^{\bm}}+o\left(\frac{1}{|x|^{\bm}}\right) \qquad \hbox{ as } x\to 0.$$
The uniqueness implies that the ratio $c_2/c_1$ is independent of the choice of $H$, hence the  `` Hardy-singular internal mass" of $\Omega$ associated to the operator $L_\gamma - a$ can be defined unambigously as 
$$m_{ \gamma, a}(\Omega):=\frac{c_2}{c_1}\in\rr.$$
\end{enumerate}
\end{proposition} 
One can then complete the picture as follows. 

\begin{table}[ht]
\begin{center}
\caption{$0\in \Omega$ (Linearly perturbed problems), $0\leq \lambda <\lambda_1(L_\gamma)$ and either $s>0$ or \{$s=0$ and $\gamma \geq 0$\}}
\begin{tabular}{|c|c|c|c|clc} \hline
{\bf  Hardy term }&  {\bf Dim.}& {\bf Sing.}& {\bf Analytic. cond.}&{\bf  Ext.}\\ \hline\hline
$-\infty < \gamma \leq \frac{(n-2)^2}{4}-1$& $n\geq 3$&$s > 0$
  & $\lambda>0$ &Yes\\  \hline
$\frac{(n-2)^2}{4}-1 <  \gamma <\frac{(n-2)^2}{4}$& $n\geq 3$&$s >0$
  & $m_{\gamma, -\lambda}(\Omega) >0$ & Yes\\ 
  \hline
$0\leq  \gamma \leq \frac{(n-2)^2}{4}-1$& $n\geq 4$&$s = 0$
  & $\lambda>0$ &Yes&\\  \hline
$\frac{(n-2)^2}{4}-1<  \gamma <\frac{(n-2)^2}{4}$& $n\geq 4$&$s=0$
  & $m_{\gamma, -\lambda}(\Omega) >0$ & Yes\\ \hline
\end{tabular}
\end{center}
\end{table}
As to the case when $s=0$ and $\gamma <0$, we need a more standard 
notion of mass associated to the operator $L_\gamma$ at an internal point $x_0\in \Omega$, which is reminiscent of Shoen-Yau's approach to complete the solution of the Yamabe conjecture in low dimensions. For that, one considers for a given $\gamma<0$, the corresponding {\it Robin function} or the regular part of the Green function with pole at $x_0\in \Omega\setminus \{0\}$. One shows that for $n=3$, any solution $G$ of 
$$\left\{\begin{array}{ll}
-\Delta G-\frac{\gamma}{|x|^2} G-\lambda G=0 &\hbox{ in } \Omega \setminus\{x_0\}\\
\hfill G>0 &\hbox{ in }\Omega \setminus\{x_0\}\\
\hfill G=0 &\hbox{ on }\partial\Omega,
\end{array}\right.$$
 is unique up to multiplication by a constant, and that there exists $R_{\gamma, \lambda}(\Omega, x_0)\in \rr$ and $c_{\gamma, \lambda}(x_0)>0$ such that
\begin{equation}
G(x)=c_{\gamma, \lambda} (x_0)\left(\frac{1}{|x-x_0|^{n-2}}+R_{\gamma, \lambda}(\Omega, x_0)\right)+o(1)\quad \hbox{ as }x\to x_0.
\end{equation}
The quantity $R_{\gamma, \lambda}(\Omega, x_0)$ is then well defined and will be called the {\it internal mass} of $\Omega$ at $x_0$. We then define
\[
\hbox{$R_{\gamma, \lambda}(\Omega)=\sup\limits_{x\in \Omega}R_{\gamma, \lambda}(\Omega, x)$ \quad and \quad $r(\Omega)=\sup\limits_{x\in \Omega}|x|^2$.}
\]
The following table summarizes the remaining situations. 

\begin{table}[ht]
\begin{center}
\caption{$0\in \Omega$ (Linearly perturbed problems): $0\leq \lambda <\lambda_1(L_\gamma)$ and  $s=0$, $\gamma <0$}

\begin{tabular}{|c|c|c|c|cl} \hline
{\bf  Hardy term }&  {\bf Dim.}& {\bf Geom. cond.} &{\bf  Extremal}\\ \hline\hline
$-\infty < \gamma<0$& $n\geq 4$& $\frac{|\gamma|}{r(\Omega)}<\lambda$ & Yes\\  \hline
$-\infty < \gamma<0$& $n\geq 3$&
    $\lambda \leq \frac{|\gamma|}{r(\Omega)}$  & No\\  \hline
$-\infty < \gamma\leq 0$&$n = 3$& $R_{\gamma, \lambda}(\Omega) >0$&Yes\\ 
 \hline \hline
\end{tabular}
\end{center}
\end{table}
The following theorem summarizes the various situations.

\begin{theorem} \label{main.interior} Let $\Omega$ be a smooth bounded domain in $\rn$ ($n\geq 3$) such that $0\in  \Omega$ and let $0\leq s < 2$, $\gamma <\frac{(n-2)^2}{4}$, and $0 < \lambda <\lambda_1(L_\gamma,\Omega)$. 
\begin{enumerate}

\item If either  $s>0$, or \{$s=0$, $\gamma\geq 0$\}, then there are extremals for $\mu_{\gamma,s}(\Omega)$ under one of the following two conditions:
\begin{enumerate} 
\item  $-\infty <\gamma\leq\frac{(n-2)^2}{4}-1$   
\item $\frac{(n-2)^2}{4}-1  <\gamma <\frac{(n-2)^2}{4}-1$ and  $m_{\gamma,-\lambda} (\Omega)$  is positive. 
\end{enumerate} 
\item If $s=0$, and $\gamma <0$, then there are extremals for $\mu_{\gamma,s}(\Omega)$ under one of the following two conditions:
\begin{enumerate} 
\item  $n\geq 4$ and  $\frac{|\gamma|}{r(\Omega)}<\lambda <\lambda_1(L_\gamma)$.
\item $n = 3$ and $R_{\gamma, -\lambda}(\Omega)>0$.
\end{enumerate} 
\item If $s=0$, $\gamma < 0$, $n\geq 3$ and $0\leq \lambda \leq \frac{|\gamma|}{r(\Omega)}$, 
then there is no extremal for $\mu_{\gamma,s}(\Omega)$.
\end{enumerate}
\end{theorem}

\medskip \noindent One also notes that the mass function $m_{\gamma, a}(\Omega)$ (when defined) satisfies the following properties:
\begin{itemize}
\item[$\bullet$] $m_{\gamma, 0}(\Omega)<0$,
\item[$\bullet$] If $a\leq a'$ and $a\not\equiv a'$, then $m_{\gamma, a}(\Omega)>m_{\gamma, a'}(\Omega)$,
\item[$\bullet$] If $\Omega\subsetneq \Omega'$, then $m_{\gamma, a}(\Omega)<m_{\gamma, a'}(\Omega')$.
\item[$\bullet$] The function $a\mapsto m_{\gamma, a}(\Omega)$ is continuous for the  $C^0(\overline{\Omega})$ norm.
\end{itemize}
It follows that $m_{\gamma, 0}(\Omega)<0$ and $\lambda\mapsto m_{\gamma, -\lambda}(\Omega)$ is strictly increasing and continuous on the interval $[0, \lambda_1(L_\gamma))$. If there exists $\lambda_0\in (0, \lambda_1(L_\gamma)$ such that $m_{\gamma, -\lambda_0}(\Omega)=0$, then %
$m_{\gamma, -\lambda}(\Omega)>0$ for $\lambda_0<\lambda <\lambda_1(L_\gamma)$, which yields that $\mu_{\gamma, s, \lambda}(\Omega)$ is achieved whenever $\lambda$ is in the latter interval.

\noindent Two open problems are worth mentioning:

{\bf Problem 1: } Find necessary and sufficient geometric conditions on $\Omega$, which guarantee 
that if $\frac{(n-2)^2}{4}-1<\gamma<\frac{(n-2)^2}{4}$, then there exists $\lambda_0\in (0, \lambda_1(L_\gamma))$ such that $m_{\gamma, -\lambda_0}(\Omega)=0$. 

{\bf Problem 2: } Assuming such a $\lambda_0(L_\gamma)$ exists, can one show that there are no extremals if   $0<\lambda <\lambda_0(L_\gamma)$.    
Note that this was verified for general domains by Druet \cite{d2} in the case $\gamma=s=0$ and $n=3$.  If $\Omega$ is a unit ball $B$, one can then show -- just as  Janelli \cite{Jan} did in the case when $\gamma >0$, $s=0$ -- that this is indeed the case by showing that if $\frac{(n-2)^2}{4}-1<\gamma<\frac{(n-2)^2}{4}$, then 
\[
\hbox{$\mu_{\gamma, s, \lambda}(B)$ is achieved \,  if and only if \,  $\lambda^*(L_\gamma)\leq \lambda <\lambda_1(L_\gamma)$,}
\]
 where 
\begin{equation}
\lambda^*(L_\gamma)=\inf \left\{\frac{\int_B\frac{|\nabla u|^2}{|x|^{2\beta_+(\gamma)}}dx}{\int_B\frac{u^2}{|x|^{2\beta_+(\gamma)}}dx}; u\in H^1_0(B) \right\}.
\end{equation}
In other words, $\lambda_0(L_\gamma)=\lambda^*(L_\gamma)$. 

The above analysis lead to the following definition of a {\it critical dimension for the operator $L_\gamma$}. It is the largest scalar $n_\gamma$ such that for $n < n_\gamma$, there exists a bounded smooth domain $\Omega \subset \rn$ and a $\lambda \in (0, \lambda_1(L_\gamma, \Omega))$ such that $\mu_{\gamma, s, \lambda}(\Omega)$ is not attained.

One can then deduce that the critical dimension for $L_\gamma$ is 
\begin{equation}
n_\gamma =\left\{\begin{array}{ll} 2\sqrt{\gamma +1}+2\quad & \, {\rm if}\, \gamma \geq -1\\
2 \quad  & \, {\rm if}\,  \gamma <-1.
\end{array}\right.
\end{equation}
Note that $n < n_\gamma$ is exactly when $\bp-\bm < 2$, which is the threshold where
the radial function $x\to |x|^{-\beta_+(\gamma)}$ is locally $L^2$-summable.
 
 \section{Borderline variational problems with a boundary singularity}

The situation changes dramatically and becomes much more interesting if the singularity $0$ belongs to the boundary of the domain $\Omega$. For one, the test functions $\eta U_\epsilon$ don't belong to $\huno$ anymore, and one cannot mimic the arguments given above. Actually, the differences already start with the most basic properties of the Hardy-Schr\"odinger operator $L_\gamma=-\Delta -\frac{\gamma}{|x|^2}$. 

To begin with, recall that if $0\in\Omega$, then   $L_\gamma$ is positive if and only if  $\gamma<\frac{(n-2)^2}{4}$, while  if 
$0\in\partial\Omega$ the operator  $L_\gamma$ could be positive for larger value of $\gamma$,  potentially reaching the maximal constant $\frac{n^2}{4}$ on convex domains. Moreover, if $0\in\Omega$, we have already noted that the best constant in the Hardy inequality $\mu_{0,2} (\Omega)$ is then always equal to $\frac{(n-2)^2}{4}$ and is never achieved, while if $0\in \partial \Omega$, the best constant $\mu_{0,2} (\Omega)$ can be anywhere in the interval $(\frac{(n-2)^2}{4}, \frac{n^2}{4}]$, and it is achieved if $\mu_{0,2} (\Omega)<\frac{n^2}{4}$ (See Ghoussoub-Robert \cite{gr4}).

The situation changes further when $0 \leq s <2$. Indeed, we had seen that whenever $0\in \Omega$, $\mu_{\gamma,s}(\Omega)=\mu_{\gamma,s}(\rn)$, and is never achieved unless $\Omega$ is essentially equal to $\rn$. The first indication that a new phenomenon may occur, when $0\in \partial \Omega$, was given by the following surprising result of  Egnell \cite{egnell} even when $\gamma=0$. 
He showed that 
 if $D$ is a nonempty connected domain of $\mathbb{S}^{n-1}$, the unit sphere in $\rn$, and  $C:=\{r\theta; \, r>0,\, \theta\in D\}$ is the cone based at $0$ induced by $D$, then there are extremals for $\mu_{0,s}(C)$ whenever $s>0$.
 
An important point to note here is that the cone $C$ is not smooth at $0$, unless it is $\rnp$ or $\rn$. 
Actually, if a general domain $\Omega$ with $0$ on its boundary is smooth, then it  looks more like the half-space $\rnp$  
around $0$, and not like $\rn$ as in the case $0\in\Omega$. One therefore has to  compare $\mu_{\gamma,s}(\Omega)$ with $\mu_{\gamma,s}(\rnp)$, which is strictly larger than $\mu_{\gamma,s}(\rn)$. One can also easily show that if $\Omega$ is smooth bounded and $0\in \partial\Omega$, then 
$$
\mu_{\gamma,s}(\rn)<\mu_{\gamma, s}(\Omega) \leq \mu_{\gamma,s}(\rnp),
$$
 and if $\Omega$ is convex (or if $\Omega\subset \rnp$), then  $\mu_{\gamma,s}(\Omega)= \mu_{\gamma,s}(\rnp)$ and again $\mu_{\gamma,s}(\Omega)$ has no extremals. 

Another discrepancy with the case where $0$ is in the interior,  is the fact that the extremals for $\mu_{\gamma, s}(\rn)$, which are the building blocks for the extremals on bounded domains, can often be written explicitly as seen above, while the ones for $\mu_{\gamma, s}(\rnp)$ are not. So one then tries to understand as much as possible the profile of such extremals, which happen to solve the equation 
\begin{eqnarray} \label{two}
\left\{ \begin{array}{llll}
-\Delta u-\gamma \frac{u}{|x|^2}&=&\frac{u^{2^*(s)-1}}{|x|^s} \ \ &\text{on } \rnp\\
\hfill u&>&0 &\text{on } \rnp\\
\hfill u&=&0 &\text{on }\partial \rnp.
\end{array} \right.
\end{eqnarray}
This was done in a recent   analysis by Ghoussoub-Robert \cite{gr4}, where the needed information on the profile is given. The non-explicit solution has the following properties:
\begin{itemize}
\item  {\it Symmetry: } $u\circ\sigma=u$ for all isometry of $\rn$ such that $\sigma(\rnp)=\rnp$. In particular, there 
exists $v\in C^2(\rr_+\times \rr)$ such that for all $x_1>0$ and all $x'\in\rr^{n-1}$, 
$$u(x_1,x')=v(x_1,|x'|).$$  
\item {\it Asymptotic profile:} If $u\not\equiv 0$, then there exist $K_1,K_2>0$ such that
\begin{equation}\nonumber
u(x)\sim_{x\to 0}K_1\frac{x_1}{|x|^{\am}}\quad \hbox{ and } \quad u(x)\sim_{|x|\to +\infty}K_2\frac{x_1}{|x|^{\ap}},
\end{equation}
\end{itemize}
where 
\begin{equation}
\alpha_{\pm}(\gamma):=\frac{n}{2}\pm\sqrt{\frac{n^2}{4}-\gamma}.
\end{equation}
 Keep in mind that $x\mapsto x_1|x|^{-\alpha}$ is a solution of  $(-\Delta -\frac{\gamma}{|x|^2})x_1|x|^{-\alpha}=0$ on $\rn\setminus\{0\}$ if and only if $\alpha\in \{\am,\ap\}$. 
Note that $\am<\frac{n}{2}<\ap$, which points to  the difference between the ``small" solution, namely $x\mapsto x_1|x|^{-\am}$, which is ``variational", i.e. is locally in $D^{1,2}(\rnp)$, and the ``large one"  $x\mapsto x_1|x|^{-\ap}$, which is not. 

It also turned out that, unlike the case where $0\in \Omega$, there are examples of domains with $0\in \partial \Omega$ such that $\mu_{\gamma,s}(\Omega)< \mu_{\gamma,s}(\rnp)$, which means that $\mu_{\gamma,s}(\Omega)$ has a good chance to be attained. This was first observed by Ghoussoub-Kang \cite{gk} in the most basic case, where $0<s<2$ and $\gamma =0$. 
Again, this condition limits the energy level of minimizing sequences, and therefore prevents the creation of bubbles (in this case around $0$) and hence ensures compactness. There are many ways to see this, and we use the opportunity to introduce Struwe's approach via his famed decomposition \cite{st}.
 
Since $\partial \Omega$ is smooth at $0$, there exists $U,V$ open subsets of $\rn$ such that $0\in U$, $0\in V$ and a $C^\infty-$diffeomorphism  $\varphi:U\to V$ such that $\varphi(0)=0$, 
$$\varphi(U\cap\{x_1>0\})=\varphi(U)\cap\Omega,\quad {\rm and} \quad \varphi(U\cap\{x_1=0\})=\varphi(U)\cap\partial\Omega.$$
Up to an affine transformation, we can assume that the differential of $\varphi$ at $0$ is the identity map. 
Letting $\eta\in C_c^\infty(U)$ be such that $\eta(x)\equiv 1$ in a neighborhood of $0$, and given $\epsilon\in (0,\crits-2)$, we consider the subcritical minimization problems:
$$\mu^\epsilon_{0,s}(\Omega):=\inf_{u\in\huno\setminus\{0\}} \frac{\int_\Omega|\nabla u|^2\, dx}{\left(\int_\Omega\frac{|u|^{\crits-\epsilon}}{|x|^s}\, dx\right)^{\frac{2}{2^*(s)-\eps}}}.$$
Since the exponent $p_\epsilon:=\crits-\eps$ is subcritical,   the embedding $\huno\hookrightarrow L^{p_\epsilon}(\Omega, |x|^{-s})$ is compact, and we therefore have a minimizer $\ue\in\huno\setminus\{0\}$ where  $\mu^\epsilon_{0,s}(\Omega)$ is attained. Regularity theory then yields that $\ue\in C^\infty(\overline{\Omega}\setminus\{0\})\cap C^1(\overline{\Omega})$ and we can assume that $\ue$ solves the equation
\begin{equation} 
\left\{\begin{array}{ll}\label{syst:ue}
\Delta\ue=\frac{\ue^{p_\eps -1}}{|x|^s}&\hbox{ in }\Omega\\
\quad \ue>0 &\hbox{ in }\Omega\\
\quad \ue=0 &\hbox{ on }\partial\Omega.
\end{array}\right.
\end{equation}
The ``free energy" of the solutions then satisfy
$\int_\Omega\frac{\ue^{p_\eps}}{|x|^s}\, dx=\mu^\epsilon_{0,s}(\Omega)^{\frac{p_\epsilon}{p_\epsilon-2}}.$ 
The standard strategy is then to analyze what happens when we let $\epsilon\to 0$. This is not straightforward since the embedding $\huno\rightharpoonup L^{\crits}(\Omega;|x|^{-s})$ is not compact. In the case $s=0$, Struwe \cite{st} gave a useful decomposition describing precisely this lack of compactness for minimzing sequences such as $(\ue)_\eps$, which was extended to this situation by Ghoussoub-Kang \cite{gk}. It says that there exists $\Lambda>0$ with $\Vert\ue\Vert_{\huno}\leq\Lambda$ for all $\epsilon>0$, $u_0\in\huno$,  as well as $N$ positive bubbles $(B_{i,\epsilon})_{\epsilon}$, $i\in\{1,...,N\}$ such that
\begin{equation}\label{dec:str}
\ue=u_0+\sum_{i=1}^NB_{i,\epsilon}+R_\epsilon,
\end{equation}
where $\lim_{\epsilon\to 0}R_\epsilon=0$ strongly in $\huno$. 

A {\it bubble} here is any family $(B_\epsilon)_{\epsilon}\in\huno$ of the form 
\begin{equation}\label{bubble}
\hbox{$B_{\epsilon}(x)=\eta(x)\mu_{\epsilon}^{-\frac{n-2}{2}}u\left(k_{\epsilon}^{-1}\varphi^{-1}(x)\right)$ \quad if $x\in U\cap\rnp$ and $0$ otherwise,}
\end{equation}
where  $u\in D^{1,2}(\rnp)\setminus\{0\}$ is a solution of 
$\Delta u=\frac{|u|^{\crits-2}u}{|x|^s}$ in $\rnp$, and $(\mu_\epsilon)_{\epsilon}\in\rr_+$ is such that $\lim_{\epsilon\to 0}\mu_{\epsilon}=0$, with $k_\epsilon=\mu_\epsilon^{1-\frac{\epsilon}{\crits-2}}$ satisfying $\lim_{\epsilon\to 0}k_\epsilon^\epsilon=c\in (0,1].$\\
Note that for any bubble, we have $\int_\Omega\frac{|B_{\epsilon}|^{p_\eps}}{|x|^s}\, dx +o(1)\geq \mu_{0,s}(\rnp)^{\frac{\crits}{\crits-2}}+o(1)$,  which means that if there is any bubble in the decomposition, then necessarily $$\int_\Omega\frac{\ue^{p_\eps}}{|x|^s}\, dx \geq \int_\Omega\frac{B_{i,\epsilon}^{p_\eps}}{|x|^s}\, dx +o(1)\geq \mu_{0,s}(\rnp)^{\frac{\crits}{\crits-2}}+o(1), \,\, \hbox{where $\lim_{\epsilon\to 0}o(1)=0$.}$$
Since $\lim_{\epsilon\to 0}\mu^\epsilon_{0,s}(\Omega)=\mu_{0,s}(\Omega)$, one then get that $\mu_{0,s}(\Omega)\geq \mu_{0,s}(\rnp),$ which contradicts the initial energy hypothesis. It follows that there is no bubble and therefore $\lim_{\epsilon\to 0}\ue=u_0$ in $\huno$, yielding that $u_0$ is an extremal for $\mu_{0,s}(\Omega)$. 

The question now is what geometric condition on $\Omega$ insures that we have the analytic condition $\mu_{\gamma, s}(\Omega) <\mu_{\gamma, s} (\rnp)$. In view of the above,  for any hope to find extremals, one has to avoid 
situations where $\Omega$ is convex or if it lies on one side of a hyperplane that is tangent at $0$. This was first confirmed by Ghoussoub-Kang \cite{gk}, who proved that this is indeed the case --and that extremals exist-- provided $n\geq 4$ and the principal curvatures of $\partial \Omega$ at $0$ are all negative. 

Concerning terminology, recall that the principal curvatures are the eigenvalues of the second fundamental form of the hypersurface $\partial\Omega$ oriented by the outward normal vector. The second fundamental form being
$$II_0(\vec{X},\vec{Y})=(dn_0(\vec{X}),\vec{Y})\hbox{ for }\vec{X},\vec{Y}\in T_0\partial\Omega,$$
where $dn_0$ is the differential of the outward normal vector at $0$ and $(\cdot,\cdot)$ is the Euclidean scalar product. 

The result of Kang-Ghoussoub was eventually improved later by Ghoussoub-Robert \cites{gr1,  gr2}, 
who also proved it for $n=3$ and by only requiring that {\it the mean curvature}, i.e., the trace of the second fundamental form, at $0$, to be negative (see also Chern-Lin  \cite{CL5}).  Qualitatively, this says that there are extremals for $\mu_{0,s}(\Omega)$, whenever the domain at $0$ has more concave directions than convex ones, in the sense that the negative principal directions dominate quantitatively the positive principal directions. This allows for new examples, which are neither convex nor concave at $0$, and for which the extremals exist. Note that this result does not give any information about the value of the best constant.

We now  illustrate how the mean curvature enters in the picture  in the simplest case, namely when $s>0$ and $\gamma=0$. It consists of performing a more refined blow-up analysis on the minimizing sequences considered above. The proof --due to Ghoussoub-Robert \cite{gr1}--  uses the machinery developed in Druet-Hebey-Robert \cite{dhr} for equations of Yamabe-type on manifolds. It also allows to tackle problems with arbitrary high energy and not just minima \cite{gr2}.

We consider again the solutions $(u_\epsilon)$ of the subcritical problems corresponding to $p_\epsilon=2^*(s)-\epsilon$ with $\epsilon\in (0,\crits-2)$, in such a way that 
\begin{equation}\label{nrj:ue}
\lim_{\eps\to 0}\int_\Omega\frac{\ue^{\crits-\eps}}{|x|^s}\, dx=\mu_{0,s}(\Omega)^{\frac{\crits}{\crits-2}}.
\end{equation}
One then proves (see Ghoussoub-Robert \cite{gr1}) that 
 either $u_\epsilon$ converges to an extremal of 
$\mu_{0,s}(\Omega)$, or  blow-up occurs in the following sense: $u_\epsilon$ converges weakly to zero  and there exists a solution $v$ for 
\begin{equation}\label{eq:lim.0}
\hbox{$-\Delta v=\frac{v^{\crits-1}}{|x|^s}$ in $\rnp$,  \, $v>0$ in $\rnp$ and  $v=0$ on $\partial\rnp$,}
\end{equation}
 such that
  \[
  \int_{\rnp}|\nabla v|^2\, 
dx=\mu_{0, s}(\Omega)^{\frac{\crits}{\crits-2}}=\mu_{s}(\rnp)^{\frac{\crits}{\crits-2}},
\]
while  -modulo passing to a subsequence- we have
\begin{equation}\label{30}
\lim_{\eps\to 0}\eps \, (\max_\Omega 
u_\eps)^{\frac{2}{n-2}}=\frac{(n-s)\int_{\partial\rnp}|x|^2|\nabla v|^2\, 
dx}{n(n-2)^2\mu_{s}(\rnp)^{\frac{n-s}{2-s}}}\cdot H_\Omega(0),
\end{equation}
where  $H_\Omega(0)$ is the mean curvature of the oriented boundary 
$\partial\Omega$ at $0$.
Note that if $H_\Omega(0) <0$, such a blow-up cannot occur and we therefore end up with an extremal. 

To sketch a proof of such a dichotomy, we start as before with the Struwe decomposition to write that  either there exists $u_0\in\huno\setminus\{0\}$ such that $\lim_{\epsilon\to 0}\ue=u_0$ in $\huno$, hence it is an extremal for $\mu_{0, s}(\Omega)$, or there exists a bubble $(B_\epsilon)_{\epsilon>0}$ such that 
\begin{equation}\label{dec:bulle}
u_\epsilon=B_\epsilon+o(1) \quad \hbox{where $\lim_{\epsilon\to 0}o(1)=0$ in $\huno$. }
\end{equation}
Moreover, the function $v\in D^{1,2}(\rnp)$ defining the bubble  is positive, in particular, $v\in D^{1,2}(\rnp)\cap C^\infty(\overline{\rnp}\setminus\{0\})\cap C^1(\overline{\rnp})$ and is a solution for (\ref{eq:lim.0}).
The idea is to prove that the family $(\ue)_{\epsilon>0}$ behaves more or less like the bubble $(B_\epsilon)_{\epsilon>0}$. In fact \eqref{dec:bulle} already indicates that these two families are equal up to the addition of a term vanishing in $\huno$. But we actually need something more precise, like a pointwise description, as opposed to a weak description in Sobolev space. This requires a good knowledge of the bubbles: a difficult question since bubbles are not explicit here as in the case of $\rn$. The proof has two main steps:

First, one shows that  there exists $C_1>0$ such that for all $\epsilon>0$,
\begin{equation}\label{est:co}
\frac{1}{C_1}\frac{\mu_\epsilon^{n/2}d(x,\partial\Omega)}{(\mu_\epsilon^2+|x|^2)^{n/2}}\leq \ue(x)\leq C_1\frac{\mu_\epsilon^{n/2}d(x,\partial\Omega)}{(\mu_\epsilon^2+|x|^2)^{n/2}} \quad \hbox{for all $x\in\Omega$,}
\end{equation}
where $(\mu_\epsilon)$ are involved in the definition (\ref{bubble}) of the bubble $(B_\epsilon)$.  

The next step is to use  
the following Pohozaev identity, 
 $$\int_{\Omega} x^i\partial_i \ue\Delta\ue\, dx+\frac{n-2}{2}\int_{\Omega}\ue\Delta\ue\, dx=-\frac{1}{2}\int_{\partial \Omega}(x,\nu)|\nabla \ue|^2\, d\sigma$$
 to get that
$$\left(\frac{n-2}{2}-\frac{n-s}{\crits-\epsilon}\right)\int_{\Omega} \frac{\ue^{\crits-\epsilon}}{|x|^s}\, dx=-\frac{1}{2}\int_{\partial \Omega}(x,\nu)|\nabla \ue|^2\, d\sigma.$$
The left-hand-side is easy to estimate with \eqref{nrj:ue}. For the right-hand-side, one uses the optimal estimate \eqref{est:co} to obtain
$$\lim_{\epsilon\to 0}\frac{\epsilon}{\mu_\epsilon}=\frac{(n-s)\int_{\partial\rnp} II_0(x,x)|\nabla v|^2\, dx}{(n-2)^2\int_{\rnp}|\nabla v|^2\, dx},$$
where $II_0$ is the second fondamental form at $0$ defined on the tangent space of $\partial\Omega$ at $0$ that we identify with $\partial\rnp$. Finally, in view of the symmetry result mentioned above for the solution $u$, that is 
$u(x_1,\bar{x})=\tilde{u}(x_1,|x|)$ where $\tilde{u}: \rr_+\times\rr\to\rr$, which means that the limit above rewrites as (\ref{30}).

Optimal pointwise estimates like (32) have their origin in the work of Atkinson-Peletier \cite{ap} and Br\'ezis-Peletier \cite{bp}. Pioneering work also include Han \cite{han} and Hebey-Vaugon \cite{hv} in the case of a Riemannian manifold. For $s=\gamma=0$, the general pointwise estimates are performed in the monograph \cite{dhr} of Druet-Hebey-Robert. We also refer to Ghoussoub-Robert \cite{gr2} for the optimal control with arbitrary high energy when $s>0$ and $\gamma=0$. Other methods developed to get pointwise estimates are due to Schoen-Zhang \cite{sz} and Kuhri-Marques-Schoen \cite{kms}.

The negativity of the mean curvature at $0$ turned out to be sufficient for the existence of extremals not only in the case where $\gamma=0$, but also for a large range of $\gamma >0$.

\begin{theorem} {\rm (Chern and Lin \cite{CL5})} Let $\Omega$ be a smooth bounded domain such that $0\in \partial \Omega$. Assume $n\geq 4$, $s\geq 0$, and $0\leq \gamma <\frac{(n-2)^2}{4}$. If the mean curvature at $0$ is negative, then  $\mu_{\gamma, s}(\Omega)$ is attained. 
\end{theorem}
The proof consists of testing the functional on minimizing sequences arising from  suitably truncated extremals of $\mu_{\gamma, s}(\rnp)$, whenever they are attained,
and showing that $\mu_{\gamma,s, \lambda}(\Omega)<\mu_{\gamma, s}(\rnp)$. 

In \cite{gr4} Ghoussoub-Robert consider the rest of the range left by Chern and Lin. 
In order to complete the picture, it was again important to know for which parameters $\gamma$ and $s$, the best constant $\mu_{\gamma, s}(\rnp)$ is attained. This is summarized in the following proposition, whose proof is given in section 5. 

\begin{proposition} Assume $\gamma <\frac{n^2}{4}$, where $n\geq 3$. Then, 
\begin{enumerate}
\item  $\mu_{\gamma, s}(\rnp)$ is attained if either $s>0$ or if $\{s=0$, $\gamma > 0$, and $n\geq 4\}$.
\item  On the other hand, if $s=0$ and $\gamma \leq 0$, then $\mu_{\gamma, s}(\rnp)$ is not attained.
\item The case when $s=0$, $\gamma >0$ and $n=3$ remains unsettled.  
\end{enumerate}
\end{proposition}
\noindent Ghoussoub-Robert first noted that the proof of Chern-Lin extends directly to the case when $\gamma<\frac{n^2-1}{4}$. The limiting case when $\gamma=\frac{n^2-1}{4}$ is already quite more 
involved and requires precise information  on the profile of the extremal for $\mu_{\gamma, s}(\rnp)$. 

However, the case when $\gamma>\frac{n^2-1}{4}$ turned out to be more intricate. The ``local condition" of negative mean curvature at $0$ is not sufficient anymore to ensure extremals for $\mu_{\gamma, s}(\Omega)$. One requires a positivity condition on the {\it Hardy-singular boundary mass} of $\Omega$ defined below. This new ``global notion" associated with the operator $L_\gamma$ could be assigned to any smooth bounded domain $\Omega$ of $\rn$ with $0\in\partial\Omega$, as long as  
$\frac{n^2-1}{4}<\gamma <\frac{n^2}{4}$.

\begin{theorem}\label{def:mass:intro} {\rm (Ghoussoub-Robert \cite{gr4})} Assume $\Omega$ is a smooth bounded domain in $\rn$ with $0\in \partial \Omega$ in such a way that $\frac{n^2-1}{4}<\gamma<\gamma_H(\Omega)$, the latter being the best Hardy constant for the domain $\Omega$.  Then, up to multiplication by a positive constant, there exists a unique function $H\in C^2(\overline{\Omega}\setminus \{0\})$ such that
\begin{equation}\label{carac:H.0}
-\Delta H-\frac{\gamma}{|x|^2}H=0\hbox{ in }\Omega\; ,\; H>0\hbox{ in }\Omega\; ,\; H=0\hbox{ on }\partial\Omega.
\end{equation}
Moreover, there exists $c_1>0$ and $c_2\in \rr$ such that
\begin{equation*}
\hbox{$H(x)=c_1\frac{d(x,\partial\Omega)}{|x|^{\ap}}+c_2\frac{d(x,\partial\Omega)}{|x|^{\am}}+ o\left(\frac{d(x,\partial\Omega)}{|x|^{\am}}\right)$\qquad as $x\to 0$. }
\end{equation*}
The quantity $b_\gamma(\Omega):=\frac{c_2}{c_1}\in \rr$, which is independent of the choice of $H$ satisfying \eqref{carac:H.0}, will be referred to as the ``Hardy-singular boundary mass" of $\Omega$. 
\end{theorem}
 
One can then complete the picture as follows. 

\begin{table}[ht]
\begin{center}
\caption{Case where either $s>0$ or $\{s=0$, $\gamma > 0$, and $n\geq 4\}$.}

\vspace{5mm}
\begin{tabular}{|c|c|c|c|c|} \hline
{\bf  Hardy term }& {\bf Singularity}& {\bf Dim. }&{\bf Geometric condition} &{\bf  Extremal}\\ \hline\hline
$-\infty < \gamma \leq \frac{n^2-1}{4}$ & $s>0$&$n\geq 3$
  & $H_\Omega(0)<0$ & Yes\\  \hline
 $\frac{n^2-1}{4} <\gamma <\frac{n^2}{4}$&$s>0$&$n\geq 3$ & $b_\gamma(\Omega)>0$ & Yes\\ \hline
 $0 < \gamma \leq \frac{n^2-1}{4}$  &$s=0$& $n\geq 4$
  & $H_\Omega(0)<0$ & Yes\\  \hline
 $\frac{n^2-1}{4} <\gamma <\frac{n^2}{4}$&$s=0$&$n\geq 4$ & $b_\gamma(\Omega)>0$ & Yes\\
  \hline
\end{tabular}
\end{center}
\end{table}

\begin{table}[ht]
\begin{center}
\caption{$s=0$ and the remaining cases.}

\vspace{5mm}
\begin{tabular}{|c|c|c|c|c|} \hline
{\bf  Hardy term }& {\bf Singularity}& {\bf Dim. }&{\bf Geometric condition} &{\bf  Extremal}\\ \hline\hline
$\gamma \leq 0$ & $s=0$&$n\geq 3$
  & -- & No\\  \hline
 $0<\gamma \leq 2$&$s=0$&$n= 3$ &$H_\Omega(0)<0$ and $R_{\gamma, 0}(\Omega)>0$& Yes\\ \hline
 $2< \gamma <\frac{9}{4}$  &$s=0$& $n=3$
  & $b_\gamma(\Omega)>0$ and $R_{\gamma, 0}(\Omega)>0$ & Yes\\ 
  \hline
\end{tabular}
\end{center}
\end{table}
The following theorem summarizes the various situations

\begin{theorem} \label{main.boundary} Let $\Omega$ be a smooth bounded domain in $\rn$ ($n\geq 3$) such that $0\in \partial \Omega$ and let $0\leq s < 2$ and $\gamma <\frac{n^2}{4}$. 
\begin{enumerate}
\item If $s=0$ and $\gamma\leq 0$, then $\mu_{\gamma,s}(\Omega)=\mu_{0,0}(\rn)$
and there is no extremal for $\mu_{\gamma,s}(\Omega)$.

\item If either  $s>0$ or \{$s=0$, $\gamma>0$, $n\geq 4$\}, then there are extremals for $\mu_{\gamma,s}(\Omega)$ under one of the following two conditions:
\begin{enumerate} 
\item  $\gamma\leq\frac{n^2-1}{4}$ and the mean curvature of $\partial \Omega$ at $0$ is negative. 
\item $\gamma>\frac{n^2-1}{4}$ and the Hardy boundary-mass $b_\gamma (\Omega)$ of $\Omega$ is positive. 
\end{enumerate} 
\item If $s=0$, $n=3$, $\gamma>0$ and the internal mass $R_{\gamma, 0}(\Omega, x_0)$ is positive for some $x_0\in \Omega$, then there are extremals for $\mu_{\gamma,s}(\Omega)$ under one of the following two conditions:
\begin{enumerate} 
\item  $\gamma\leq 2$ and the mean curvature of $\partial \Omega$ at $0$ is negative. 
\item $\gamma >2$ and the Hardy boundary-mass $b_\gamma (\Omega)$ of $\Omega$ is positive. 
\end{enumerate} 

\end{enumerate}
\end{theorem}

Here are some of the remarkable properties of the Hardy-singular boundary mass. 
\begin{itemize}
\item The map $\Omega \to b_\gamma(\Omega)$ is a monotone increasing function on the class of domains having zero on their boundary, once ordered by inclusion.

\item One can also define the mass of unbounded sets as long as they can be ``inverted" via a Kelvin transform into a smooth bounded domain.  
For example,  $b_\gamma(\rnp)=0$ for any $\frac{n^2-1}{4}<\gamma<\frac{n^2}{4}$, and therefore the mass of any one of its subsets having zero on its boundary is non-positive. In particular, $b_\gamma(\Omega)<0$ whenever $\Omega$ is convex and $0\in \partial \Omega$.

\item  There are also examples of bounded domains $\Omega$ in $\rn$ with $0\in \partial \Omega$ that have positive Hardy-singular boundary mass. Actually these domains can be  {\it locally} strictly convex at $0$.

\item On the other hand, there are also examples of domains $\Omega$ with negative principal curvatures at $0$, but with negative Hardy-singular boundary mass. 
\end{itemize} 
 In other words, the sign of the Hardy-singular boundary mass can be totally independent of the local properties of $\partial\Omega$ around $0$, as illustrated by the following result.

\begin{proposition}\label{prop:ex:any:behave} {\rm (Ghoussoub-Robert \cite{gr4})} Let $\omega$ be a smooth open set of $\rn$ such that $0\in \partial \omega$. Then, there exist  two smooth bounded domains $\Omega_+,\Omega_-$ of $\rn$ with Hardy constants $>\frac{n^2-1}{4}$, and there exists $r_0>0$  such that
$$\Omega_+\cap B_{r_0}(0)=\Omega_-\cap B_{r_0}(0)=\omega\cap B_{r_0}(0),$$
and 
$$b_\gamma(\Omega_+)>0>b_\gamma(\Omega_-),$$
for any $\gamma \in (\frac{n^2-1}{4},\min\{\gamma_H(\Omega_+),\gamma_H(\Omega_-)\})$.\end{proposition}

The above analysis also leads to the following definition of another {\it critical dimension for the operator $L_\gamma$}, which concerns domains having $0$ on their boundary. It is the largest scalar $\bar n_\gamma$ such that for every $n < \bar n_\gamma$, there exists a bounded smooth domain $\Omega \subset \rn$ with $0\in \partial \Omega$ and with negative mean curvature at $0$ such that $\mu_{\gamma, s}(\Omega)$ is not attained.

\noindent{\bf Problem 3:} An interesting question is to verify that if $0\in \partial \Omega$, then the critical dimension for $L_\gamma$ is given by the formula
\begin{equation}
\bar n_\gamma =\left\{\begin{array}{ll} \sqrt{4\gamma +1}\quad & \, {\rm if}\, \gamma \geq 0\\
4 \quad  & \, {\rm if}\,  \gamma <0.
\end{array}\right.
\end{equation}

Note that the above results yield that $\bar n_\gamma \leq \sqrt{4\gamma +1}$ and that   $n < \sqrt{4\gamma +1}$ corresponds to when $\ap-\am < 1$, which is the threshold where the radial function $x\to =|x|^{1-\alpha_+(\gamma)}$ is in $L^2(\partial \rnp)$.

\part{Caffarelli-Kohn-Nirenberg inequalities on $\rn$ and $\rnp$}

\section{Inequalities of Hardy, Sobolev, and Caffarelli-Kohn-Nirenberg}

We start by deriving these inequalities and show how they are interrelated. \\

\noindent{\bf The Hardy inequality:} 
It states that 
\begin{equation}\label{ineq:hardy.0}
\hbox{$\frac{(n-2)^2}{4}\int_{\rn}\frac{u^2}{|x|^2}\, dx\leq \int_{\rn}|\nabla u|^2\, dx$\quad for all $u\in C^\infty_c(\rn)$,}
\end{equation}
which also yields that  $\mu_{0, 2}(\Omega) \geq \frac{(n-2)^2}{4}$ for all $\Omega \subset \rn$, and that  $\mu_{\gamma, s}(\Omega) \geq 0$ for all $\gamma \leq \frac{(n-2)^2}{4}$.  An elementary proof of this inequality goes as follows: 

Associate to any smooth radial positive functions $u \in C^2_c(B_R)$, where $B_R$ is the ball of radius $R$ in $\rn$ the function $v(r)=u(r)r^{(n-2)/2}$ where $r=|x|$. Denoting $\omega_{n-1}$ the volume of the unit sphere, one can estimate the quantity 
\[
I(u):=\int_{\Omega}|\nabla u|^{2}dx -( \frac{n-2}{2})^{2} 
\int_{\Omega}\frac{u^{2}}{|x|^{2}}dx,
\]
 as follows:
\begin{eqnarray*}
I(u)&=&\omega_{n-1}\int^{R}_{0}|\frac{n-2}{2}r^{-n/2}v(r)-r^{1-n/
2}v'(r)|^{2}r^{n-1}d
r
-( \frac{n-2}{2})^{2}\omega_{n-1}\int^{R}_{0}\frac{v^{2}(r)}{r}dr \\
&=&\omega_{n-1}( 
\frac{n-2}{2})^{2}\int^{R}_{0}v^{2}\big[(1-\frac{2v'(r)r}{(n-2)v(r)})^{2}-1\big]\frac{dr}{r}\\
&=&\omega_{n-1}\int^{R}_{0} v'(r)^{2}r\, dr-\omega_{n-1}( \frac{n-2}{2})\int^{R}_{0}v(r)v'(r)dr \\
&=&\omega_{n-1}\int^{R}_{0} v'(r)^{2}r\, dr, 
\end{eqnarray*}
which is obviously non-negative. 

If now $u$ is a non-radial  function  on general domain $\Omega$,  we consider its symmetric decreasing rearrangement $u^*$, defined by
 \[
u^*(x)=\int_0^{+\infty} \chi^*_{\{|u|>t\}}(x)\, dt,
\]
where for a general set $A \subset \rn$, we denote by $\chi^*_A$ the characteristic function of
a ball of volume $|A|$ centered at the origin. the function  $u^*$ is then symmetric-decreasing, and satisfies $\|\frac{u^*}{|x|}\|_p\geq \|\frac{u}{|x|}\|_p$ for any $p$, since the rearrangement does not change the values of $u$, while only changing the places where these values occur.  What is less obvious is that 
\begin{equation}
\int_\Omega |\nabla u^*|^2\, dx \leq \int_\Omega |\nabla u|^2\, dx, 
\end{equation} 
a proof of which can be found in \cite{17baernstein-unified}.

 Let now $B_{R}$ be a ball 
having the same volume as $\Omega$ with $R=(|\Omega|/\omega_{n})^{1/n}$. If $u \in 
H_{0}^{1}(\Omega)$, then $u^{*} \in H_{0}^{1}(B_{R})$, has the same $L^{p}$-norm as $u$, while  decreasing the Dirichlet energy. Hence, 
(\ref{ineq:hardy.0}) holds for every $u \in H^{1}_{0}(\Omega)$.

To see that
$$\gamma_H(\Omega):=\inf\left\{\frac{\int_{\Omega}|\nabla u|^2\, dx}{\int_{\Omega}\frac{u^2}{|x|^2}\, dx};\,  u\in D^{1,2}(\Omega)\setminus\{0\}\right\}$$
is not achieved, if the singularity $0$ belongs to the interior of $\Omega$, 
assume that $u\geq 0$ is a weak solution of the corresponding Euler-Lagrange equation.
 \begin{equation*}\label{att.equ}
 \begin{cases}
\Delta u+\left( \frac{n-2}{2}\right)^{2}\frac{u}{|x|^{2}}&=0 \ \ \ \ {\rm in} \ \  
\Omega, \\
\hfill u&>0 \ \ \ \ {\rm in} \ \   \Omega \setminus \{0\}, \\
\hfill u&=0 \ \quad {\rm in} \ \  \partial \Omega.
\end{cases}
\end{equation*}
 By standard
elliptic regularity we know that $u \in C^{2,\alpha}_{loc}(\Omega
\setminus \{0\})$. Since $0\in \Omega$, we can assume that the unit ball $B_1$ is contained in $\Omega$.
The function 
\[v(r)=\frac{1}{n\omega_{n-1}r^{n-1}}\int_{\partial B_r}u(x)dS=\frac{1}{n \omega_{n-1}}\int_{|\sigma|=1} u(r\sigma)d\sigma,\]
then satisfies,
\begin{equation*}\label{att.ode}
v''(r)+\frac{n-1}{r}v'(r)+\frac{(\frac{n-2}{2})^2}{r^2}v(r)=0.
\ \ \ \ 0<r\leq1,
\end{equation*}
Hence the function 
$w(r)=r^{(n-2)/2}v(r)>0$ for $r>0$, satisfies
$(rw')'= 0$ for $0<r\leq 1$, and therefore $w'(r)=\frac{C}{r}$ for some constant $C>0$ and $w(r)=C\ln (r)+D$. 
On the other hand, the Sobolev inequality yields that  if $u\in H^1_0(\Omega)$, then $u \in L^{2n/(n-2)}(B_1)$ and
$\liminf\limits_{r\downarrow 0}w(r)=0$, which would lead to a contradiction.

More recently, it was observed by Brezis-Vasquez \cite{BV} and others \cite{FT} that the inequality \label{cl-hardy} can be improved.
The story here is the link --discovered by Ghoussoub-Moradifam \cites{gm0, gm}-- between various improvements of this inequality confined to bounded domains and Sturm's theory regarding the oscillatory behavior of certain linear ordinary equations.

Following Ghoussoub-Moradifam \cite{gm}, we say that a non-negative  $C^1$-function $P$ defined on an interval $(0, R)$ is a {\it Hardy Improving Potential} (abbreviated as {\it HI-potential}) 
if the  following  improved Hardy inequality holds on every domain $\Omega$ contained in a ball of radius $R$: 
\begin{equation}\label{gen-hardy.0}
\hbox{$\int_{\Omega}|\nabla u |^{2}dx - ( \frac{n-2}{2})^{2} 
\int_{\Omega}\frac{u^{2}}{|x|^{2}}dx\geq \int_{\Omega} P(|x|)u^{2}dx$ \quad for $u \in 
H^{1}_{0}(\Omega)$.}
\end{equation}
It turned out that a necessary and sufficient condition for $P$ to be an {\it HI-potential} on a ball $B_R$,  is for the following ordinary differential equation associated to $P$ 
\begin{equation}\label{}
y''+\frac{1}{r}y'+P(r)y=0,
\end{equation}
to have a positive solution on the interval $(0, R)$. Elementary examples of HI-potentials are:
\begin{itemize}
\item  $P \equiv 0$ on any interval $(0,  R)$; 
\item $P\equiv 1$ on $(0, z_0)$, where $z_{0}=2.4048...$ is the first root of the Bessel function $J_0$;
\item More generally,  $P (r)=r^{-a}$ with $0\leq a<2$ on $(0, z_a)$, where  $z_{a}$ is the first root of the largest solution of the equation $y''+\frac{1}{r}y'+r^{-a}y=0$. 
\item  $P_\rho(r)=\frac{1}{4r^{2}(log\frac{\rho}{r})^{2}}$ on $(0, \frac{\rho}{e})$;
\item  $P_{k, \rho}(r)=\frac{1}{r^{2}}\sum\limits_{j=1}^k\big(\prod^{j}_{i=1}log^{(i)}\frac{\rho}{r}\big)^{-2}$ on $(0, \frac{\rho}{e^{e^{e^{.^{.^{e(k-times)}}}}}} )$.
\end{itemize}
This connection to the oscillatory theory of ODEs leads to a large supply of explicit Hardy improving potentials. One can show for instance that there is no $c>0$ for which $P(r)=cr^{-2}$ is an $HI$-potential, which means that $\frac{(n-2)^2}{4}$ is the best constant for $\gamma_H(\Omega)$. 
 
Actually, the value of the following best constant 
\begin{equation}\mu_{0,2}(P, \Omega):= \inf_
{\genfrac{}{}{0pt}{}{\scriptstyle{u\in
H^1_0(\Omega)}}{\scriptstyle{u\neq 0}}}
~\frac{\displaystyle\int_{\Omega}|\nabla u|^2~dx-\int_{\Omega}P(|x|)u^2~dx}
{\displaystyle \int_{\Omega}|x|^{-2}|u|^2~dx}
\end{equation}
is still equal to  $\frac{(n-2)^2}{4}$, and is never attained in $H^1_0(\Omega)$, whenever $\Omega$ contains $0$ in its interior. \\

\noindent{\bf The Hardy-Sobolev inequalities:} The basic Sobolev inequality states that there exists a constant $C(n)>0$, such that 
\begin{equation}\label{ineq:sobo}
\hbox{$\left(\int_{\rn}|u|^{\frac{2n}{n-2}}\, dx\right)^{\frac{n-2}{n}}\leq C(n)\int_{\rn}|\nabla u|^2\, dx$ \quad for all $u\in C^\infty_c(\rn)$,   }
\end{equation}
in such a way that $\mu_{0, 0} (\Omega)>0$ for every $\Omega \subset \rn$.   Actually, the Sobolev inequality can be derived from Hardy's except for the value of the best constant, which we will discuss later. We first derive the inequality for radial decreasing functions. The general case follows  from the properties of symmetric rearrangements noted above. The argument goes as follows: If $u$ is radial and decreasing and $p>2$, then for any $y\in \rn$ we have
 \[
\|u\|_p^p= \int_{\rn}|u|^p\, dx \geq u(y)^p|y|^n\omega_{n}, 
 \]
 where $\omega_{n}$ is the volume of the unit ball in $\rn$. 
Now take this to the power $1-\frac{2}{p}$, multiply by  $|u(y)|^2|y|^{\frac{n(2-p)}{p}}$ and integrate over $y$ to obtain 
\[
\int_{\rn}\frac{|u(y)|^2}{|y|^{\frac{n(p-2)}{p}}}\, dy \geq \omega_{n}^{1-\frac{2}{p}}\|u\|_p^2.
\]
It now suffices to take $p:=\frac{2n}{n-2}$ and use Hardy's inequality to conclude.

A H\"older-type interpolation between the Hardy and Sobolev  inequalities yields the Hardy-Sobolev inequality, which states that for any $s\in [0, 2]$, there exists $C(s,n)>0$ such that
\begin{equation}\label{ineq:hs}
\hbox{$\left(\int_{\rn}\frac{|u|^{\crits}}{|x|^s}\, dx\right)^{\frac{2}{2^\star}}\leq C(s,n)\int_{\rn}|\nabla u|^2\, dx$ \quad for all $u\in C^\infty_c(\rn)$, }
\end{equation}
where
$\crits:=\frac{2(n-s)}{n-2}.$ In other words, $\mu_{0, s}(\Omega)>0$ for every $s\in [0, 2]$. 

Indeed, by applying H\"older's inequality, then Hardy's and Sobolev's, we get
\begin{eqnarray*}
\int_{\rn} \frac{|u|^{ 2^{*}(s)}}{|x|^{s}} \, dx&=&
\int_{\rn} \frac{|u|^{s}}{|x|^{s}}\cdot |u|^{2^{*}(s)-s}\, dx\\
&\leq& (\int_{\rn}\frac{|u|^{2}}{|x|^{2}}\, dx)^{\frac{s}{2}}
(\int_{\rn} |u|^{(2^{*}(s)-s)\frac{2}{2-s}}dx)^{\frac{2-s}{2}}\\
&=&  (\int_{\rn}\frac{|u|^{2}}{|x|^{2}}\, dx)^{\frac{s}{2}}
(\int_{\rn} |u|^{2^{*}}dx)^{\frac{2-s}{2}}\\
&\leq& (C_1\int_{\rn} |\nabla u|^{2})^{\frac{s}{2}}\, dx)
(C_2\int_{\rn} |\nabla u|^{2}dx)^{\frac{2^{*}}{2}\cdot \frac{2-s}{2}}\\
&=& C (\int_{\rn} |\nabla u|^{2})^{\frac{n-s}{n-2}}\, dx.
\end{eqnarray*}
It is remarkable that when $s\in (0,2)$, the Hardy-Sobolev inequality inherits the singularity at $0$ from the Hardy inequality and the superquadratic exponent from the Sobolev inequality. 

Now what about the dependence on $\gamma$.  
Combining the above three inequalities, one obtain that for each $\gamma < \frac{(n-2)^2}{4} \leq \gamma_H(\Omega)=\mu_{0, 2}(\Omega)$, the latter being the best constant   in the  Hardy inequality on $\Omega$,  we have that inequality (\ref{general}) holds with $C>0$, in other words,
\begin{equation}\label{general.bis}
\hbox{$\mu_{\gamma, s}(\Omega) >0$ for all $s\in [0, 2]$ and $\gamma < \frac{(n-2)^2}{4}$.}
\end{equation}
We shall see later that this may hold true for values of $\gamma$ beyond $\frac{(n-2)^2}{4}$. \\

\noindent{\bf The Caffarelli-Kohn-Nirenberg inequalities:} We now show that (\ref{general.bis}) also contains the celebrated Caffarelli-Kohn-Nirenberg inequalities \cite{ckn}, which state that there is a constant $C:=C(a,b,n)>0$ such that the following inequality holds:
\begin{equation} \label{CKN:1}
\hbox{$\left(\int_{\rn}|x|^{-bq}|u|^q \right)^{\frac{2}{q}}\leq C\int_{\rn}|x|^{-2a}|\nabla u|^2 dx $ \quad for all $u\in C^\infty_0(\rn)$,}
\end{equation}
where
\begin{equation}\label{cond1:bis}
-\infty<a<\frac{n-2}{2}, \ \ 0 \leq b-a\leq 1, \ \ {\rm and}\ \ q=\frac{2n}{n-2+2(b-a)}.
\end{equation}
Indeed, by setting $w(x)=|x|^{-a}u(x)$, we see that for 
any $u \in C_0^{\infty}(\Omega)$,
\begin{eqnarray*}
\int_{\Omega}|x|^{-2a}|\nabla u|^2 dx&=&\int_{\Omega}|x|^{-2a}(a^2|x|^{2a-2} w^2+2a|x|^{2a-2}w x.\nabla w+|x|^{2a}|\nabla w|^2)dx\\
&=&\int_{\Omega}|\nabla w|^2dx +a^2\int_{\Omega}\frac{w^2}{|x|^2}dx+\int_{\Omega} 2a|x|^{-2}wx.\nabla wdx\\
&=&\int_{\Omega}|\nabla w|^2dx +a^2\int_{\Omega}\frac{w^2}{|x|^2}dx+a\int_{\Omega} |x|^{-2}x.\nabla (w^2)dx\\
&=&\int_{\Omega}|\nabla w|^2dx -\gamma \int_{\Omega}\frac{w^2}{|x|^2}dx,
\end{eqnarray*}
with 
$\gamma= a(n-2-a)$,
and where  the last equality is obtained by integration by parts.  Now note that if $a<\frac{n-2}{2}$, then by Hardy's inequality, $\int_{\Omega}|x|^{-2a}|\nabla u|^2 dx <+\infty$ if and only if 
both $\int_{\Omega}|\nabla w|^2 dx <+\infty$ and $\int_{\Omega}\frac{|w|^2}{|x|^2} dx <+\infty$.
Furthermore, 
\begin{equation}
\frac{\int_{\Omega}|x|^{-2a}|\nabla u|^2dx}{\left(\int_{\Omega}|x|^{-bq}|u|^q \right)^{\frac{2}{q}}}=\frac{\int_{\Omega} |\nabla w|^2-\gamma \int_{\Omega}\frac{w^2}{|x|^2}dx}{(\int_{\Omega}\frac{w^{2^*(s)}}{|x|^s}dx)^{\frac{2}{2^*(s)}}}, 
\end{equation}
where $s=(b-a)q$. This readily implies that (\ref{general}) and (\ref{CKN:1}) are equivalent under the above conditions on $a, b, q, s,$ and $\gamma$.  \hfill $\Box$ \\

\section{Caffarelli-Kohn-Nirenberg type inequalities on $\rnp$}

\noindent{\bf A general form for the Hardy-Sobolev inequality:}  The following has been noted by many authors. See for example \cites{gm, craig}.  

\begin{proposition} \label{cowan} Let $\Omega$ be an open subset of $\rn$ and consider $\rho\in C^\infty(\Omega)$ such that $\rho>0$ and $-\Delta\rho>0$. Then for any $u\in \dundeux$ we have that  $\sqrt{\rho^{-1}(-\Delta)\rho}u\in L^2(\Omega)$ and 
\begin{equation}
\int_\Omega\frac{-\Delta \rho}{\rho} u^2\, dx\leq \int_\Omega |\nabla u|^2\, dx.\label{ineq:hardy}
\end{equation}
Moreover, the case of equality is achieved exactly on $\rr\rho\cap \dundeuxr$. In particular, if $\rho\not\in \dundeux$, there are no nontrival extremals for \eqref{ineq:hardy}.
\end{proposition}
\noindent The proof relies on the following integral identity:
\begin{equation*}\label{id:ipp}
\int_\Omega |\nabla (\rho v)|^2\, dx-\int_\Omega\frac{-\Delta\rho}{\rho}(\rho v)^2\, dx=\int_\Omega\rho^2|\nabla v|^2\, dx\geq 0
\end{equation*}
for all $v\in C^\infty_c(\Omega)$. This identity is a straightforward integration by parts. Since $\rho,-\Delta\rho>0$ in $\Omega$, it follows from density arguments that for any $u\in \dundeux$, then $\sqrt{\rho^{-1}(-\Delta)\rho}u\in L^2(\Omega)$ and \eqref{ineq:hardy} holds. 

There are many interesting examples of  weights of the form $\frac{-\Delta \rho}{\rho}$ besides $\frac{(n-2)^2}{4|x|^2}$, which could reflect  the nature of the domain. Here is one that will concern us throughout this paper. 

Fix $1\leq k\leq n$, and take $\rho(x):=x_1...x_k|x|^{-\alpha}$ for all $x\in\Omega:=\rr_+^k\times\rr^{n-k}\setminus\{0\}$. Then $\frac{-\Delta \rho}{\rho}=\frac{\alpha(n+2k-2-\alpha)}{|x|^2}$. Maximize the constant by taking $\alpha:=(n+2k-2)/2$. Since $\rho\not\in D^{1,2}(\rr_+^k\times\rr^{n-k})$, the above proposition applies and we obtain that for all $u\in D^{1,2}(\rr_+^k\times\rr^{n-k})$, 
\begin{equation}\label{ineq:hardy:rk}
\left(\frac{n+2k-2}{2}\right)^2\int_{\rr_+^k\times\rr^{n-k}}\frac{u^2}{|x|^2} \, dx\leq \int_{\rr_+^k\times\rr^{n-k}} |\nabla u|^2\, dx.
\end{equation}

Actually, we have that 

\begin{equation}\label{rnpk}
\left(\frac{n+2k-2}{2}\right)^2=\inf_u \frac{\int_{\rr_+^k\times\rr^{n-k}} |\nabla u|^2\, dx}{\int_{\rr_+^k\times\rr^{n-k}}\frac{u^2}{|x|^2} \, dx},
\end{equation}
where the infimum, taken over all $u\in D^{1,2}(\rr_+^k\times\rr^{n-k})\setminus\{0\}$, is never achieved. Note that, in particular,
\begin{equation}
\gamma_H(\rnp):=\mu_{0, 2}(\rnp)=\frac{n^2}{4}.
\end{equation}

By H\"older-interpolating between the above general Hardy inequality and the Sobolev inequality, one gets the following  
generalized Caffarelli-Kohn-Nirenberg inequality.

\begin{proposition}\label{prop:ckn:gen} Let $\Omega$ be an open subset of $\rn$. Let $\rho,\rho'\in C^\infty(\Omega)$ be such that $\rho,\rho'>0$ and $-\Delta\rho,-\Delta\rho'>0$. Fix $s\in [0,2]$ and assume that there exists $\eps\in (0,1)$ and $\rho_\eps\in C^\infty(\Omega)$ with $\rho_\eps,-\Delta\rho_\eps>0$ such that
$$\frac{-\Delta \rho}{\rho}\leq (1-\eps)\frac{-\Delta\rho_\eps}{\rho_\eps}\quad \hbox{\rm on }\Omega.
$$
Then, for all $u\in C^\infty_c(\Omega)$, 
\begin{equation}\label{ineq:ckn:36}
\left(\int_\Omega \left(\frac{-\Delta \rho'}{\rho'}\right)^{s/2}\rho^{\crits}|u|^{\crits}\, dx\right)^{\frac{2}{\crits}}\leq C \int_{\Omega}\rho^{2}|\nabla u|^2\, dx.
\end{equation}
\end{proposition}
\noindent By applying the above to $\rho(x)=\rho'(x)=\left(\Pi_{i=1}^kx_i \right)|x|^{-a}$ and $\rho_\eps(x)=\left(\Pi_{i=1}^kx_i \right)|x|^{-\frac{n-2+2k}{2}}$ for $x\in\rr_+^k\times\rr^{n-k}$, by noting  that
$$\frac{\Delta\rho'}{\rho'}=\frac{a(n-2+2k-a)}{|x|^2}\hbox{ and }\frac{-\Delta\rho_\eps}{\rho_\eps}=\frac{(n-2+2k)^2}{4|x|^2},$$
and by applying Proposition \ref{prop:ckn:gen} with suitably chosen $a,b,q$, we get the following  inequalities isolated by Ghoussoub-Robert \cite{gr4}, which reduce to the Caffarelli-Kohn-Nirenberg inequalities when $k=0$. 
\begin{equation} \label{CKN:2}
\left(\int_{\rr_+^k\times\rr^{n-k}}|x|^{-bq}\left(\Pi_{i=1}^kx_i\right)^{q}|u|^q \right)^{\frac{2}{q}}\leq C\int_{\rr_+^k\times\rr^{n-k}}\left(\Pi_{i=1}^kx_i\right)^{2}|x|^{-2a}|\nabla u|^2 dx,
\end{equation}
where
\begin{equation}\label{cond2}
-\infty<a<\frac{n-2+2k}{2}, \ \ 0 \leq b-a\leq 1\,\,  {\rm and}\,\, q=\frac{2n}{n-2+2(b-a)}.
\end{equation}

\section{Attainability of the extremals on $\rn$ and $\rnp$}

Let $\cone$ be an open connected cone of $\rn$, 
$n\geq 3$, centered at $0$, that is
\begin{equation}\label{def:cone}
\left\{\begin{array}{c}
\cone \hbox{ is a domain (that is open and connected)}\\
\forall x\in\cone,\, \forall r>0, \; rx\in\cone.
\end{array}\right.
\end{equation}
Fix $\gamma<\gamma_H(\cone)$, and consider the question of whether there is an extremal $u_0\in\dundeuxc\setminus\{0\}$, where $\mu_{\gamma, s}(\cone)$ is attained. The question of the extremals on general cones has been tackled by Egnell \cite{egnell} in the case $\{\gamma=0\hbox{ and }s>0\}$. Theorem \ref{th:ext:cone} below has been noted in several contexts by Bartsche-Peng-Zhang \cite{BPZ} and Lin-Wang \cite{CL5}. 
We shall sketch below an independent proof. 

\begin{theorem}\label{th:ext:cone} Let $\cone$ be a cone of $\rn$, $n\geq 3$, as in \eqref{def:cone}, $s\in [0,2)$ and  $\gamma<\gamma_H(\cone$). 
\begin{enumerate}
\item If either $\{s>0\}$ or $\{s=0$, $\gamma>0$, $n\geq 4\}$, then extremals for $\mu_{\gamma, s}(\cone)$ exist.
\item If $\{s=0$ and $\gamma< 0\}$, there are no extremals for $\mu_{\gamma,0}(\cone)$.
\item If $\{s=0$ and $\gamma= 0\}$, there are extremals for $\mu_{0,0}(\cone)$ if and only if there exists $z\in \rn$ such that $(1+|x-z|^2)^{1-n/2}\in \dundeuxc$ (in particular $\overline{\cone}=\rn$).
\end{enumerate}
Moreover, if there are no extremals for $\mu_{\gamma, 0}(\cone)$, then $\mu_{\gamma, 0}(\cone)=\mu_{0, 0}(\cone)$, and 
\begin{equation}
\mu_{\gamma, 0}(\cone)=\frac{1}{K(n,2)^2}:=\mu_{0, 0}(\rn)=\inf_{u\in D^{1,2}(\rn)\setminus\{0\}}\frac{\int_{\rn}|\nabla u|^2\, dx}{\left(\int_{\rn}|u|^{\crit}\, dx\right)^{\frac{2}{\crit}}}.
\end{equation}
\end{theorem}

\begin{rem} Note that the case when $\{s=0$, $n=3$ and $\gamma>0\}$ remains unsettled. 
\end{rem} 

\medskip\noindent We isolate two corollaries. The first one is essentially what we need when  $\cone=\rnp$. The second deals with the case $\cone=\rn$. There is no issue for $n=3$ in the second corollary.
\begin{coro}\label{coro:cone:1} Let $\cone$ be a cone of $\rn$, $n\geq 3$, as in \eqref{def:cone} such that $\overline{\cone}\neq\rn$. We let $s\in [0,2)$ and  $\gamma<\gamma_H(\cone)$. Then, 
\begin{enumerate}
\item If $\{s>0\}$ or $\{s=0$, $\gamma>0$, $n\geq 4\}$, then there are extremals for $\mu_{\gamma, s}(\cone)$.
\item If $\{s=0$ and $\gamma\leq 0\}$, there are no extremals for $\mu_{\gamma,0}(\cone)$.
\end{enumerate}

\end{coro}

\begin{coro}\label{coro:cone:2} Let $\cone$ be a cone of $\rn$, $n\geq 3$, as in \eqref{def:cone}. We assume that there exists $z\in \rn$ such that $(1+|x-z|^2)^{1-n/2}\in \dundeuxc$ (in particular, if $\overline{\cone}=\rn$). We fix $s\in [0,2)$ and $\gamma<\gamma_H(\cone$). Then, 
\begin{enumerate}
\item If $\{s>0\}$ or $\{s=0$ and $\gamma\geq 0\}$, then there are extremals for $\mu_{\gamma, s}(\cone)$.
\item If $\{s=0$ and $\gamma< 0\}$, there are no extremals for $\mu_{\gamma,0}(\cone)$.
\end{enumerate}
\end{coro}
\begin{rem} \label{used} \rm We shall frequently use the following simple observations: 
If $s=0$, then for all $\gamma$, we always have $\mu_{\gamma, 0}(\Omega) \leq \frac{1}{K(n,2)^2}$. Indeed, fix $x_0\in\Omega \setminus\{0\}$ and let $\eta\in C^\infty_c(\Omega)$ be such that $\eta(x)=1$ around $x_0$. Set $\ue(x):=\eta(x)\left(\frac{\eps}{\eps^2+|x-x_0|^2}\right)^{\frac{n-2}{2}}$. 
Since $x_0\neq 0$, it is easy to check that $\lim_{\eps\to 0}\int_\Omega \frac{\ue^2}{|x|^2}\, dx=0$. It is also classical (see for example Aubin \cite{aubin}) that $$\lim_{\eps\to 0}\frac{\int_{\Omega}|\nabla \ue|^2\, dx}{\left(\int_\Omega |\ue|^{\crit}\, dx \right)^{\frac{2}{\crit}}}= \frac{1}{K(n,2)^2}.$$ 
It follows that $\mu_{\gamma,0}(\Omega)\leq \frac{1}{K(n,2)^2}$. 

\noindent As an easy consequence, we get that if  $s=0$ and $\gamma \leq 0$, then $\mu_{\gamma, 0}(\Omega) = \frac{1}{K(n,2)^2}$. 
\end{rem}

\smallskip\noindent{\it Proof of Theorem \ref{th:ext:cone}:} This goes as the classical proof of the existence of extremals for the Sobolev inequalities using Lions's concentration-compactness Lemmae (\cites{lions1,lions2}, see also Struwe \cite{st} for an exposition in book form).

\smallskip\noindent We let $(\tilde{u}_k)_k\in \dundeuxrnp$ be a minimizing sequence for $\mu_{\gamma, s}(\cone)$ such that
$$\int_{\cone}\frac{|\tilde{u}_k|^{\crits}}{|x|^s}\, dx=1\hbox{ and }\lim_{k\to +\infty}\int_{\cone}\left(|\nabla \tilde{u}_k|^2-\frac{\gamma}{|x|^2}\tilde{u}_k^2\right)\, dx=\mu_{\gamma, s}(\cone).$$
For any $k$, there exists $r_k>0$ such that $\int_{B_{r_k}(0)\cap\cone}\frac{|\tilde{u}_k|^{\crits}}{|x|^s}\, dx=1/2$. Define $u_k(x):=r_k^{\frac{n-2}{2}}u_k(r_k x)$ for all $x\in\cone$. Since $\cone$ is a cone, we have that $u_k\in \dundeuxc$. We then have that
\begin{equation}\label{mes:l}
\lim_{k\to +\infty}\int_{\cone}\left(|\nabla u_k|^2-\frac{\gamma}{|x|^2}u_k^2\right)\, dx=\mu_{\gamma, s}(\cone),
\end{equation}
and
\begin{equation}\label{mes:n}
\int_{\cone}\frac{|u_k|^{\crits}}{|x|^s}\, dx=1 \; , \; \int_{B_1(0)\cap\cone}\frac{|u_k|^{\crits}}{|x|^s}\, dx=\frac{1}{2}.
\end{equation}

\medskip\noindent
We first claim that, up to a subsequence,
\begin{equation}\label{cpct}
\lim_{R\to +\infty}\lim_{k\to +\infty}\int_{B_R(0)\cap\cone}\frac{|u_k|^{\crits}}{|x|^s}\, dx=1.
\end{equation}
Indeed, for $k\in\nn$ and $r\geq 0$, we define 
$$Q_k(r):=\int_{B_r(0)\cap\cone}\frac{|u_k|^{\crits}}{|x|^s}\, dx.$$
Since $0\leq Q_k\leq 1$ and $r\mapsto Q_k(r)$ is nondecreasing for all $k\in\nn$, then, up to a subsequence, there exists $Q: [0,+\infty)\to\rr$ nondecreasing such that $Q_k(r)\to Q(r)$ as $k\to +\infty$ for a.e. $r>0$. Set
$$\alpha:=\lim_{r\to +\infty}Q(r).$$
It follows from \eqref{mes:l} and \eqref{mes:n} that $\frac{1}{2}\leq \alpha\leq 1$. Up to taking another subsequence, there exists $(R_k)_k,(R'_k)_k\in (0,+\infty)$ such that
$$\left\{\begin{array}{l}
2 R_k\leq R'_k\leq 3R_k\hbox{ for all }k\in\nn,\\
\lim_{k\to +\infty}R_k=\lim_{k\to +\infty}R'_k=+\infty,\\
\lim_{k\to +\infty}Q_k(R_k)=\lim_{k\to +\infty}Q_k(R'_k)=\alpha.
\end{array}\right.$$
In particular,
\begin{equation}\label{dicho}
\lim_{k\to +\infty}\int_{B_{R_k}(0)\cap\cone}\frac{|u_k|^{\crits}}{|x|^s}\, dx=\alpha\hbox{\, and \, }\lim_{k\to +\infty}\int_{(\rn\setminus B_{R'_k}(0))\cap\cone}\frac{|u_k|^{\crits}}{|x|^s}\, dx=1-\alpha.
\end{equation}
We claim that
\begin{equation}\label{lim:u2}
\lim_{k\to +\infty}R_k^{-2}\int_{(B_{R'_k}(0)\setminus B_{R_k}(0))\cap\cone}u_k^2\, dx=0.
\end{equation}
Indeed, for all $x\in B_{R'_k}(0)\setminus B_{R_k}(0)$, we have that $R_k\leq |x|\leq 3R_k$. Therefore, H\"older's inequality yields
\begin{eqnarray*}
\int_{(B_{R'_k}(0)\setminus B_{R_k}(0))\cap\cone}u_k^2\, dx&\leq & \left(\int_{(B_{R'_k}(0)\setminus B_{R_k}(0))\cap\cone}\, dx\right)^{1-\frac{2}{\crits}}\left(\int_{(B_{R'_k}(0)\setminus B_{R_k}(0))\cap\cone}|u_k|^{\crits}\, dx\right)^\frac{2}{\crits}\\
&\leq & C R_k^{2}\left(\int_{(B_{R'_k}(0)\setminus B_{R_k}(0))\cap\cone}\frac{|u_k|^{\crits}}{|x|^s}\, dx\right)^\frac{2}{\crits}
\end{eqnarray*}
for all $k\in\nn$. Conclusion \eqref{lim:u2} then follows from \eqref{dicho}. 

\medskip\noindent We now let $\varphi\in C^\infty_c(\rn)$ be such that $\varphi(x)=1$ for $x\in B_1(0)$ and $\varphi(x)=0$ for $x\in \rn\setminus B_2(0)$. For $k\in\nn$, we define
$$\varphi_k(x):=\varphi\left(\frac{|x|}{R'_k-R_k}+\frac{R'_k-2R_k}{R'_k-R_k}\right)\hbox{ for all }x\in\rn.$$
One can easily check that $\varphi_ku_k,(1-\varphi_k)u_k\in \dundeuxc$ for all $k\in\nn$. Therefore, we have that
\begin{eqnarray*}
\int_{\cone}\frac{|\varphi_k u_k|^{\crits}}{|x|^s}\, dx&\geq& \int_{B_{R_k}(0)\cap\cone}\frac{|u_k|^{\crits}}{|x|^s}\, dx=\alpha+o(1),\\
\int_{\cone}\frac{|(1-\varphi_k) u_k|^{\crits}}{|x|^s}\, dx&\geq& \int_{(\rn\setminus B_{R'_k}(0))\cap\cone}\frac{|u_k|^{\crits}}{|x|^s}\, dx=1-\alpha+o(1)
\end{eqnarray*}
as $k\to +\infty$. The Hardy-Sobolev inequality 
and \eqref{lim:u2} yield
\begin{eqnarray*}
&&\mu_{\gamma,s}(\cone)\left(\int_{\cone}\frac{|\varphi_k u_k|^{\crits}}{|x|^s}\, dx\right)^{\frac{2}{\crits}} \leq \int_{\cone} \left(|\nabla (\varphi_k u_k)|^2-\frac{\gamma}{|x|^2}\varphi_k^2u_k^2\right)\, dx\\
&&\leq \int_{\cone} \varphi_k^2\left(|\nabla u_k|^2-\frac{\gamma}{|x|^2}u_k^2\right)\, dx+O\left(R_k^{-2}\int_{(B_{R'_k}(0)\setminus B_{R_k}(0))\cap\cone}u_k^2\, dx\right)\\
&&\leq \int_{\cone} \varphi_k^2\left(|\nabla u_k|^2-\frac{\gamma}{|x|^2}u_k^2\right)\, dx+o(1)
\end{eqnarray*}
as $k\to +\infty$. Similarly,
\begin{eqnarray*}
&&\mu_{\gamma,s}(\cone)\left(\int_{\cone}\frac{|(1-\varphi_k) u_k|^{\crits}}{|x|^s}\, dx\right)^{\frac{2}{\crits}} \leq \int_{\cone} (1-\varphi_k)^2\left(|\nabla u_k|^2-\frac{\gamma}{|x|^2}u_k^2\right)\, dx+o(1)
\end{eqnarray*}
as $k\to +\infty$. Therefore, we have that
\begin{eqnarray*}
&&\mu_{\gamma,s}(\cone)\left(\alpha^{\frac{2}{\crits}}+(1-\alpha)^{\frac{2}{\crits}}+o(1)\right)\\
&&\leq \mu_{\gamma,s}(\cone)\left(\left(\int_{\cone}\frac{|\varphi_k u_k|^{\crits}}{|x|^s}\, dx\right)^{\frac{2}{\crits}}+\left(\int_{\cone}\frac{|(1-\varphi_k) u_k|^{\crits}}{|x|^s}\, dx\right)^{\frac{2}{\crits}}\right)\\
&&\leq \int_{\cone} ( \varphi_k^2+(1-\varphi_k)^2)\left(|\nabla u_k|^2-\frac{\gamma}{|x|^2}u_k^2\right)\, dx+o(1)\\
&&\leq \int_{\cone} ( 1-2\varphi_k(1-\varphi_k))\left(|\nabla u_k|^2-\frac{\gamma}{|x|^2}u_k^2\right)\, dx+o(1)\\
&&\leq \mu_{\gamma,s}(\cone)+2\int_{\cone} \varphi_k(1-\varphi_k)\frac{\gamma}{|x|^2}u_k^2\, dx+o(1)\\
&&\leq \mu_{\gamma,s}(\cone)+O\left(R_k^{-2}\int_{(B_{R'_k}(0)\setminus B_{R_k}(0))\cap\cone}u_k^2\, dx\right)+o(1)\leq \mu_{\gamma,s}(\cone)+o(1)
\end{eqnarray*}
as $k\to +\infty$. Hence, $\alpha^{\frac{2}{\crits}}+(1-\alpha)^{\frac{2}{\crits}}\leq 1$, which implies that $\alpha=1$ since $0<\alpha\leq 1$. This proves the claim in (\ref{cpct}).

We now claim that there exists $u_\infty\in \dundeuxc$ such that $u_k\rightharpoonup u_\infty$ weakly in $\dundeuxc$ as $k\to +\infty$, $x_0\neq 0$ 
such that
\begin{eqnarray}
&\quad \hbox{either }\lim_{k\to +\infty}\frac{|u_k|^{\crits}}{|x|^s}{\bf 1}_{\cone}\, dx=\frac{|u_\infty|^{\crits}}{|x|^s}{\bf 1}_{\cone}\, dx&\hbox{ and }\int_{\cone}\frac{|u_\infty|^{\crits}}{|x|^s}\, dx=1\label{case:1}\\
&\hbox{ or }\lim_{k\to +\infty}\frac{|u_k|^{\crits}}{|x|^s}{\bf 1}_{\cone}\, dx=\delta_{x_0}&\hbox{ and }u_\infty\equiv 0.\label{case:2}
\end{eqnarray}
Arguing as above, we get that for all $x\in \rn$, we have that
\begin{equation*}
\lim_{r\to 0}\lim_{k\to +\infty}\int_{B_r(0)\cap\cone}\frac{|u_k|^{\crits}}{|x|^s}\, dx=\alpha_x\in \{0,1\}.
\end{equation*}
It then follows from the second identity of \eqref{mes:n} that $\alpha_0\leq 1/2$, and therefore $\alpha_0=0$. Moreover, it follows from the first identity of \eqref{mes:n} that there exist as most one point $x_0\in\rn$ such that $\alpha_{x_0}=1$. In particular $x_0\neq 0$ since $\alpha_0=0$. It then follows from Lions's second concentration compactness lemma \cites{lions1,lions2} (see also Struwe \cite{st} for an exposition in book form) that, up to a subsequence, there exists $u_\infty\in \dundeuxc$, $x_0\in \rn\setminus\{0\}$ and $C\in \{0,1\}$ such that $u_k\rightharpoonup u_\infty$ weakly in $\dundeuxc$ and 
\begin{equation*}
\lim_{k\to +\infty}\frac{|u_k|^{\crits}}{|x|^s}{\bf 1}_{\cone}\, dx=\frac{|u_\infty|^{\crits}}{|x|^s}{\bf 1}_{\cone}\, dx +C\delta_{x_0}\hbox{ in the sense of measures}
\end{equation*}
In particular, due to \eqref{mes:n} and  \eqref{cpct}, we have that
\begin{equation*}
1=\lim_{k\to +\infty}\int_{\cone}\frac{|u_k|^{\crits}}{|x|^s}\, dx=\int_{\cone}\frac{|u_\infty|^{\crits}}{|x|^s}\, dx +C.
\end{equation*}
Since $C\in \{0,1\}$, the claims in (\ref{case:1}) and (\ref{case:2}) follow.

We now assume that $u_\infty\not\equiv 0$, and we claim that $\lim_{k\to +\infty}u_k= u_\infty$ strongly in $\dundeuxc$ and that $u_\infty$ is an extremal for $\mu_{\gamma,s}(\cone)$. 

\smallskip\noindent 
Indeed, it  follows from \eqref{case:1} that $\int_{\cone}\frac{|u_\infty|^{\crits}}{|x|^s}\, dx=1$, hence 
\begin{equation*}
\mu_{\gamma,s}(\cone)\leq \int_{\cone}\left(|\nabla u_\infty|^2-\frac{\gamma}{|x|^2}u_\infty^2\right)\, dx. 
\end{equation*}
Moreover, since $u_k\rightharpoonup u_\infty$ weakly as $k\to +\infty$, we have that
$$\int_{\cone}\left(|\nabla u_\infty|^2-\frac{\gamma}{|x|^2}u_\infty^2\right)\, dx\leq \liminf_{k\to +\infty}\int_{\cone}\left(|\nabla u_k|^2-\frac{\gamma}{|x|^2}u_k^2\right)\, dx=\mu_{\gamma,s}(\cone).$$
Therefore, equality holds in this latest inequality, $u_\infty$ is an extremal for $\mu_{\gamma,s}(\cone)$ and boundedness yields the weak convergence of $(u_k)$ to $u_\infty$ in $\dundeuxc$. This proves the claim.

We now assume that $u_\infty\equiv 0$ and show that as $k\to +\infty$, 
\begin{equation}\label{step4}
s=0\;,\; \lim_{k\to +\infty}\int_{\cone}\frac{u_k^2}{|x|^2}\, dx=0\hbox{\,  and \, }|\nabla u_k|^2\, dx\rightharpoonup \mu_{\gamma,s}(\cone)\delta_{x_0} \, \hbox{in the sense of measures.}
\end{equation}
\smallskip\noindent
Indeed, since $u_k\rightharpoonup u_\infty\equiv 0$ weakly in $\dundeuxc$ as $k\to +\infty$, then for any $1\leq q<\crit:=\frac{2n}{n-2}$, $u_k\to 0$ strongly in $L^q_{loc}(\cone)$ when $k\to +\infty$. Assume by contradiction that $s>0$,  then $\crits<\crit$ and therefore, since $x_0\neq 0$, we have that 
$$\lim_{k\to +\infty}\int_{B_\delta(x_0)\cap\cone}\frac{|u_k|^{\crits}}{|x|^s}\, dx=0,$$
for $\delta>0$ small enough, contradicting \eqref{case:2}. Therefore $s=0$ and the first part of the claim is proved.

\medskip\noindent 
For the rest, we let $f\in C^\infty(\rn)$ be such that $f(x)=0$ for $x\in B_\delta(x_0)$, $f(x)=1$ for $x\in \rn\setminus B_{2\delta}(x_0)$ and $0\leq f\leq 1$. We define $\varphi:=1-f^2$ and $\psi:=f\sqrt{2-f^2}$. Clearly $\varphi,\psi\in C^\infty(\rn)$ and $\varphi^2+\psi^2=1$. Note that 
$$\mu_{\gamma,s}(\cone)\left(\int_{\cone}|\varphi u_k|^{\crit}\, dx\right)^{\frac{2}{\crit}}\leq \int_{\cone}\left(|\nabla (\varphi u_k)|^2-\frac{\gamma}{|x|^2}(\varphi u_k)^2\right)\, dx.$$
Integrating by parts, using \eqref{case:2}, the fact that $u_k\to 0$ strongly in $L^2_{loc}(\rn)$ as $k\to +\infty$, and that $\varphi^2=1-\psi^2$, we get that as $k\to +\infty$, 
\begin{eqnarray*}
\mu_{\gamma,s}(\cone)\left(|\varphi(x_0)|^{\crit}+o(1)\right)^{\frac{2}{\crit}}\leq \int_{\cone}\varphi^2\left(|\nabla  u_k|^2-\frac{\gamma}{|x|^2}u_k^2\right)\, dx+O\left(\int_{\hbox{Supp }\varphi\Delta\varphi}u_k^2\, dx\right)
\end{eqnarray*}
and 
\begin{eqnarray*}
 \mu_{\gamma,s}(\cone)+o(1)\leq\int_{\cone}\left(|\nabla  u_k|^2-\frac{\gamma}{|x|^2}u_k^2\right)\, dx -\int_{\cone}\psi^2\left(|\nabla  u_k|^2-\frac{\gamma}{|x|^2}u_k^2\right)\, dx+o(1).
\end{eqnarray*}
 Using again \eqref{mes:l}, we obtain
$$\int_{\cone}\psi^2\left(|\nabla  u_k|^2-\frac{\gamma}{|x|^2}u_k^2\right)\, dx\leq o(1)\qquad \hbox{as $k\to +\infty$.}$$
 Integrating again by parts and using the strong local convergence to $0$, we get that
$$\int_{\cone}\left(|\nabla  (\psi u_k)|^2-\frac{\gamma}{|x|^2}(\psi u_k)^2\right)\, dx\leq o(1) \qquad \hbox{as $k\to +\infty$.}$$
The coercivity 
then yields that $\lim_{k\to +\infty}\Vert \nabla(\psi u_k)\Vert_2=0$, and the Hardy inequality yields the  convergence of $|x|^{-1}(\psi u_k)_k$ to $0$ in $L^2(\cone)$. Therefore,
$$\lim_{k\to +\infty}\int_{(B_{2\delta}(x_0))^c\cap\cone}\frac{u_k^2}{|x|^2}\, dx=0.$$
Taking $\delta>0$ small enough and combining this result with the strong convergence of $(u_k)_k$ in $L^2_{loc}$ around $x_0\neq 0$ yields 
$$\lim_{k\to +\infty}\int_{\cone}\frac{u_k^2}{|x|^2}\, dx=0, $$
which once combined with the fact that  $\lim_{k\to +\infty}\Vert \nabla(\psi u_k)\Vert_2=0$ and \eqref{mes:l}, yields the third part of the claim. 

We now show that if $u_\infty\equiv 0$, then $s=0$ and
$$\mu_{\gamma,s}(\cone)=\mu_{0,0}(\rn)=\frac{1}{K(n,2)^2}.$$
\noindent
Indeed, since $u_k\in \dundeuxc\subset\dundeuxr$, we have that 
$$\mu_{0,0}(\rn)\left(\int_{\rn}|u_k|^{\crit}\, dx\right)^{\frac{2}{\crit}}\leq \int_{\rn}|\nabla u_k|^2\, dx.$$
It then follows from (\ref{step4}), \eqref{mes:l} and \eqref{mes:n} that $\mu_{0,0}(\rn)\leq \mu_{\gamma,s}(\cone)$. Conversely, Remark \ref{used}   
yields that  $\mu_{\gamma,s}(\cone)\leq \mu_{0,0}(\rn)=K(n,2)^{-1}$. These two inequalities prove the claim.

Note now that if $s=0$, $\gamma>0$ and $n\geq 4$, then necessarily
\begin{equation}
\mu_{\gamma,s}(\cone)<\mu_{0,0}(\rn)=\frac{1}{K(n,2)^2}.
\end{equation}
\noindent
Indeed, consider the family $u_\eps$ as in Remark \ref{used}.  
Well known computations by Aubin \cite{aubin} yield  
$$J_{\gamma,s}^{\cone}(u_\eps)=K(n,2)^{-2}-\gamma|x_0|^{-2}c\theta_\eps+o(\theta_\eps)\hbox{\, as $\eps\to 0$,}$$
where $c>0$, $\theta_\eps=\eps^2$ if $n\geq 5$ and $\theta_\eps=\eps^2\ln\eps^{-1}$ if $n=4$. It follows that if $\gamma>0$ and $n\geq 4$, then $\mu_{\gamma, s}(\cone)<K(n,2)^{-1}$. This proves the claim. 

\medskip\noindent
As noted in Remark \ref{used},   it is easy to see that 
if $s=0$ and $\gamma\leq 0$, then 
\begin{equation}
\mu_{\gamma,s}(\cone)=\mu_{0,0}(\rn)=\frac{1}{K(n,2)^2}.
\end{equation}
Moreover, if there are extremals then $\gamma=0$.

\medskip\noindent
We now show that in this case,  
there are extremals iff  there exists $z\in \rn$ such that $(1+|x-z|^2)^{1-n/2}\in \dundeuxc$ (in particular, if $\overline{\cone}=\rn$).\par

\smallskip\noindent Indeed, the potential extremals for $\mu_{0,0}(\cone)$ are extremals for $\mu_{0,0}(\rn)$, and therefore of the form $x\mapsto a(b+|x-z_0|^2)^{1-n/2}$ for some $a\neq 0$ and $b>0$ (see Aubin \cite{aubin} or Talenti \cite{Tal}). Using the homothetic invariance of the cone, we  get that there is an extremal of the form $x\mapsto (1+|x-z|^2)^{1-n/2}$ for some $z\in\rn$. Since an extremal has support in $\overline{\cone}$, we then get that $\overline{\cone}=\rn$. This proves the claim.

Finally, assume that $s=0$ and that there exists $z\in \rn$ such that $x\mapsto (1+|x-z|^2)^{1-n/2}\in \dundeuxc$. Then $\mu_{\gamma,0}(\cone)<\frac{1}{K(n,2)^2}$ for all $\gamma>0$. For that it 
 suffices to consider $U(x):=(1+|x-z|^2)^{1-n/2}$ for all $x\in\rn$, and to note that $J_{\gamma,0}^{\cone}(U)=J_{\gamma,0}^{\rn}(U)<J_{0,0}^{\rn}(U)=K(n,2)^{-1}$. \\
This ends the proof of Theorem \ref{th:ext:cone} and Corollaries  \ref{coro:cone:1},  \ref{coro:cone:2}.

\part{When $0$ is an interior singularity for the operator $L_\gamma$}

\section{Analytic conditions for the existence of extremals}

We  now consider the quantity
\begin{equation}\label{def:mu:gamma:s}
\mu_{\gamma, s, \lambda}(\Omega):=\inf\left\{\frac{\int_{\Omega} |\nabla u|^2dx-\gamma \int_{\Omega}\frac{u^2}{|x|^2}dx-\lambda \int_\Omega u^2dx}{(\int_{\Omega}\frac{u^{2^*}}{|x|^s}dx)^{\frac{2}{2^*}}};\,  u\in D^{1.2}(\Omega)\setminus\{0\}\right\},
\end{equation}
in such a way that $\mu_{\gamma, s, 0}(\Omega)=\mu_{\gamma, s}(\Omega)$.  The following proposition is straightforward. 

\begin{proposition}\label{prop:pptes:inf} Let $\Omega$ be a bounded smooth domain such that $0\in \Omega$ and assume $0\leq s\leq 2$.  
 If $\gamma <\frac{{(n-2)}^2}{4}$, then  
 \begin{equation}
 \sup_{\lambda\in\rr}\mu_{\gamma, s, \lambda}(\Omega)= \mu_{\gamma, s}(\rn).
 \end{equation}

\end{proposition}
Note that if $0\in \Omega$, then $\mu_{\gamma, s, 0}(\Omega)=\mu_{\gamma, s}(\rn)$, which then imply in view of the above proposition that $\mu_{\gamma, s, \lambda}(\Omega)=\mu_{\gamma, s}(\rn)$ for all $\lambda \leq 0$. These are the cases, where there are no extremals for $\mu_{\gamma, s, \lambda}(\Omega)$. Now, we consider the case when $\mu_{\gamma, s, \lambda}(\Omega)<\mu_{\gamma, s}(\rn)$. The following proposition is standard but crucial to what follows. 

We shall denote by $\lambda_1(L_\gamma):=\lambda_1(L_\gamma, \Omega)$ the first eigenvalue of the operator $L_\gamma$, that is 
\[
\lambda_1(L_\gamma)=\inf\left\{\frac{\int_{\Omega} |\nabla u|^2dx-\gamma \int_{\Omega}\frac{u^2}{|x|^2}dx}{\int_\Omega u^2dx};\,  u\in D^{1.2}(\Omega)\setminus\{0\}\right\}.
\]

\begin{proposition} \label{tool.bis} Let $\Omega$ be a bounded domain in $\rn$ ($n\geq 3$)  such that $0\in \Omega$, and assume that $\gamma<\frac{(n-2)^2}{4}$ and $0\leq s\leq 2$. 
 If $\mu_{\gamma, s, \lambda}(\Omega)<\mu_{\gamma,s}(\rn)$ for some $\lambda \geq 0$, then there are extremals for  $\mu_{\gamma, s, \lambda}(\Omega)$ in $H^1_0(\Omega)$. 

If in addition $0<\lambda <\lambda_1(L_\gamma)$ and $s <2$,  
then $\mu_{\gamma, s, \lambda}(\Omega)>0$, 
 and there exists a positive solution to the equation 
\begin{eqnarray} \label{pos}
\left\{ \begin{array}{llll}
-\Delta u-\gamma \frac{u}{|x|^2}-\lambda u&=&\frac{u^{2^*(s)-1}}{|x|^s} \ \ &\text{on } \Omega\\
\hfill u&>&0 &\text{on }\partial \Omega\\
\hfill u&=&0 &\text{on }\partial \Omega.
\end{array} \right.
\end{eqnarray} 

\end{proposition}
\begin{proof}  
Let $(u_i)\in H^1_0(\Omega)\setminus\{0\}$ be a minimizing sequence for $\mu_{\gamma,s}(\Omega)$, that is $J^\Omega_{\gamma,s}(u_i)=\mu_{\gamma,s}(\Omega)+o(1)$ as $i\to +\infty$. Up to multiplying by a constant, we assume that
\begin{eqnarray}
&&\int_\Omega \frac{|u_i|^{\crits}}{|x|^s}\, dx=1\label{eq:pf:min:1}\\
&&\int_\Omega\left(|\nabla u_i|^2-\gamma\frac{u_i^2}{|x|^2}- \lambda  u_i^2\right)\, dx=\mu_{\gamma,s}(\Omega)+o(1)\hbox{ as }i\to +\infty.\label{eq:pf:min:2}
\end{eqnarray}
\medskip\noindent We claim that $(u_i)_i$ is bounded in $H^1_0(\Omega)$. 
Indeed, \eqref{eq:pf:min:1} clearly yields that 
\begin{equation}\label{bnd:u:l2}
\hbox{$\int_\Omega u_i^2\, dx\leq C<+\infty$ for all $i$.}
\end{equation}
Since $\gamma<\frac{(n-2)^2}{4}$, the Hardy inequality
combined with \eqref{eq:pf:min:2} yield the boundedness of $(u_i)_i$ in $H^1_0(\Omega)$. 
It follows that there exists $u\in H^1_0(\Omega)$ such that, up to a subsequence, $(u_i)$ goes to $u$ weakly in $H^1_0(\Omega)$ and strongly in $L^2(\Omega)$ as $i\to +\infty$.

We now show that $\int_\Omega \frac{|u|^{\crits}}{|x|^s}\, dx=1$. For that,  define $\theta_i:=u_i-u\in H^1_0(\Omega)$. In particular, $\theta_i$ goes to $0$ weakly in $H^1_0(\Omega)$ and strongly in $L^2(\Omega)$ as $i\to +\infty$. Hence, 
\begin{equation}\label{eq:pf:min:3}
1=\int_\Omega \frac{|u_i|^{\crits}}{|x|^s}\, dx=\int_\Omega \frac{|u|^{\crits}}{|x|^s}\, dx+\int_\Omega \frac{|\theta_i|^{\crits}}{|x|^s}\, dx+o(1)
\end{equation}
and
\begin{equation}\label{eq:pf:min:4}
\mu_{\gamma,s, \lambda}(\Omega)= \int_\Omega\left(|\nabla u|^2-\gamma\frac{u^2}{|x|^2}-\lambda u^2\right)\, dx+\int_\Omega\left(|\nabla \theta_i|^2-\gamma\frac{\theta_i^2}{|x|^2}\right)\, dx+o(1).
\end{equation}
From the definition of $\mu_{\gamma,s, \lambda}(\Omega)$, and the fact that $\mu_{\gamma,s}(\Omega)=\mu_{\gamma,s}(\rn)$, we have 
\begin{equation}
\mu_{\gamma,s, \lambda}(\Omega)\left(\int_\Omega \frac{|u|^{\crits}}{|x|^s}\, dx\right)^{\frac{2}{\crits}}\leq \int_\Omega\left(|\nabla u|^2-\gamma\frac{u^2}{|x|^2}-\lambda u^2\right)\, dx,\label{eq:pf:min:5}
\end{equation}
and 
\begin{equation}
\mu_{\gamma,s}(\rn)\left(\int_\Omega \frac{|\theta_i|^{\crits}}{|x|^s}\, dx\right)^{\frac{2}{\crits}}\leq \int_\Omega\left(|\nabla \theta_i|^2-\gamma\frac{\theta_i^2}{|x|^2}\right)\, dx+o(1). \label{eq:pf:min:6}
\end{equation}
Summing these two inequalities and using \eqref{eq:pf:min:3} and \eqref{eq:pf:min:4} and passing to the limit as $i\to +\infty$ yields
\begin{eqnarray*}
\mu_{\gamma,s, \lambda}(\Omega)\left(1-\left(\int_\Omega \frac{|u|^{\crits}}{|x|^s}\, dx\right)^{\frac{2}{\crits}}\right)&\geq& \mu_{\gamma,s}(\rn)\left(1-\int_\Omega \frac{|u|^{\crits}}{|x|^s}\, dx\right)^{\frac{2}{\crits}}.
\end{eqnarray*}
Since $\mu_{\gamma,s, \lambda}(\Omega)<\mu_{\gamma,s}(\rnp)$,
we finally conclude that $\int_\Omega \frac{|u|^{\crits}}{|x|^s}\, dx=1$. \\
 It remains to show that $u$ is an extremal for $\mu_{\gamma,s, \lambda}(\Omega)$. For that, note that since $\int_\Omega \frac{|u|^{\crits}}{|x|^s}\, dx=1$, the definition of $\mu_{\gamma,s, \lambda}(\Omega)$ 
yields that 
$$ \int_\Omega\left(|\nabla u|^2-\gamma\frac{u^2}{|x|^2}-\lambda u^2\right)\, dx\geq \mu_{\gamma,s, \lambda}(\Omega).$$
 The second term in the right-hand-side of \eqref{eq:pf:min:4} is nonnegative due to \eqref{eq:pf:min:6}. Therefore, we get that $\int_\Omega\left(|\nabla u|^2-\gamma\frac{u^2}{|x|^2}-\lambda u^2\right)\, dx= \mu_{\gamma,s, \gamma}(\Omega)$. This proves the claim and ends the proof of the first part of Proposition \ref{tool.bis}.

Now assume that $0<\lambda <\lambda_1:=\lambda_1(L_\gamma)$, then 
for all $u\in H^1_0(\Omega)\setminus\{0\}$, 
\begin{eqnarray*}
J^\Omega_{\gamma,s}(u)=\frac{\int_{\Omega} (|\nabla u|^2-\gamma \frac{u^2}{|x|^2}-\lambda u^2)dx}{(\int_{\Omega}\frac{u^{\crits}}{|x|^s}dx)^{\frac{2}{\crits}}}
&\geq& 
\left(1-\frac{\lambda}{\lambda_1(L_\gamma)}\right)\frac{\int_\Omega(|\nabla u|^2-\gamma \frac{u^2}{|x|^2})\, dx}{\left(\int_\Omega\frac{|u|^{\crits}}{|x|^s}\, dx\right)^{\frac{2}{\crits}}}\\
&\geq&  \left(1-\frac{\lambda}{\lambda_1(L_\gamma)}\right)\left(1-\frac{4\gamma}{(n-2)^4}\right)\mu_{0,s}(\Omega).
\end{eqnarray*}
 Therefore $\mu_{\gamma,s, \lambda}(\Omega)>0$.
 \end{proof}

\section{Existence of extremals when either $s>0$ or $\{s=0$ and $\gamma \geq 0$\}}

In this section, we sketch the ideas behind the following result. Brezis-Nirenberg \cite{bn} pioneered this line of inquiry when $\gamma =0, s=0$ and $n\geq 4$. Janelli \cite{Jan} did the case where $0<\gamma < \frac{(n-2)^2}{4}-1$ and $s=0$, while Ruiz-Willem \cite{RW} considered the situation when $\gamma <0$. The remaining cases were dealt with in Ghoussoub-Robert \cite{gr5}.

\begin{theorem} Let $\Omega$ be a smooth bounded domain of $\rn$ such that $0\in\Omega$. Fix $\gamma<\frac{(n-2)^2}{4}$, $\lambda <\lambda_1(L_\gamma)$ and assume that either $s>0$ or $\left\{ s=0\hbox{ and }\gamma\geq 0\right\}$. 

\begin{enumerate}
\item If $\gamma\leq \frac{(n-2)^2}{4}-1$, then there are extremals for $\mu_{s,\gamma, \lambda}(\Omega)$ if and only if $\lambda>0$.
\item If $\gamma> \frac{(n-2)^2}{4}-1$, then there are extremals for $\mu_{s,\gamma, \lambda}(\Omega)$ if $m_{\gamma, -\lambda}(\Omega)>0$. 
\end{enumerate}
\end{theorem}

\begin{proof}  We construct a minimizing sequence $u_\epsilon$ in $ H^1_0(\Omega)\setminus\{0\}$ for the functional 
$$J_{\gamma,s, \lambda}(u):=\frac{\int_\Omega\left(|\nabla u|^2-\gamma\frac{u^2}{|x|^2}-\lambda u^2\right)\, dx}{\left(\int_\Omega\frac{|u|^{\crits}}{|x|^s}\, dx\right)^{\frac{2}{\crits}}},$$
in such a way that $\mu_{s,\gamma, \lambda}(\Omega)<\mu_{s,\gamma}(\rn)$.

\medskip\noindent If either $s>0$ or $\gamma\geq 0$, then the infimum $\mu_{\gamma,s}(\rn)$ is achieved by 
the function 
$$U(x):=\frac{1}{\left(|x|^{\frac{(2-s)\bm}{n-2}}+|x|^{\frac{(2-s)\bp}{n-2}}\right)^{\frac{n-2}{2-s}}}
\hbox{ for }x\in\rn\setminus\{0\}.$$
Define the test-functions
$$u_\eps(x):=\eta(x) \eps^{-\frac{n-2}{2}}U(\eps^{-1}x)\hbox{ for all }x\in \Omega,$$
where $\eta\in C^\infty_c(\Omega)$ is such that $\eta(x)=1$ around $0\in\Omega$. A straightforward computation yields
\begin{equation}
J_{\gamma,s, \lambda }(u_\eps)=\mu_{\gamma,s}(\rn)+o(1) \quad \hbox{ as }\eps\to 0.
\end{equation}
Going further in the expansion, one can show the following:

\medskip\noindent{\bf Claim 1:} If $\gamma<\frac{(n-2)^2}{4}-1$, then 
\begin{equation}
J_{\gamma,s, \lambda}(u_\eps)=\mu_{\gamma,s}(\rn)-\lambda C\eps^2+o(\eps^2) \quad \hbox{ as } \eps\to 0,
\end{equation}
where 
\begin{equation}
C:=\frac{\int_{\rn}U^2\, dx}{\left(\int_{\rn}\frac{U^{\crits}}{|x|^s}\, dx\right)^{\frac{2}{\crits}}}<+\infty.
\end{equation}
Note that $C<+\infty$ if and only if $\gamma<\frac{(n-2)^2}{4}-1$, which happens if and only if $\beta_+(\gamma)-\beta_-(\gamma) > 2$. This explains the obstruction on the dimension in this situation, since the $L^2-$concentration allows to overlook the role of the cut-off function. 

Pushing the expansion to the limit, we have the following 

\medskip\noindent{\bf Claim 2:} If $\gamma=\frac{(n-2)^2}{4}-1$, then 
\begin{equation}
J_{\gamma,s, \lambda}(u_\eps)=\mu_{\gamma,s}(\rn)-\lambda C'\eps^2\ln(\eps^{-1})+O(\eps^2)\hbox{ as }\eps\to 0,
\end{equation}
where $C'$ is a positive consatnt.

\medskip\noindent When $\gamma>\frac{(n-2)^2}{4}-1$, the above test functions do not suffice, and one needs more global test functions . We therefore  let $H\in C^\infty(\overline{\Omega}\setminus\{0\})$ as in Proposition \ref{interior.mass}. Up to multiplying by a constant, we assume that $C_1=1$. We let $\ell \in H^1_0(\Omega)\cap C^0(\Omega)$ be such that
$$H(x)=\frac{\eta(x)}{|x|^{\bp}}+\ell(x)\hbox{ for all }x\in\Omega.$$
Note that $\ell (x)= \frac{m_{\gamma, -\lambda}(\Omega)}{|x|^{\bm}} +o(\frac{1}{|x|^{\bm}})$, where $m_{\gamma,-\lambda}(\Omega)$ is the Hardy-interior mass. The test-functions can be taken in this case to be
\begin{equation}
v_\eps(x):=u_\eps(x) + \eps^{\frac{\bp-\bm}{2}}\ell (x) \quad \hbox{for all $x\in \Omega$.}
\end{equation} 
One can then show the following.\\
 
\noindent{\bf Claim 3:} If $\frac{(n-2)^2}{4}-1<\gamma <\frac{(n-2)^2}{4}$, then 
 \begin{equation}
 J_{\gamma,s, \lambda}(u_\eps)=\mu_{\gamma,s}(\rn)-m_{\gamma, -\lambda}(\Omega)\eps^{\bp-\bm}+o\left(\eps^{\bp-\bm}\right) \quad \hbox{as $\eps\to 0$}.
 \end{equation}
 Note that in this case $\beta_+(\gamma)-\beta_-(\gamma) < 2$.
\end{proof}

\section{Existence of extremals when $s=0$ and $\gamma <0$}

Recall from the introduction that $R_{\gamma,\lambda}(x_0)$ is the Robin function at $x_0$, that is the value at $x_0$ of the regular part of the Green's function of $-\Delta-\gamma|x|^{-2}-\lambda$ at $x_0$.
We sketch the proof of the remaining cases.

\begin{theorem} Let $\Omega$ be a smooth bounded domain of $\rn$ such that $0\in\Omega$. Fix $\gamma<\frac{(n-2)^2}{4}$, $\lambda <\lambda_1(L_\gamma)$ and assume that  $s=0$ and $\gamma<0$.

\begin{enumerate}
\item If $n\geq 4$, then there are extremals for $\mu_{s,\gamma, \lambda}(\Omega)$ iff $\lambda>\frac{|\gamma|}{\max_{x\in \Omega}|x|^2}$.
\item If $n=3$, then there are extremals for $\mu_{s,\gamma, \lambda}(\Omega)$ provided there exists $x_0$ in $\Omega\setminus\{0\}$ such that $R_{\gamma,-\lambda}(x_0)>0$. 

\end{enumerate}

\end{theorem}

\begin{proof} By Theorem \ref{th:ext:cone}, this is the case when $\mu_{0, 0}(\rn)=\mu_{\gamma, 0}(\rnp)$. 
Consider the following known extremal for $\mu_{0,0}(\rn)$, 
$$U(x):=\frac{1}{\left(1+|x|^{2}\right)^{\frac{n-2}{2}}} \hbox{ for }x\in\rn.$$
Fix $x_0\in\Omega$, $x_0\neq 0$, and define the test-function
$$u_\eps(x):=\eta(x) \eps^{-\frac{n-2}{2}}U(\eps^{-1}(x-x_0))\hbox{ for all }x\in \Omega,$$
where $\eta\in C^\infty_c(\Omega)$ is such that $\eta(x)=1$ around $x_0\in\Omega$. A straightforward computation yields
$$J_{\gamma,0, \lambda}(u_\eps)=\mu_{0,0}(\rn)+o(1)\hbox{ as }\eps\to 0,$$
which yields that $\mu_{\gamma, 0, \lambda}(\Omega)\leq \mu_{0,0}(\rn)$.

\medskip\noindent Note now that if $\lambda\leq \frac{|\gamma|}{\max_{x\in \Omega}|x|^2}$, then $\lambda+\frac{\gamma}{|x|^2}\leq 0$ for all $x\in\Omega$, and therefore $\mu_{\gamma, 0, \lambda}(\Omega)\geq \mu_{0,0}(\Omega)$. We therefore have equality, and there is no extremal  for $\mu_{\gamma, 0, \lambda}(\Omega)$ since the extremals on $\rn$ are rescaled and translated versions of $U$.

On the other hand, one can argue as in Aubin \cite{aubin} and prove the following

\medskip\noindent {\bf Claim 1:}   If  $x_0\in\Omega\setminus \{0\}$ is such $\lambda+\frac{\gamma}{|x_0|^2}>0$ and $n\geq 5$, then 
\begin{equation}
J_{\gamma,0, \lambda}(u_\eps)=\mu_{\gamma,0}(\rn)-\left(\lambda+\frac{\gamma}{|x_0|^2} \right)C\eps^2+o(\eps^2)\hbox{ as }\eps\to 0,
\end{equation}
where $$C:=\frac{\int_{\rn}U^2\, dx}{\left(\int_{\rn}U^{\crit}\, dx\right)^{\frac{2}{\crit}}}<+\infty.$$
Note that $C<+\infty$ iff $n\geq 4$, in which case the $L^2-$concentration again allows to overlook the cut-off function. 

For $n=4$ one needs to push the expansion further.

\medskip\noindent {\bf Claim 2:}   If  $x_0\in\Omega\setminus \{0\}$ is such $\lambda+\frac{\gamma}{|x_0|^2}>0$, and $n =4$, then 
\begin{equation}
J_{\gamma,0, \lambda}(u_\eps)=\mu_{\gamma, 0}(\rn)-\left(\lambda+\frac{\gamma}{|x_0|^2} \right) C'\eps^2\ln(\eps^{-1})+O(\eps^2) \quad \hbox{ as }\eps\to 0,
\end{equation}
where $C'$ is a positive constant.

\medskip\noindent In order to deal with the case $n=3$, global test-functions are again required. We let $G_{x_0}\in C^\infty(\overline{\Omega}\setminus\{0\})$ be the Green's function of $-\Delta-\lambda-\gamma|x|^{-2}$ at $x_0$. Up to multiplying by a constant, we may assume that $C_1=1$. Let $\beta\in H^1_0(\Omega)\cap C^0(\Omega)$ be such that
$$G_{x_0}(x)=\omega_2^{-1}\left(\frac{\eta(x)}{|x-x_0|}+\beta(x)\right)\hbox{ for all }x\in\Omega\setminus\{x_0\}.$$
Note that $\beta(x_0)=R_{\gamma,\lambda}(x_0)$ is the Robin function at $x_0$. 

Define now the test-function
\begin{equation}
u_\eps(x):=\eta(x)\left(\frac{\eps}{\eps+|x|^{2}}\right)^{\frac{1}{2}}+\eps^{\frac{1}{2}} \beta(x)\quad \hbox{
for all $x\in \Omega$. }
\end{equation}
One can then show the following 

\medskip\noindent {\bf Claim 3:}   If  $x_0\in\Omega\setminus \{0\}$ is such $\lambda+\frac{\gamma}{|x_0|^2}>0$ and $n =3$, then 
\begin{equation}
J_{\gamma,0, \lambda}(u_\eps)=\mu_{\gamma,0}(\rn)-R_{\gamma,\lambda}(x_0)\eps +o\left(\eps\right) \quad \hbox{ as }\eps\to 0.
\end{equation}
\end{proof}

\part{When $0$ is a boundary singularity for the operator $L_\gamma$}

\section{\, Analytic conditions for the existence of extremals when $0\in \partial \Omega$}

As mentioned in the introduction, the case when the singularity $0\in \partial \Omega$ is  more intricate as far as  the operator $-\Delta -\frac{\gamma}{|x|^2}$ is concerned. This is already apparent in the following linear situation.

\begin{proposition}\label{prop:gamma} $\gamma_H$ satisfies the following properties on the class of bounded smooth domains $\Omega$ in $\rn$ such that $0\in \partial \Omega$:
\begin{enumerate}

\item If $0\in \partial \Omega$, then $\frac{(n-2)^2}{4}<\gamma_H(\Omega)\leq\frac{n^2}{4}$. 
\item $\gamma_H(\rnp)=\frac{n^2}{4}$, and $\gamma_H(\Omega)=\frac{n^2}{4}$ for every $\Omega$ such that $0\in \partial \Omega$ and $\Omega \subset \rnp$.
\item We have $\inf\{\gamma_H(\Omega);\,  0\in\partial\Omega\}=\frac{(n-2)^2}{4}$.

\item For every $\epsilon>0$, there exists a smooth domain $\rnp\subset \Omega_\epsilon \subset \rn$ such that $0\in \partial\Omega_\epsilon$ and $\frac{n^2}{4}-\epsilon \leq \gamma_H(\Omega_\epsilon) < \frac{n^2}{4}$. 

\end{enumerate}

\end{proposition}
The above mentioned properties of $\gamma_H$ were noted in \cite{gm} and \cite{gr4}. We sketch the proofs. We have already noted in section 1,  that $\gamma_H(\rn)=\frac{(n-2)^2}{4}$, while equation (\ref{rnpk}) yields that $\gamma_H(\rnp)=\frac{n^2}{4}$. It is also easy to see that if $B_r$ is a ball of radius $r$ such that $0\in \partial B_r$, then we also have $\gamma_H (B_r)=\gamma_H(\rnp)=\frac{n^2}{4}$. If now $\partial \Omega$ is smooth at $0\in \partial \Omega$, we can always find such a ball with $B_r\subset \Omega$, from which follows that $\gamma_H (\Omega)\geq \gamma_H (B_r)=\frac{n^2}{4}$.

To prove 3), one first shows that $\gamma_H(\rn)$ can be approached by the following nonsmooth conical domains. Let $\Omega_0$ be a bounded domain of $\rn$ such that $0\in\Omega_0$ (i.e., it is not on the boundary). Given $\delta>0$, define
$$\Omega_\delta:=\Omega_0\setminus \{(x_1,x')/\, x_1\leq 0\hbox{ and }|x'|\leq\delta\}.$$
For $\delta>0$ small enough, $0\in\partial\Omega$, and one can show that 
$
\lim_{\delta\to 0}\gamma_H(\Omega_\delta)=\frac{(n-2)^2}{4}.
$
Note that this works for $n\geq 4$. A different construction is needed for $n=3$. 
Now to check the infimum for smooth domains, note that for each $\delta>0$ small, there exists $\Omega_\delta'$ a smooth bounded domain of $\rn$ such that $\Omega_\delta\subset\Omega_\delta'$ and $0\in \Omega_\delta'$. Since $\Omega\mapsto\gamma_H(\Omega)$ is nonincreasing, we have that $\gamma_H(\rn)\leq \gamma_H(\Omega_\delta')\leq \gamma_H(\Omega_\delta)$ and therefore
$\limsup\limits_{\delta\to 0}\gamma_H(\Omega_\delta')=\frac{(n-2)^2}{4}.$

For 4) let $\varphi\in C^\infty(\rr^{n-1})$ be such that $0\leq\varphi\leq 1$, $\varphi(0)=0$, and $\varphi(x')=1$ for all $x'\in\rr^{n-1}$ such that $|x'|\geq 1$. For $t\geq 0$, define $\Phi_t(x_1,x'):=(x_1-t\varphi(x'), x')$ for all $(x_1,x')\in\rn$. Set $\tilde{\Omega}_t:=\Phi_t(\rnp)$. Now note that  $\lim_{\eps\to 0}\gamma_H(\tilde{\Omega}_t)=\gamma_H(\rnp)=\frac{n^2}{4}$. Since $\varphi\geq 0$, we have that $\rnp\subset \tilde{\Omega}_t$ for all $t>0$. It now suffices to take $\Omega_\eps:=\tilde{\Omega}_t$ for $t$ small enough.

As to whether $\gamma_H(\Omega)$ is attained or not, it depends -- in contrast with the case when $0\in \Omega$ -- on whether it is strictly less than  $\frac{n^2}{4}$. It is a particular case of the following general result, which is key to the sequel.

\begin{theorem} \label{tool} Let $\Omega$ be a bounded domain in $\rn$ ($n\geq 3$)  such that $0\in \partial \Omega$, and assume that $\gamma<\frac{n^2}{4}$ and $0\leq s\leq 2$. 
\begin{enumerate}
\item If $\mu_{\gamma, s}(\Omega)<\mu_{\gamma,s}(\rnp)$, then there are extremals for  $\mu_{\gamma, s}(\Omega)$. \\
 In particular,  If $\gamma_H(\Omega)< \frac{n^2}{4}$, then the best constant in the Hardy inequality on $\Omega$ is attained in $H^1_0(\Omega)$.

\item If $\gamma <\gamma_H(\Omega)$ then $\mu_{\gamma, s}(\Omega)>0$, and if also $\mu_{\gamma, s}(\Omega) <\mu_{\gamma,s}(\rnp)$ and $s <2$, then there exists a positive solution to the equation 
\begin{eqnarray} \label{pos:bis}
\left\{ \begin{array}{llll}
-\Delta u-\gamma \frac{u}{|x|^2}&=&\frac{u^{2^*(s)-1}}{|x|^s} \ \ &\text{on } \Omega\\
\hfill u&>&0 &\text{on }\partial \Omega\\
\hfill u&=&0 &\text{on }\partial \Omega.
\end{array} \right.
\end{eqnarray} 
\item If $\gamma_H(\Omega)<\gamma <\frac{n^2}{4}$ then $\mu_{\gamma, s}(\Omega)<0$, and  
 if $s < 2$ then there exists a positive solution to the equation 
\begin{eqnarray} \label{neg}
\left\{ \begin{array}{llll}
-\Delta u-\gamma \frac{u}{|x|^2}&=&-\frac{u^{2^*(s)-1}}{|x|^s} \ \ &\text{on } \Omega\\
\hfill u&>&0 &\text{on }\partial \Omega\\
\hfill u&=&0 &\text{on }\partial \Omega.
\end{array} \right.
\end{eqnarray} 
\end{enumerate}

\end{theorem}
Here again one starts by establishing the following improved inequality on bounded domains. See Ghoussoub-Robert \cite{gr4}. 

\begin{proposition}\label{lem:ineq:eps} Assume $\gamma<\frac{n^2}{4}$ and $s\in [0,2]$. If $\Omega$ is a bounded domain of $\rn$ such that $0\in \partial \Omega$, then for any $\eps>0$, there exists $C_\eps>0$ such that for all $u\in H^1_0(\Omega)$, 
\begin{equation}\label{ineq:sobo:eps:bis}
\left(\int_\Omega \frac{|u|^{\crits}}{|x|^s}\, dx\right)^{\frac{2}{\crits}}\leq (\frac{1}{\mu_{\gamma,s}(\rnp)}+\epsilon)\int_\Omega\left(|\nabla u|^2-\gamma\frac{u^2}{|x|^2}\right)\, dx+C_\eps\int_\Omega u^2\, dx.
\end{equation}
\end{proposition}
\noindent{\it Proof of  Proposition \ref{lem:ineq:eps}:} 
Fix $\epsilon>0$.
We first claim that there exists $\delta_\epsilon>0$ such that for all $u\in C^\infty_c(\Omega\cap B_{\delta_\epsilon}(0))$, 
\begin{equation}\label{eq:step:1}
\left(\int_{\Omega\cap B_{\delta_\epsilon}(0)} \frac{|u|^{\crits}}{|x|^s}\, dx\right)^{\frac{2}{\crits}}\leq (\mu_{\gamma,s}(\rnp)^{-1}+\epsilon)\int_{\Omega\cap B_{\delta_\epsilon}(0)}\left(|\nabla u|^2-\gamma\frac{u^2}{|x|^2}\right)\, dx.
\end{equation}
Indeed, for two open subsets of $\rn$ containing $0$, we may define a diffeomorphism $\varphi: U\to V$ 
such that $\varphi(0)=0$, $\varphi(U\cap\rnp)=\varphi(U)\cap \Omega$ and $\varphi(U\cap\partial\rnp)=\varphi(U)\cap \partial\Omega$. Moreover, we can also assume that $d\varphi_0$ is a linear isometry. In particular 
\begin{equation}\label{eq:pf:1.00}
|\varphi^\star\eucl-\eucl|(x)\leq C|x|\hbox{ and }|\varphi(x)|=|x|\cdot(1+O(|x|))
\end{equation}
for $x\in U$. If now $u\in C^\infty_c(\varphi(B_\delta(0))\cap\Omega)$, then $v:=u\circ\varphi\in C^\infty_c(B_\delta(0)\cap\rnp)$. If $g:=\varphi^{-1\star}\eucl$ denotes the metric induced by $\varphi$, then we get  from  \eqref{eq:pf:1.00}, 
 \begin{eqnarray}\label{eq:pf:2.00}
\qquad \left(\int_\Omega \frac{|u|^{\crits}}{|x|^s}\, dx\right)^{\frac{2}{\crits}}&\leq&\left(\int_{B_\delta(0)\cap\rnp} \frac{|v|^{\crits}}{|\varphi(x)|^s}|\hbox{Jac } \varphi(x)|\, dx\right)^{\frac{2}{\crits}}\nonumber\\
&\leq & (1+C\delta)\left(\int_{B_\delta(0)\cap\rnp} \frac{|v|^{\crits}}{|x|^s}\, dx\right)^{\frac{2}{\crits}}\nonumber\\
&\leq & (1+C\delta)\mu_{\gamma,s}(\rnp)^{-1}\int_{B_\delta(0)\cap\rnp}\left(|\nabla v|^2-\gamma\frac{v^2}{|x|^2} \right)\, dx\nonumber\\
&\leq & \frac{1+C\delta}{\mu_{\gamma,s}(\rnp)}\int_{\varphi(B_\delta(0))\cap\Omega}\left(|\nabla u|_{g}^2-\frac{\gamma u^2}{|\varphi^{-1}(x)|^2}\right)|\hbox{Jac } \varphi^{-1}(x)|\, dx\nonumber\\
&\leq&  (1+C_1\delta)\mu_{\gamma,s}(\rnp)^{-1}\int_{\Omega}\left(|\nabla u|^2-\gamma\frac{u^2}{|x|^2}\right)\, dx\nonumber\\
&& +C_2\delta \int_{\Omega}\left(|\nabla u|^2
+\frac{u^2}{|x|^2}\right)\, dx. 
\end{eqnarray}
We also have that
\begin{eqnarray*}
\int_\Omega\frac{u^2}{|x|^2}\, dx&=&\int_{\varphi(B_\delta(0)\cap\rnp)}\frac{u^2}{|x|^2}\, dx=\int_{B_\delta(0)\cap\rnp}\frac{v^2}{|\varphi(x)|^2}|\hbox{Jac}(\varphi)(x)|\, dx\\
&=&\int_{B_\delta(0)\cap\rnp}\frac{v^2}{|x|^2}(1+O(|x|)\, dx\leq (1+C_1\delta) \int_{\rnp}\frac{v^2}{|x|^2}\, dx
\end{eqnarray*}
and
\begin{eqnarray*}
\int_\Omega|\nabla u|^2\, dx&=&\int_{\varphi(B_\delta(0)\cap\rnp)}|\nabla u|^2\, dx=\int_{B_\delta(0)\cap\rnp}|\nabla v|^2_{\varphi^\star\eucl}|\hbox{Jac}(\varphi)(x)|\, dx\\
&=&\int_{B_\delta(0)\cap\rnp}|\nabla v|^2(1+O(|x|)\, dx\geq (1-C_2\delta) \int_{\rnp}|\nabla v|^2\, dx,
\end{eqnarray*}
where $C_1,C_2>0$ are independent of $\delta$ and $v$. Hardy's inequality \eqref{ineq:hardy:rk} then yields for all $u\in C^\infty_c(\varphi(B_\delta(0)\cap\rnp))$,
\begin{equation}\label{ineq:hardy:opt:1}
\frac{n^2}{4}\int_\Omega\frac{u^2}{|x|^2}\, dx\leq \frac{1+C_1\delta}{1-C_2\delta}\int_\Omega|\nabla u|^2\, dx\leq (1+C_3\delta)\int_{\Omega}|\nabla u|^2\, dx.
\end{equation}
Since $\gamma<\frac{n^2}{4}$, there exists then $c>0$ such that for $\delta>0$ small enough,
\begin{equation*}
c^{-1}\int_{\Omega}|\nabla u|^2\, dx\leq \int_{\Omega}\left(|\nabla u|^2-\gamma\frac{u^2}{|x|^2}\right)\, dx\leq c\int_{\Omega}|\nabla u|^2\, dx
\end{equation*}
for all $u\in C^\infty_c(\varphi(B_\delta(0))\cap\Omega)$. Plugging these latest inequalities in \eqref{eq:pf:2.00} yields \eqref{eq:step:1} by taking $\delta_\epsilon$ small enough. 

\medskip\noindent Consider now $\eta\in C^\infty(\rn)$ such that  $\sqrt{\eta},\sqrt{1-\eta}\in C^2(\rn)$, 
such that $\eta(x)=1$ for $x\in B_{\delta_\eps/2}(0)$ and $\eta(x)=0$ for $x\not\in B_{\delta_\eps}(0)$. We shall use the notation 
\[
\|w\|_{{p,|x|^{-s}}}=\left(\int_\Omega \frac{|w|^p}{|x|^s}\, dx\right)^{1/p}.
\]
For $u\in C^\infty_c(\Omega)$, use  H\"older's inequality to write
\begin{eqnarray*}
\left(\int_\Omega \frac{|u|^{\crits}}{|x|^s}\, dx\right)^{\frac{2}{\crits}}&=&\Vert u^2\Vert_{\frac{\crits}{2},|x|^{-s}}=\Vert \eta u^2+(1-\eta)u^2\Vert_{\frac{\crits}{2},|x|^{-s}}\\
&\leq & \Vert \eta u^2\Vert_{\frac{\crits}{2},|x|^{-s}} +\Vert (1-\eta) u^2\Vert_{\frac{\crits}{2},|x|^{-s}}\\
&\leq & \Vert \sqrt{\eta} u\Vert^2_{\crits,|x|^{-s}} +\Vert \sqrt{1-\eta} u\Vert^2_{\crits,|x|^{-s}}.
\end{eqnarray*} 
Since $\sqrt{\eta}u\in C^\infty_c(B_{\delta_\eps}(0)\cap\Omega)$, it follows from inequality \eqref{eq:step:1} 
 that
\begin{eqnarray}
\left(\int_\Omega \frac{|u|^{\crits}}{|x|^s}\, dx\right)^{\frac{2}{\crits}}
&\leq & (\mu_{\gamma,s}(\rnp)^{-1}+\epsilon)\int_\Omega\left(|\nabla (\sqrt{\eta}u)|^2-\gamma\frac{\eta u^2}{|x|^2}\right)\, dx\nonumber\\
&& +\Vert \sqrt{1-\eta} u\Vert^2_{\crits,|x|^{-s}}\nonumber\\
&\leq & (\mu_{\gamma,s}(\rnp)^{-1}+\epsilon)\int_\Omega\eta \left(|\nabla u|^2-\gamma\frac{u^2}{|x|^2}\right)\, dx + C\int_\Omega u^2\, dx\nonumber\\
&&+\Vert \sqrt{1-\eta} u\Vert^2_{\crits,|x|^{-s}}\label{eq:pf:3.00}
\end{eqnarray} 
\medskip\noindent {\bf Case 1: $s=0$.} Then $\crits=\crit$ and it follows from Sobolev's inequality that
\begin{eqnarray}\label{eq:pf:4.00}
\Vert \sqrt{1-\eta} u\Vert^2_{\crits,|x|^{-s}}&\leq& K(n,2)^2\int_\Omega |\nabla(\sqrt{1-\eta}u)|^2\, dx\nonumber \\
&\leq& K(n,2)^2\int_\Omega(1-\eta)|\nabla u|^2\, dx+C\int_\Omega u^2\, dx,
\end{eqnarray}
where $K(n,2)$ is the optimal Sobolev constant. 
 Since $s=0$, it follows from Remark \ref{used}  
 that $K(n,2)^2\leq \mu_{\gamma,s}(\rnp)^{-1}$, and from  \eqref{eq:pf:4.00} that
\begin{eqnarray}\label{eq:pf:5.00}
\Vert \sqrt{1-\eta} u\Vert^2_{\crits,|x|^{-s}}&\leq&  (\mu_{\gamma,s}(\rnp)^{-1}+\epsilon)\int_\Omega(1-\eta)\left(|\nabla u|^2-\gamma\frac{u^2}{|x|^2}\right)\, dx\nonumber\\
&&+C\int_\Omega u^2\, dx.
\end{eqnarray}
Plugging together \eqref{eq:pf:3.00} and \eqref{eq:pf:5.00} yields \eqref{ineq:sobo:eps:bis} when $s=0$. \\

\medskip\noindent {\bf Case 2: $0<s<2$.} We let $\nu>0$ be a positive number to be fixed later. Since $2<\crits<\crit$, the interpolation inequality yields the existence of $C_\nu>0$ such that
\begin{eqnarray*}
\Vert \sqrt{1-\eta} u\Vert^2_{\crits,|x|^{-s}}&\leq & C \Vert \sqrt{1-\eta} u\Vert^2_{\crits}\\
&\leq& C\left( \nu \Vert \sqrt{1-\eta} u\Vert^2_{\crit}+C_\nu \Vert \sqrt{1-\eta} u\Vert^2_{2}\right)\\
&\leq& C\left( \nu K(n,2)^2 \Vert \nabla(\sqrt{1-\eta} u)\Vert^2_{2}+C_\nu \Vert \sqrt{1-\eta} u\Vert^2_{2}\right).
\end{eqnarray*}
We choose $\nu>0$ such that  $\nu K(n,2)^2 <\mu_{\gamma,s}(\rnp)^{-1}+\epsilon$. Then we get \eqref{eq:pf:5.00} and we conclude \eqref{ineq:sobo:eps:bis} in the case when $2>s>0$ by combining it with \eqref{eq:pf:3.00}.

\medskip\noindent {\bf Case 3: $s=2$.} This is the easiest case, since then 
\[
\Vert \sqrt{1-\eta} u\Vert^2_{\crits,|x|^{-s}}=\int_\Omega \frac{((1-\eta)u)^2}{|x|^2}\, dx \leq C_\delta \int_\Omega u^2\, dx.
\]
This completes the proof of \eqref{ineq:sobo:eps:bis} for all $s\in [0,2]$.
\hfill $\Box$

The following corollary is an easy consequence of the above.

\begin{proposition}\label{prop:pptes:inf:bis} Let $\Omega$ be a bounded smooth domain such that $0\in \partial \Omega$ and assume $0\leq s\leq 2$.  
\begin{enumerate}
\item If $\gamma <\frac{n^2}{4}$, then  
\begin{equation}
\hbox{$ -\infty <\mu_{\gamma, s, \lambda}(\Omega) \leq \mu_{\gamma, s}(\rnp),$ for each $\lambda \in \rr$, }
 \end{equation}
 and
 \begin{equation}
 \sup_{\lambda\in\rr}\mu_{\gamma, s, \lambda}(\Omega)= \mu_{\gamma, s}(\rnp).
 \end{equation}
 In particular, 
\begin{equation} \label{lam}
\sup\limits_{\lambda \in \rr}\mu_{0, 2, \lambda}(\Omega)=\frac{n^2}{4}.
\end{equation}
\item If $\gamma >\frac{n^2}{4}$, then $\mu_{\gamma, s}(\Omega)=-\infty$.
\end{enumerate}
\end{proposition}
Note that the case $\gamma=\frac{n^2}{4}$ is unclear as it seems that anything can happen at that value of $\gamma$. For example, if $\gamma_H(\Omega)<\frac{n^2}{4}$ then $\mu_{\frac{n^2}{4},s}(\Omega)<0$, while if $\gamma_H(\Omega)=\frac{n^2}{4}$ then $\mu_{\frac{n^2}{4},s}(\Omega)\geq 0$. It is our guess that many examples reflecting different regimes can be constructed.

\section{Analysis of the operator $L_\gamma=-\Delta -\frac{\gamma}{|x|^2}$ when $0\in \partial \Omega$}

In the sequel, we shall be looking for geometric conditions on $\Omega$ that insures that $\mu_{\gamma, s}(\Omega) <\mu_{\gamma,s}(\rnp)$. As before, 
we need to compute the functional $J^\Omega_{\gamma, s}$ at bubbles modeled on extremals for $\mu_{\gamma,s}(\rnp)$ and to make a Taylor expansion, hoping that one succeeds in getting below the energy threshold. But at this stage, a difficulty occurs: the extremals for $\mu_{\gamma,s}(\rnp)$ are not explicit, and therefore the coefficients that appear in the estimate of $J^\Omega_{\gamma, s}$ at the bubbles are not explicit enough. One needs to know more about the profile of the solutions for the linear and nonlinear equations involving the operator $L_\gamma$ on $\rnp$.  

As noted in the introduction, the most basic solutions for $L_\gamma u=0$, with $u=0$ on $\partial \rnp$ are of the form $u(x)=x_1|x|^{-\alpha}$, and a straightforward computation yields
$-\Delta (x_1|x|^{-\alpha})=\frac{\alpha(n-\alpha)}{|x|^2}x_1|x|^{-\alpha}$ on $\rnp$, 
which means that
\begin{equation}\nonumber
\hbox{$\left(-\Delta-\frac{\gamma}{|x|^2}\right)\left(x_1|x|^{-\alpha}\right)=0$ on $\rnp$,}
\end{equation}
for $\alpha\in \{\am,\ap\}$ where $\alpha_{\pm}(\gamma):=\frac{n}{2}\pm\sqrt{\frac{n^2}{4}-\gamma}.$ 
This turned out to be a general fact since we shall show that $x\mapsto d(x, \partial \Omega)|x|^{-\am}$ is essentially the profile at $0$ of any variational solution --positive or not-- of equations of the form $L_\gamma u=f(x, u)$ on a domain $\Omega$, as long as the nonlinearity $f$ is dominated by $C(|v|+\frac{|v|^{\crits-1}}{|x|^s})$.  

We use the following terminology. Say that $u\in \dundeux_{loc, 0}$ if there exists $\eta\in C^\infty_c(\rn)$ such that $\eta\equiv 1$ around $0$ and $\eta u\in \dundeux$. Note that if $u\in \dundeux_{loc, 0}$, then  $\eta u\in \dundeux$ for all $\eta\in C^\infty_c(\rn)$. Say that $u\in \dundeux_{loc, 0}$ is a weak solution to the equation 
\[-\Delta u=F\in \left(\dundeux_{loc, 0}\right)^\prime,
\]
 if for any $\varphi\in \dundeux$ and $\eta\in C^\infty_c(\rn)$, we have 
$\int_\Omega (\nabla u,\nabla (\eta\varphi))\, dx=\left\langle F ,\eta\varphi\right\rangle.$

The following theorem was established by Ghoussoub-Robert in \cite{gr4}.

\begin{theorem}[Optimal regularity and Generalized Hopf's Lemma] \label{Hopf} Let $\Omega$ be a smooth domain in $\rn$ such that $0\in \partial \Omega$, and let $f: \Omega\times \rr\to \rr$ be a Caratheodory function such that
$$|f(x, v)|\leq C|v| \left(1+\frac{|v|^{\crits-2}}{|x|^s}\right)\hbox{ for all }x\in \Omega\hbox{ and }v\in \rr.$$
Assume $\gamma<\frac{n^2}{4}$ and  let $u\in \dundeux_{loc, 0}$ be such that for some $\tau>0$, 
\begin{equation}\label{regul:eq.0}
-\Delta u-\frac{\gamma+O(|x|^\tau)}{|x|^2}u=f(x,u)\hbox{ weakly in }\dundeux_{loc, 0}.
\end{equation}
 Then, there exists $K\in\rr$ such that
\begin{equation}\label{eq:hopf.0}
\lim_{x\to 0}\frac{u(x)}{d(x,\partial\Omega)|x|^{-\am}}=K.
\end{equation}
Moreover, if $u\geq 0$ and $u\not\equiv 0$, we then have that $K>0$.
\end{theorem}
This theorem can be seen as a generalization of Hopf's Lemma \cite{gt} in the following sense: when $\gamma=0$ (and therefore $\am=0$), the classical Nash-Moser regularity scheme then yields that $u\in C^1_{loc}$, and when   $u\geq 0$, $u\not\equiv 0$, Hopf's comparison principle yields $\partial_\nu u(0)<0$, which is really a reformulation of \eqref{eq:hopf.0} in the case where  $\am=0$.

The proof of this theorem is quite interesting since, unlike the regular case (i.e., when $L_\gamma=L_0=-\Delta$) or in the   situation when the singularity $0$ is in the interior of the domain $\Omega$, the application of the standard Nash-Moser iterative scheme is not sufficient to obtain the required regularity. Indeed, the scheme only yields the existence of $p_0$, with $1<p_0<\frac{n}{\am-1}$ such that  $u\in L^p$ for all $p<p_0$. Unfortunately, $p_0$ does not reach $\frac{n}{\am-1}$, which is the optimal rate of integration needed to obtain the profile \eqref{eq:hopf.0} for $u$. However, the improved order $p_0$ is enough to allow for the inclusion of the nonlinearity $f(x,u)$ in the linear term of \eqref{regul:eq.0}. We are then reduced to the analysis of the linear equation, that is \eqref{regul:eq.0} with $f(x,u)\equiv 0$. When $u\geq 0$, $u\not\equiv 0$, we get the conclusion by constructing super- and sub- solutions to the linear equation behaving like \eqref{eq:hopf.0}.  

As a corollary, one obtains a relatively detailed description of the profile of variational solutions  of (\ref{one}) on $\rnp$, which improves greatly on a result of Chern-Lin \cite{CL5}, hence allowing us to construct sharper test functions and to prove existence of solutions for (\ref{one}) when $\gamma=\frac{n^2-1}{4}$. 

In order to deal with the remaining cases for $\gamma$, that is when $\gamma \in (\frac{n^2-1}{4}, \frac{n^2}{4})$, Ghoussoub-Robert \cite{gr4} prove the following result which describes the general profile of any positive solution of $L_\gamma u=a(x)u$, albeit variational or not.  

\begin{theorem}[Classification of singular solutions]
Assume $\gamma < \frac{n^2}{4}$ and let $u\in C^2(B_\delta(0)\cap (\overline{\Omega}\setminus\{0\}))$ be such that
\begin{equation}
\left\{\begin{array}{ll}
-\Delta u-\frac{\gamma+O(|x|^\tau)}{|x|^2}u=0 &\hbox{\rm in } \Omega\cap B_\delta(0)\\
u > 0&\hbox{\rm in } \Omega\cap B_\delta(0)\\
u=0&\hbox{\rm on } (\partial\Omega\cap B_\delta(0))\setminus \{0\}.\end{array}\right.
\end{equation}
Then, there exists $K>0$ such that
$$\hbox{either }u(x)\sim_{x\to 0}K\frac{d(x,\partial\Omega)}{|x|^{\am}}\quad \hbox{ or } \quad u(x)\sim_{x\to 0}K\frac{d(x,\partial\Omega)}{|x|^{\ap}}.$$
In the first case, the solution $u$ is variational; in the second case, it is not.
\end{theorem}
This result then allows us to completely classify all positive solutions to $L_\gamma u=0$ on $\rnp$. One can therefore deduce the following.
\begin{proposition} Assume $\gamma < \frac{n^2}{4}$ and let $u\in C^2(\overline{\rnp}\setminus \{0\})$ be such that
\begin{equation}
\left\{\begin{array}{ll}
-\Delta u-\frac{\gamma}{|x|^2}u=0 &\hbox{\rm in } \rnp\\
u>0&\hbox{\rm in }  \rnp\\
u=0&\hbox{\rm on } \partial \rnp.\end{array}\right.
\end{equation}
Then, there exist $\lambda-, \lambda_+\geq 0$ such that
\begin{equation}
u(x)=\lambda_- x_1|x|^{-\am}+\lambda_+ x_1|x|^{-\ap}\hbox{ for all }x\in \rnp.
\end{equation}
\end{proposition}\vskip 10pt

\section{The profile of the extremals for $\mu_{\gamma, s}(\rnp)$} 

The following is a useful description of the solution profile for the extremals on $\rnp$. We shall give below a proof of the symmetry.  

\begin{theorem}\label{th:sym}
Let $n\geq 3$, $s\in [0,2)$, $\gamma<\frac{n^2}{4}$. We consider $u\in D^{1,2}(\rnp)\setminus \{0\}$ such that $u\geq 0$ and 
\begin{equation}\label{sys:sym}
-\Delta u-\frac{\gamma}{|x|^2}u=\frac{u^{\crits-1}}{|x|^s} \hbox{ weakly in }\rnp.
\end{equation}
Then, the following hold:
\begin{enumerate}
\item  $u\circ\sigma=u$ for all isometry of $\rn$ such that $\sigma(\rnp)=\rnp$. In particular, there 
exists $v\in C^2(\rr_+\times \rr)$ such that for all $x_1>0$ and all $x'\in\rr^{n-1}$, 
$$u(x_1,x')=v(x_1,|x'|).$$
\item If $u\not\equiv 0$, then there exist $K_1,K_2>0$ such that
\begin{equation}\nonumber
u(x)\sim_{x\to 0}K_1\frac{x_1}{|x|^{\am}}\quad \hbox{ and } \quad u(x)\sim_{|x|\to +\infty}K_2\frac{x_1}{|x|^{\ap}}.
\end{equation}
\end{enumerate}
\end{theorem}
The above theorem yields in particular, the existence of a solution $U$ for (\ref{sys:sym}) which satisfies for some $C>0$, the estimates
\begin{equation}\label{bnd:U.0}
U(x)\leq C x_1 |x|^{-\ap} \quad \hbox{ and \quad $|\nabla U(x)|\leq C |x|^{-\ap}$ for all $x\in\rnp$. }
\end{equation}
Noting that 
\begin{eqnarray*}
\gamma < \frac{n^2-1}{4} & \Leftrightarrow & \ap-\am > 1, 
\end{eqnarray*}
it follows from   \eqref{bnd:U.0}, that whenever $\gamma<\frac{n^2-1}{4}$, then  $|x'|^2|\partial_1 U|^2=O(|x'|^{2-2\ap})$ as $|x'|\to +\infty$ on $\partial\rnp=\rr^{n-1}$, from which we could deduce that $x'\mapsto |x'|^2|\partial_1 U(x')|^2$ is in $L^1(\partial\rnp)$. This estimate --which does not hold when $\gamma>\frac{n^2-1}{4}$ --  is key for the construction of test functions for $\mu_{\gamma, s}(\Omega)$ based on the solution  $U$ of (\ref{sys:sym}), in the case when $\gamma\leq \frac{n^2-1}{4}$.  

\medskip The proof of symmetry goes as follows. It was established by Chern-Lin \cite{CL5}) for $\gamma <0$ and by Ghoussoub-Robert \cite{gr1} in the case when $\gamma = 0$, a proof which extends immediately to the case $\gamma \geq 0$. Here is a sketch.

Denoting by $\vec{e}_1$ the first vector of the canonical 
basis of $\rn$,  we consider the open ball $D:=B_{1/2}\left(\frac{1}{2}\vec{e}_1\right)$
and define
\begin{equation*}
v(x):=|x|^{2-n}u\left(-\vec{e}_1+\frac{x}{|x|^2}\right)
\end{equation*}
for all $x\in D$. As one checks, $v\in D^{1,2}(D)$ and
\begin{equation}\label{eq:v}
-\Delta v-\gamma \frac{v}{|x|^2\left|x-\vec{e}_1\right|^2}=\frac{v^{\crits-1}}{|x|^s\left|x-\vec{e}_1\right|^s}\hbox{ weakly in }D.
\end{equation}
It then follows from standard regularity theory and Theorem \ref{Hopf} that $v\in C^2(\overline{D}\setminus\{0,\vec{e}_1\})$ and that there exists $K_1,K_2>0$ such that
\begin{equation*}
v(x)\sim_{x\to 0}K_1 \frac{d(x,\partial D)}{|x|^{\am}}\hbox{ and } v(x)\sim_{x\to \vec{e}_1}K_2 \frac{d(x,\partial D)}{|x-\vec{e}_1|^{\am}}.
\end{equation*} 
We now use the moving plane method to prove the symmetry property of  $v$, which 
is defined on a ball. 
For 
$\mu\geq 0$ and $x=(x',x_n)\in \rn$, where $x'\in \rr^{n-1}$ and 
$x_n\in\rr$, we let
$$x_\mu=(x',2\mu-x_n)\hbox{ and }D_\mu=\{x\in D/\, x_\mu\in D\}.$$
It follows from Hopf's Lemma that there exists $\epsilon_0>0$ such that for any $\mu\in 
(1-\epsilon_0,1)$, we have that $D_\mu\neq\emptyset$ 
and $v(x)\geq v(x_\mu)$ for all $x\in D_\mu$ such that $x_n\leq\mu$. We 
let $\mu\geq 0$. We say that $(P_\mu)$ holds if:
\[
\hbox{ $D_{\mu}\neq \emptyset$ and
$v(x)\geq v(x_{\mu})$
for all $x\in D_{\mu}$ such that $x_n\leq\mu$.}
\]
We let
\begin{equation}\label{def:lambda}
\lambda:=\min\left\{\mu\geq 0;\,  (P_{\nu})\hbox{ holds for all }\nu\in 
\left(\mu,1\right) \right\}.
\end{equation}
We claim that $\lambda=0$. Indeed, otherwise we have $\lambda>0$, $D_{\lambda}\neq\emptyset$ and that $(P_{\lambda})$ holds. 
We let
$$w(x):=v(x)-v(x_{\lambda})$$
for all $x\in D_{\lambda}\cap\{x_n<\lambda\}$. Since $(P_{\lambda})$ 
holds, we have that $w(x)\geq 0$ for all $x\in 
D_{\lambda}\cap\{x_n<\lambda\}$. With the equation (\ref{eq:v}) of $v$ and 
$(P_{\lambda})$, we get that\footnote{this is where $\gamma\geq 0$ is used}
\begin{eqnarray*}
-\Delta w&=& \frac{v(x)^{\crits-1}}{|x+|x|^2\vec{e}_1|^s}- 
\frac{v(x_{\lambda})^{\crits-1}}{|x_{\lambda}+|x_{\lambda}|^2\vec{e}_1|^s}+\gamma\left(\frac{v(x)}{|x+|x|^2\vec{e}_1|^2}- 
\frac{v(x_{\lambda})}{|x_{\lambda}+|x_{\lambda}|^2\vec{e}_1|^2}\right)\\
&\geq &v(x_{\lambda})^{\crits-1}\left(\frac{1}{|x+|x|^2\vec{e}_1|^s}- 
\frac{1}{|x_{\lambda}+|x_{\lambda}|^2\vec{e}_1|^s}\right)\\
&&+\gamma v(x_\lambda) \left(\frac{1}{|x+|x|^2\vec{e}_1|^2}- 
\frac{1}{|x_{\lambda}+|x_{\lambda}|^2\vec{e}_1|^2}\right)
\end{eqnarray*}
for all $x\in D_{\lambda}\cap\{x_n<\lambda\}$. With straightforward 
computations, we have that
\begin{eqnarray*}
&& |x_{\lambda}|^2-|x|^2=4\lambda(\lambda-x_n)\\
&& 
|x_{\lambda}-|x_{\lambda}|^2\vec{e}_1|^2-|x-|x|^2\vec{e}_1|^2=(|x_{\lambda}|^2-|x|^2)\left(1+|x_{\lambda}|^2+|x|^2-2x_1)\right)
\end{eqnarray*}
for all $x\in\rn$. It follows that $-\Delta w(x)>0$ for all $x\in 
D_{\lambda}\cap\{x_n<\lambda\}$. Note that we have used that $\lambda>0$. 
It then follows from Hopf's Lemma and the strong comparison principle that
\begin{equation}\label{ppty:w}
w>0\hbox{ in }D_\lambda\cap\{x_n<\lambda\}\hbox{ and }\frac{\partial 
w}{\partial \nu}<0\hbox{ on }D_\lambda\cap\{x_n=\lambda\}.
\end{equation}
By definition, there exists a sequence 
$(\lambda_i)_{i\in\mathbb{N}}\in\rr$ and a sequence 
$(x^i)_{i\in\mathbb{N}}\in D$ such that $\lambda_i<\lambda$, $x^i\in 
D_{\lambda_i}$, $(x^i)_n<\lambda_i$, $\lim_{i\to 
+\infty}\lambda_i=\lambda$ and
\begin{equation}\label{ineq:sym:2}
v(x^i)<v((x^i)_{\lambda_i})
\end{equation}
for all $i\in\mathbb{N}$. Up to extraction a subsequence, we assume that 
there exists $x\in \overline{D}_{\lambda}\cap\{x_n\leq \lambda\}$ such 
that $\lim_{i\to +\infty}x^i=x$ with $x_n\leq \lambda$. Passing to the 
limit $i\to +\infty$ in (\ref{ineq:sym:2}), we get that $v(x)\leq 
v(x_{\lambda})$. It follows from this last inequality and (\ref{ppty:w}) 
that $v(x)-v(x_\lambda)=w(x)=0$, and then $x\in 
\partial(D_{\lambda}\cap\{x_n<\lambda\})$.

\smallskip\noindent{\it Case 1:} If $x\in\partial D$. Then 
$v(x_{\lambda})=0$ and $x_{\lambda}\in\partial D$. Since $D$ is a ball and 
$\lambda>0$, we get that $x=x_{\lambda}\in\partial D$. Since $v$ is $C^1$, 
we get that there exists $\tau_i\in ((x^i)_n,2\lambda_i-(x^i)_n)$ such 
that
$$v(x^i)-v((x^i)_{\lambda_i})=\partial_n v((x')^i,\tau_i)\times 
2((x^i)_n-\lambda_i)$$
Letting $i\to +\infty$, using that $(x^i)_n<\lambda_i$ and 
(\ref{ineq:sym:2}), we get that $\partial_n v(x)\geq 0$. On the other 
hand, we have that
$$\partial_n v(x)=\frac{\partial v}{\partial\nu}(x)\cdot 
(\nu(x)|\vec{e}_n)=\frac{\lambda}{|x-\vec{e}_1/2|}\frac{\partial 
v}{\partial\nu}(x).$$
Therefore $\frac{\partial 
v}{\partial\nu}(x)\leq 0$: this is a contradiction with Hopf's Lemma.

\smallskip\noindent{\it Case 2:} If $x\in D$. Since $v(x_{\lambda})=v(x)$, 
we then get that $x_{\lambda}\in D$. Since $x\in 
\partial(D_{\lambda}\cap\{x_n<\lambda\})$, we then get that $x\in 
D\cap\{x_n=\lambda\}$. With the same argument as in the preceding step, we 
get that $\partial_n v(x)\geq 0$. On the other hand, with (\ref{ppty:w}), we get that $2\partial_n v(x)=\partial_n w(x)<0$. A contradiction.

This proves that $\lambda=0$ in either one of the two cases considered above.  It now follows from the definition (\ref{def:lambda}) of 
$\lambda$ that $v(x',x_n)\geq v(x',-x_n)$ for all $x\in D$ such that 
$x_n\leq 0$. With the same technique, we get the reverse inequality, and 
then, we get that $v(x',x_n)=v(x',-x_n)$ for all $x=(x',x_n)\in D$. In other words, $v$ is symmetric with respect 
to the hyperplane $\{x_n=0\}$. The same analysis holds for any 
hyperplane containing $\vec{e_1}$. Coming back to the initial function 
$u$, this complete the proof of the symmetry of $u$.

\section{Extremals when either $s>0$ or $\{s=0$, $\gamma>0$ and $n\geq 4$\}} 
 
 Recall that if $0\in \partial \Omega$, then $\frac{(n-2)^2}{4} <\gamma_H(\Omega) \leq \frac{n^2}{4}$. If now $\gamma_H(\Omega)\leq \gamma < \frac{n^2}{4}$, then $\mu_{\gamma, s}(\Omega) \leq 0 <\mu_{\gamma, s}(\rnp)$ and it is therefore attained. 
 In this section, we deal with the more interesting cases when $\gamma <\gamma_H(\Omega)\leq \frac{n^2}{4}$. In the sequel, $H_\Omega(0)$ will denote the mean curvature of $\partial \Omega$ at $0$. The orientation is chosen such that the mean curvature of the canonical sphere (as the boundary of the ball) is positive. 
 
  We now outline the proof of the following existence result.

\begin{theorem} \label{main} Let $\Omega$ be a smooth bounded domain in $\rn$ ($n\geq 3$) with $0\in \partial \Omega$ so that $\frac{(n-2)^2}{4} <\gamma_H(\Omega) \leq \frac{n^2}{4}$. Let $0\leq s < 2$ and $\gamma <\gamma_H(\Omega)$. \\
Assume that either  $s>0$ or \{$s=0$, $n\geq 4$ and $\gamma>0$\}.
\begin{enumerate}
\item If $0<\gamma\leq \frac{n^2-1}{4}$, and the mean curvature of $\partial \Omega$ at $0$ is negative, then there are extremals for $\mu_{\gamma, s}(\Omega)$. 

\item If $\frac{n^2-1}{4}<\gamma <\frac{n^2}{4}$, and the Hardy-singular boundary mass $b_\gamma(\Omega)$ is positive, then there are extremals for $\mu_{\gamma, s}(\Omega)$.

 \end{enumerate}

  \end{theorem}

\begin{proof}According to Theorem \ref{tool}, in order to establish existence of extremals,  it suffices to show that  $\mu_{\gamma,s}(\Omega)<\mu_{\gamma,s}(\rnp)$. The rest of the section consists of showing that the above mentioned geometric conditions lead to such gap.

Since either $s>0$ or \{$s=0$, $n\geq 4$ and $\gamma>0$\}, we have seen in section 5 that there exists  $U\in D^{1,2}(\rnp)\setminus \{0\}$, $U\geq 0$, that is a minimizer for $\mu_{\gamma,s}(\rnp)$. In other words, 
\begin{equation}\nonumber
J^{\rnp}_{\gamma,s}(U)=\frac{\int_{\rnp}\left(|\nabla U|^2-\frac{\gamma}{|x|^2}U^2\right)\, dx}{\left(\int_{\rnp}\frac{|U|^{\crits}}{|x|^s}\, dx\right)^{\frac{2}{\crits}}}=\mu_{\gamma,s}(\rnp), 
\end{equation}
and there exists $\lambda>0$ such that
\begin{equation}\left\{\begin{array}{lll}\label{eq:U}
-\Delta U-\frac{\gamma}{|x|^2}U&=\lambda\frac{U^{\crits-1}}{|x|^s}&\hbox{ in }\rnp\\
\hfill U&>0 &\hbox{ in }\rnp\\
\hfill U&=0&\hbox{ in }\partial\rnp.
\end{array}\right.\end{equation}
By the results of section 10,  there are $K_1,K_2>0$ such that 
\begin{equation}\label{asymp:U}
U(x)\sim_{x\to 0}K_1\frac{x_1}{|x|^{\ams}}\hbox{ and }U(x)\sim_{|x|\to +\infty}K_2\frac{x_1}{|x|^{\aps}},
\end{equation}
and $U(x_1,x')=\tilde{U}(x_1, |x'|)$ for all $(x_1,x')\in\rnp$ for some function $\tilde U$ on $\rr^+\times \rr$. 

Here and in the sequel, we write for convenience
$$\aps:=\ap\hbox{ and }\ams:=\am.$$
In particular, there exists $C>0$ such that
\begin{equation}\label{bnd:U}
\hbox{$U(x)\leq C x_1 |x|^{-\aps}\hbox{ and }|\nabla U(x)|\leq C |x|^{-\aps}$ for all $x\in\rnp$.}
\end{equation}
One constructs suitable test-functions for each range of $\gamma$. \\

\noindent For $r>0$, we let $\tilde{B}_{r}:=(-r,r)\times B_{r}^{(n-1)}(0)\subset \rr\times\rr^{n-1},$ 
and denote  
$V_+:=V\cap\rnp$
for any given $V\subset\rn$. Since $\Omega$ is smooth, then, up to a rotation, there exists  $\delta>0$ and $\varphi_0: B_\delta^{(n-1)}(0)\to \rr$ such that $\varphi_0(0)=|\nabla\varphi_0(0)|=0$ and 
\begin{equation}\label{def:phi}
\left\{\begin{array}{llll}
\varphi: & \tilde{B}_{3\delta} & \to &\rn\\
& (x_1,x') & \mapsto & (x_1+\varphi_0(x'), x'),
\end{array}\right.
\end{equation}
that is a diffeomorphism onto its image such that
$$\varphi(\tilde{B}_{3\delta}\cap\rnp)=\varphi(\tilde{B}_{3\delta})\cap\Omega\hbox{ and }\varphi(\tilde{B}_{3\delta}\cap\partial\rnp)=\varphi(\tilde{B}_{3\delta})\cap\partial\Omega.$$
Let $\eta\in C^\infty_c(\rn)$ be such that $\eta(x)=1$ for all $x\in \tilde{B}_\delta$, $\eta(x)=0$ for all $x\not\in \tilde{B}_{2\delta}$. For $\epsilon>0$, define 
\begin{equation}
u_\epsilon(x):=\left(\eta\epsilon^{-\frac{n-2}{2}}U(\epsilon^{-1}\cdot)\right)\circ \varphi^{-1}(x)\hbox{ for }x\in\varphi(\tilde{B}_\delta)\cap\Omega\hbox{ and }0\hbox{ elsewhere.}
\end{equation}
Note that  $(\ue)_{\eps>0}\in \dundeux$. One aims for  a Taylor expansion of $J^\Omega_{s,\gamma}(\ue)$ as $\epsilon\to 0$. 

Given $(a_\epsilon)_{\epsilon>0}\in\rr$, $\Theta_\gamma(a_\epsilon)$ will denote a quantity such that, as $\epsilon\to 0$.
$$\Theta_\gamma(a_\epsilon):=\left\{\begin{array}{ll}
o(a_\epsilon) & \hbox{ if }\gamma<\frac{n^2-1}{4}\\
O(a_\epsilon) & \hbox{ if }\gamma=\frac{n^2-1}{4}\\
\end{array}\right. 
$$
Tedious calculations eventually show that as $\eps\to 0$,
\begin{eqnarray}\label{dev:J:2}
\quad \quad J^\Omega_{\gamma, s}(\ue)&=& \mu_{\gamma,s}(\rnp)\left(1+\epsilon \frac{H_\Omega(0)\int_{\partial\rnp\cap \tilde{B}_{\eps^{-1}\delta}} |x'|^2(\partial_1 U)^2\, dx' }{2(n-1)\lambda \int_{\rnp}\frac{|U|^{\crits}}{|x|^s}\, dx} +\Theta_\gamma(\eps)\right)
\end{eqnarray}
\noindent{\bf Claim 1:} If $\gamma<\frac{n^2-1}{4}$, then we have 
\begin{equation}\label{asymp:J:gamma:small}
J(u_\epsilon)=\mu_{\gamma,s}(\rnp)\left(1+c_{\gamma,s}\cdot H_\Omega(0)\cdot \eps +o(\eps)\right)\hbox{ when }\eps\to 0.
\end{equation}
where  $c_{\gamma,s}>0$ is a positive constant.

Indeed, note that $\gamma<\frac{n^2-1}{4}  \Leftrightarrow  \aps-\ams>1$, and the bound \eqref{bnd:U} yields $|x'|^2|\partial_1 U|^2=O(|x'|^{2-2\aps})$ when $|x'|\to +\infty$. Since $\partial\rnp=\rr^{n-1}$, we then get that $x'\mapsto |x'|^2|\partial_1 U(x')|^2$ is in $L^1(\partial\rnp)$, and therefore \eqref{dev:J:2} yields \eqref{asymp:J:gamma:small}
with
$$c_{\gamma, s}:=\frac{\int_{\partial\rnp} |x'|^2(\partial_1 U)^2\, dx'} {2(n-1)\lambda \int_{\rnp}\frac{|U|^{\crits}}{|x|^s}\, dx}>0.$$

\noindent{\bf Claim 2:} If $\gamma=\frac{n^2-1}{4}$, then we have 
\begin{equation}\label{asymp:J:gamma:crit}
J(u_\epsilon)=\mu_{\gamma,s}(\rnp)\left(1+c'_{\gamma,s}\cdot H_\Omega(0)\cdot \eps\ln\frac{1}{\eps} +o(\eps\ln\frac{1}{\eps})\right)\hbox{ when }\eps\to 0.
\end{equation}
where $c'(\gamma,s)$ is a positive constant.

 Indeed, it follows from \eqref{asymp:U} 
that
$$\lim_{x\to +\infty}|x'|^{\aps}|\partial_1 U(0,x')|=K_2>0.$$
Since $2\aps-2=n-1$, we get that
$$\int_{\partial\rnp\cap \tilde{B}_{\eps^{-1}\delta}} |x'|^2(\partial_1 U)^2\, dx'=\omega_{n-1}K_2^2\ln\frac{1}{\eps}+o\left(\ln\frac{1}{\eps}\right)\quad \hbox{as $\eps\to 0$.}$$
Therefore, \eqref{dev:J:2} yields \eqref{asymp:J:gamma:crit} with
$$c'_{\gamma,s}:=\frac{\omega_{n-1}K_2^2 }{2(n-1)\lambda \int_{\rnp}\frac{|U|^{\crits}}{|x|^s}\, dx}>0.$$

Now we consider the case when $\frac{n^2-1}{4} <\gamma <\frac{n^2}{4}$. One starts by considering $H\in C^2(\Omega)$ as in Proposition \ref{def:mass:intro} such that 
\begin{equation}\label{exp:H}
\hbox{$H(x)=\frac{d(x,\partial\Omega)}{|x|^{\aps}}+b_\gamma(\Omega)\frac{d(x,\partial\Omega)}{|x|^{\ams}}+ o\left(\frac{d(x,\partial\Omega)}{|x|^{\ams}}\right)$ \quad when $x\to 0$. }
\end{equation}
As above, fix $\eta\in C^\infty_c(\rn)$ such that $\eta(x)=1$ for all $x\in \tilde{B}_\delta$, $\eta(x)=0$ for all $x\not\in \tilde{B}_{2\delta}$. Define $\beta$ such that
$$H(x)=\left(\eta\frac{x_1}{|x|^{\aps}}\right)\circ \varphi^{-1}(x)+\beta(x)\quad \hbox{for all $x\in \Omega$. }$$
An essential point underlying the analysis of this case is that since $\aps-\ams<1$, we have
$$\hbox{$|x|=o\left(|x|^{\aps-\ams}\right)$\quad as $x\to 0$. }$$
This implies for example that $\beta\in H^1_0(\Omega)$ and that
\begin{equation}\label{asymp:beta}
\hbox{$\beta(x)=b_\gamma(\Omega)\frac{d(x,\partial\Omega)}{|x|^{\ams}}+ o\left(\frac{d(x,\partial\Omega)}{|x|^{\ams}}\right)$\quad as $x\to 0$. }
\end{equation}
Choose again $U$ as in \eqref{eq:U}. Up to multiplication by a constant, we can assume that 
\begin{equation}\label{asymp:U:bis}
U(x)\sim_{x\to 0}K_1\frac{x_1}{|x|^{\ams}}\hbox{ and }U(x)\sim_{|x|\to +\infty}\frac{x_1}{|x|^{\aps}}.
\end{equation}
The test-functions that one need to analyze here are defined as:
\begin{equation}\label{def:ue:mass}
v_\epsilon(x):=\left(\eta\epsilon^{-\frac{n-2}{2}}U(\epsilon^{-1}\cdot)\right)\circ \varphi^{-1}(x)+\epsilon^{\frac{\aps-\ams}{2}}\beta(x)\hbox{ for }x\in\Omega\hbox{ and }\epsilon>0.
\end{equation}
Note that for any $k\geq 0$, we have
\begin{equation}\label{lim:ue:H}
\lim_{\eps\to 0}\frac{v_\eps}{\eps^{\frac{\aps-\ams}{2}}}=H\hbox{ in }C^k_{loc}(\overline{\Omega}\setminus\{0\}).
\end{equation}
The ultimate goal is to establish the following expansion as $\eps\to 0$.\\

\noindent{\bf Claim 3:} If $\frac{n^2-1}{4} <\gamma <\frac{n^2}{4}$, then we have 
\begin{equation}\label{asymp:J:gamma:crit.bis}
J(u_\epsilon)=\mu_{\gamma,s}(\rnp)\left(1-c_{\gamma,s}''b_\gamma(\Omega)\eps^{\aps-\ams}+o\left(\eps^{\aps-\ams}\right)\right) \, \hbox{as $\eps\to 0$,}
\end{equation}
where 
$$
c''_{\gamma,s}:=\frac{\left(\aps-\frac{n}{2}\right)\omega_{n-1}}{n\lambda \int_{\rnp}\frac{U^{\crits}}{|x|^s}\, dx}>0.
$$ 
\end{proof}

\section{The remaining $3$-dimensional cases}

It is easy to see that if $s=0$ and $\gamma\leq 0$, then $\mu_{\gamma,s}(\Omega)=\mu_{0,0}(\rn)$
and there is no extremal for $\mu_{\gamma,s}(\Omega)$. So the remaining case is when  $n=3$, $s=0$ and $\gamma>0$. 
But, we have seen that in this case, there may or may not be extremals for $\mu_{\gamma, 0}(\rnp)$. If they do exist, we can then argue as before --using the same test functions-- to conclude that there are extremals under the same conditions, that is if either $\gamma\leq\frac{n^2-1}{4}$ and the mean curvature of $\partial \Omega$ at $0$ is negative, or $\gamma>\frac{n^2-1}{4}$ and the Hardy-singular boundary mass $b_\gamma(\Omega)$ is positive. 

However, if no extremal exist for $\mu_{\gamma, 0}(\rnp)$, then we have seen in section 5, that
\begin{equation}
\mu_{\gamma, 0}(\rnp)= \inf_{u\in D^{1,2}(\rn)\setminus\{0\}}\frac{\int_{\rn}|\nabla u|^2\, dx}{\left(\int_{\rn}|u|^{\crit}\, dx\right)^{\frac{2}{\crit}}},
\end{equation}
and therefore we are back to the case where the boundary singularity does not contribute anything. This means that 
one needs to resort to the standard notion of mass $R_{\gamma, 0}(\Omega, x_0)$ for a domain $\Omega$ associated to an interior point $x_0\in \Omega$ and construct test-functions in the spirit of Schoen.

\begin{theorem}  Let $\Omega$ be a bounded smooth domain of $\rr^3$ such that $0\in \partial \Omega$, in such a way that $\frac{1}{4}<\gamma_H(\Omega) \leq \frac{9}{4}$.
\begin{enumerate}
\item If $\gamma_H(\Omega)\leq \gamma <\frac{9}{4}$, then there are extremals for $\mu_{\gamma,0}(\Omega)$.
\item If $0<\gamma <\gamma_H(\Omega)$, 
and if there exists $x_0\in\Omega$ such that $R_{\gamma, 0}(\Omega, x_0)>0$, then there are  extremals for  $\mu_{\gamma,0}(\Omega)$, under either one of the following conditions:
\begin{enumerate}
\item  $\gamma\leq 2$ and the mean curvature of $\partial \Omega$ at $0$ is negative.
\item  $\gamma>2$ and the boundary mass $b_\gamma(\Omega)$ is positive.
\end{enumerate}  
\end{enumerate}
\end{theorem}
\section{Examples of domains with positive mass}\label{sec:ex:mass}
We now assume that $\gamma\in (\frac{n^2-1}{4},\frac{n^2}{4})$ and would like to construct domains with either negative or positive mass. Since $\rnp$ is the main reference set in this theory, one needs to define a notion of mass for certain unbounded sets that include $\rnp$. For that, define 
the following Kelvin transformation.
For any $x_0\in \rn$, let
\begin{equation}
i_{x_0}(x):=x_0+|x_0|^2\frac{x-x_0}{|x-x_0|^2} \quad \hbox{for all $x\in\rn\setminus\{x_0\}$.}
\end{equation}
 The inversion $i_{x_0}$ is clearly the identity map on  $\partial B_{|x_0|}(x_0)$ (the ball of center $x_0$ and of radius $|x_0|$), and in particular $i_{x_0}(0)=0$.

\begin{defi} Say that a domain $\Omega\subset\rn$ ($0\in\partial\Omega$) is {\it conformally bounded} if there exists $x_0\not\in \overline{\Omega}$ such that $i_{x_0}(\Omega)$  is a smooth bounded domain of $\rn$ having both $0$ and $x_0$ on its boundary $ \partial(i_{x_0}(\Omega))$. 
\end{defi}
 
 The following proposition shows that the notion of mass  
 extends to unbounded domains that are conformally bounded. 
 \begin{proposition}\label{prop:pert:mass} Let $\Omega$ be a conformally bounded domain in $\rn$   such that $0\in\partial\Omega$. Assume that $\gamma_H(\Omega)>\frac{n^2-1}{4}$ and that $\gamma\in \left(\frac{n^2-1}{4}, \gamma_H(\Omega)\right)$. Then, up to a multiplicative constant, there exists a unique function $H\in C^2(\overline{\Omega}\setminus\{0\})$ such that
\begin{equation}\label{def:H:2}
\left\{\begin{array}{ll}
-\Delta H-\frac{\gamma}{|x|^2}H=0&\hbox{\rm in }\Omega\\
\hfill H>0&\hbox{\rm  in }\Omega\\
\hfill H=0&\hbox{\rm on }\partial\Omega\setminus\{0\}\\
\hfill H(x)\leq C|x|^{1-\ap}&\hbox{\rm for }x\in\Omega.
\end{array}\right.
\end{equation}
Moreover, there exists $c_1>0$ and $c_2\in\rr$ such that
$$H(x)=c_1\frac{d(x,\partial\Omega)}{|x|^{\ap}}+c_2\frac{d(x,\partial\Omega)}{|x|^{\am}}+o\left(\frac{d(x,\partial\Omega)}{|x|^{\am}}\right) \quad \hbox{as $x\to 0$. }$$
We define the mass $b_\gamma(\Omega):=\frac{c_2}{c_1}$, which is independent of the choice of $H$ in \eqref{def:H:2}.
\end{proposition}
\noindent One can easily check that $\rnp$ is a conformally bounded domain   (take $x_0:=(-1,0,\dots,0)$), and the results of section 10 indicate that $b_\gamma(\rnp)=0$. Since the Hardy b-mass is strictly increasing and continuous, it follows that 
 the mass is negative whenever $\Omega\subset \rnp=T_0\partial\Omega$. In particular, $b_\gamma(\Omega)<0$ if $\Omega$ is convex and $\frac{n^2-1}{4}<\gamma<\frac{n^2}{4}$. 
 
This also suggests that a conformally bounded set strictly containing $\rnp$ must have positive mass, which was proved by Ghoussoub-Robert \cite{gr4}.  

\begin{proposition}\label{prop:ex:pos:unbounded} Let $\Omega$ be a conformally bounded domain such that $0\in\partial\Omega$. Assume that $\gamma_H(\Omega)>\frac{n^2-1}{4}$ and fix $\gamma\in \left(\frac{n^2-1}{4},\gamma_H(\Omega)\right)$. Then $b_\gamma(\Omega)>0$ if $\rnp\subsetneq \Omega$, and $b_\gamma(\Omega)<0$ if $\Omega\subsetneq \rnp$.
\end{proposition}
Note that if the set is not too far from $\rnp$, then it must have a Hardy constant between $\frac{n^2-1}{4}$ and $\frac{n^2}{4}$. The construction of such domains is technical but straightforward. Theorem \ref{prop:ex:any:behave} illustrates that one can construct smooth bounded domains with either positive or negative mass and having any type of behavior at $0$.


\begin{thebibliography}{99}

\bib{ap}{article}{
   author={Atkinson, F. V.},
   author={Peletier, L. A.},
   title={Elliptic equations with nearly critical growth},
   journal={J. Differential Equations},
   volume={70},
   date={1987},
   number={3},
   pages={349--365},
}


\bib{AMP}{article}{
   author={Attar, Ahmed},
   author={Merch{\'a}n, Susana},
   author={Peral, Ireneo},
   title={A remark on the existence properties of a semilinear heat equation involving a Hardy-Leray potential},
   journal={Journal of Evolution Equations},
   note={to appear},
   
}


\bib{aubin}{article}{
   author={Aubin, Thierry},
   title={Probl\`emes isop\'erim\'etriques et espaces de Sobolev},
   language={French},
   journal={J. Differential Geometry},
   volume={11},
   date={1976},
   number={4},
   pages={573--598},
   
}
\bib{17baernstein-unified}{article}{
   author={Baernstein, A.},
   title={A unified approach to symmetrization},
   language={English},
   journal={Partial differential equations of elliptic type, Cortona, Cambridge Univ. Press},
   volume={},
   date={1992},
   number={4},
   pages={57--91},
   
}

\bib{BB}{article}{
   author={Bahri, A.},
   author={Brezis, H.},
   title={Non-linear elliptic equations on Riemannian manifolds with the
   Sobolev critical exponent},
   conference={
      title={Topics in geometry},
   },
   book={
      series={Progr. Nonlinear Differential Equations Appl.},
      volume={20},
      publisher={Birkh\"auser Boston, Boston, MA},
   },
   date={1996},
   pages={1--100},
   
}

\bib{BPZ}{article}{
   author={Bartsch, Thomas},
   author={Peng, Shuangjie},
   author={Zhang, Zhitao},
   title={Existence and non-existence of solutions to elliptic equations
   related to the Caffarelli-Kohn-Nirenberg inequalities},
   journal={Calc. Var. Partial Differential Equations},
   volume={30},
   date={2007},
   number={1},
   pages={113--136},
   
}
\bib{bn}{article}{
   author={Brezis, Ha{\"{\i}}m},
   author={Nirenberg, Louis},
   title={Positive solutions of nonlinear
elliptic equations involving critical exponents},
   journal={Comm.  Pure Appl. Math},
   volume={36},
   date={1983},
   number={},
   pages={437-477},
   
}

\bib{bp}{article}{
   author={Brezis, Ha{\"{\i}}m},
   author={Peletier, Lambertus A.},
   title={Asymptotics for elliptic equations involving critical growth},
   conference={
      title={Partial differential equations and the calculus of variations,
      Vol.\ I},
   },
   book={
      series={Progr. Nonlinear Differential Equations Appl.}
      },
      volume={1},
      publisher={Birkh\"auser Boston, Boston, MA},
   date={1989},
   pages={149--192},
}

\bib{BV}{article}{
 author={Brezis, Ha{\"{\i}}m},
   author={V\'{a}zquez, Juan Leon},
   title={Blowup solutions of some nonlinear elliptic 
problems},
   journal={Revista Mat. Univ. Complutense Madrid}, 
   volume={10},
date={1997},
   pages={443-469},
}



\bib{beck}{article}{
   author={Beckner, William},
   title={On Hardy-Sobolev embedding},
   journal={arXiv:0907.3932v1[math.AP]},
   volume={},
   date={2009},
   number={},
   pages={},
   
}
\bib{Beckner-93}{article}{
  author={Beckner, William},
   title={Sharp {S}obolev inequalities on the sphere and the
 {M}oser-{T}rudinger inequality},
   journal={Ann. of Math.},
   volume={138},
   date={1993},
   number={2},
   pages={213--242},

}

\bib{BE}{article}{
   author={Berestycki, Henri},
   author={Esteban, Maria J.},
   title={Existence and bifurcation of solutions for an elliptic degenerate
   problem},
   language={English, with English and French summaries},
   journal={J. Differential Equations},
   volume={134},
   date={1997},
   number={1},
   pages={1--25},
   
}

\bib{ckn}{article}{
   author={Caffarelli, Luis},
   author={Kohn, Robert V.},
   author={Nirenberg, Louis},
   title={First order interpolation inequalities with weights},
   journal={Compositio Math.},
   volume={53},
   date={1984},
   number={3},
   pages={259--275},
   }
   
   \bib{CHa}{article}{
   author={Cao, D.},
   author={Han, P.},
   title={Solutions for semilinear elliptic equations with critical exponents and Hardy potential},
   journal={J. Differential Equations},
   volume={205},
   date={2004},
   number={3},
   pages={521-537},
   }

 \bib{CHP}{article}{
   author={Cao, D.},
   author={He, X.},
   author={Peng S.},
   title={Positive solutions for some singular critical growth nonlinear elliptic equations},
   journal={Nonlinear Anal.},
   volume={60},
   date={2005},
   pages={589-609},
   }
\bib{CP}{article}{
   author={Cao, D.},
   author={Peng S.},
   title={A note on the sign-changing solutions to elliptic problems with critical Sobolev and
Hardy terms},
   journal={J. Differential Equations},
   volume={193},
   date={2003},
   pages={424-434},
   }


\bib{CW}{article}{
   author={Catrina, Florin},
   author={Wang, Zhi-Qiang},
   title={On the Caffarelli-Kohn-Nirenberg inequalities: sharp constants,
   existence (and nonexistence), and symmetry of extremal functions},
   journal={Comm. Pure Appl. Math.},
   volume={54},
   date={2001},
   number={2},
   pages={229--258},
   
}


\bib{CC}{article}{
   author={Chou, Kai Seng},
   author={Chu, Chiu Wing},
   title={On the best constant for a weighted Sobolev-Hardy inequality},
   journal={J. London Math. Soc. (2)},
   volume={48},
   date={1993},
   number={1},
   pages={137--151},
   
}
\bib{CL2}{article}{
   author={Chen, Chiun Chuan},
   author={Lin, Chang Shou},
   title={Local behavior of singular positive solutions of semilinear
   elliptic equations with Sobolev exponent},
   journal={Duke Math. J.},
   volume={78},
   date={1995},
   number={2},
   pages={315--334},
   
}


\bib{CL1}{article}{
   author={Chen, Chiun-Chuan},
   author={Lin, Chang-Shou},
   title={A spherical Harnack inequality for singular solutions of nonlinear
   elliptic equations},
   journal={Ann. Scuola Norm. Sup. Pisa Cl. Sci. (4)},
   volume={30},
   date={2001},
   number={3-4},
   pages={713--738 (2002)},
   
}
		

\bib{CL4}{article}{
   author={Chern, Jann-Long},
   author={Lin, Chang-Shou},
   title={The symmetry of least-energy solutions for semilinear elliptic
   equations},
   journal={J. Differential Equations},
   volume={187},
   date={2003},
   number={2},
   pages={240--268},
   
}


\bib{CL5}{article}{
   author={Chern, Jann-Long},
   author={Lin, Chang-Shou},
   title={Minimizers of Caffarelli-Kohn-Nirenberg inequalities with the
   singularity on the boundary},
   journal={Arch. Ration. Mech. Anal.},
   volume={197},
   date={2010},
   number={2},
   pages={401--432},
   
   
}
		
\bib{craig}{article}{
   author={Cowan, Craig},
   title={Optimal Hardy inequalities for general elliptic operators with
   improvements},
   journal={Commun. Pure Appl. Anal.},
   volume={9},
   date={2010},
   number={1},
   pages={109--140},
   
}



\bib{DP}{article}{
   author={D{\'a}vila, Juan},
   author={Peral, Ireneo},
   title={Nonlinear elliptic problems with a singular weight on the
   boundary},
   journal={Calc. Var. Partial Differential Equations},
   volume={41},
   date={2011},
   number={3-4},
   pages={567--586},
   
}


\bib{DFP}{article}{
   author={Devyver, Baptiste},
   author={Fraas, Martin},
   author={Pinchover, Yehuda},
   title={Optimal hardy weight for second-order elliptic operator: an answer
   to a problem of Agmon},
   journal={J. Funct. Anal.},
   volume={266},
   date={2014},
   number={7},
   pages={4422--4489},
 
}


\bib{DELT}{article}{
   author={Dolbeault, Jean},
   author={Esteban, Maria J.},
   author={Loss, Michael},
   author={Tarantello, Gabriella},
   title={On the symmetry of extremals for the Caffarelli-Kohn-Nirenberg
   inequalities},
   journal={Adv. Nonlinear Stud.},
   volume={9},
   date={2009},
   number={4},
   pages={713--726},
   
}
	

\bib{d2}{article}{
   author={Druet, Olivier},
   title={Elliptic equations with critical Sobolev exponents in dimension 3},
   language={English, with English and French summaries},
   journal={Ann. Inst. H. Poincar\'e Anal. Non Lin\'eaire},
   volume={19},
   date={2002},
   number={2},
   pages={125--142},
   
}
		
\bib{d4}{article}{
   author={Druet, Olivier},
   title={Optimal Sobolev inequalities and extremal functions. The
   three-dimensional case},
   journal={Indiana Univ. Math. J.},
   volume={51},
   date={2002},
   number={1},
   pages={69--88},
  
}

\bib{dhr}{book}{
   author={Druet, Olivier},
    author={Hebey, Emmanuel},
     author={Robert, Fr\'ed\'eric},
   title={Blow up theory for elliptic PDE's in Riemannian geometry},
 series={Mathematical Notes, Princeton University Press},
  volume={45},
   date={2003},
    
}


\bib{egnell}{article}{
   author={Egnell, H.},
   title={Positive solutions of semilinear equations in cones},
   journal={Trans.  Amer.  Math.  Soc.},
   volume={11},
   date={1992},
   number={},
   pages={191-201},
   
}



\bib{ekg}{article}{
   author={Ekeland, Ivar},
  author={Ghoussoub, Nassif},
   title={Selected new aspects of the calculus of variations in the large},
   journal={Bull. Amer. Math. Soc.},
   volume={32},
   date={2002},
   number={2},
   pages={207--265},
   
}



\bib{F}{article}{
   author={Fall, Mouhamed Moustapha},
   title={On the Hardy-Poincar\'e inequality with boundary singularities},
   journal={Commun. Contemp. Math.},
   volume={14},
   date={2012},
   number={3},
   pages={1250019, 13},
   
}

\bib{FM}{article}{
   author={Fall, Mouhamed Moustapha},
   author={Musina, Roberta},
   title={Hardy-Poincar\'e inequalities with boundary singularities},
   journal={Proc. Roy. Soc. Edinburgh Sect. A},
   volume={142},
   date={2012},
   number={4},
   pages={769--786},
   
}

\bib{FeGa}{article}{
   author={Alberto Ferrero},
   author={Filippo Gazzola},
   title={Existence of Solutions for Singular Critical Growth
Semilinear Elliptic Equations},
   language={English},
   journal={Journal of Differential Equations},
   volume={177},
   date={2001},
   number={},
   pages={494Ð522},
   
}


\bib{FT}{article}{
   author={Filippas, Stathis},
   author={Tertikas, Achilles},
   title={Optimizing improved Hardy inequalities},
   journal={J. Funct. Anal.},
   volume={192},
   date={2002},
   number={1},
   pages={186--233},
   
   }
   
   
\bib{FiPuRo}{article}{
   author={Filippucci, Roberta},
   author={Pucci, Patrizia},
   author={Robert, Fr{\'e}d{\'e}ric},
   title={On a $p$-Laplace equation with multiple critical nonlinearities},
   language={English, with English and French summaries},
   journal={J. Math. Pures Appl. (9)},
   volume={91},
   date={2009},
   number={2},
   pages={156--177},
   
}


\bib{gk}{article}{
   author={Ghoussoub, Nassif},
   author={Kang, Xiao Song},
   title={Hardy-Sobolev critical elliptic equations with boundary
   singularities},
   language={English, with English and French summaries},
   journal={Ann. Inst. H. Poincar\'e Anal. Non Lin\'eaire},
   volume={21},
   date={2004},
   number={6},
   pages={767--793},
   
}

\bib{gm0}{article}{
   author={Ghoussoub, Nassif},
   author={Moradifam, Amir},
   title={On the best possible remaining term in the Hardy  
inequality},
   journal={Proc. Nat. Acad. Sci.},
   volume={105},
   number={37}
   date={2008},
   pages={13746-13751},
   
}

\bib{gm}{book}{
   author={Ghoussoub, Nassif},
   author={Moradifam, Amir},
   title={Functional inequalities: new perspectives and new applications},
   series={Mathematical Surveys and Monographs},
   volume={187},
   publisher={American Mathematical Society, Providence, RI},
   date={2013},
   pages={xxiv+299},
   
}


		
\bib{gr1}{article}{
   author={Ghoussoub, Nassif},
   author={Robert, Fr\'ed\'eric},
   title={The effect of curvature on the best constant in the Hardy-Sobolev
   inequalities},
   journal={Geom. Funct. Anal.},
   volume={16},
   date={2006},
   number={6},
   pages={1201--1245},
   
}
		
\bib{gr2}{article}{
   author={Ghoussoub, Nassif},
   author={Robert, Fr\'ed\'eric},
   title={Concentration estimates for Emden-Fowler equations with boundary
   singularities and critical growth},
   journal={IMRP Int. Math. Res. Pap.},
   date={2006},
   pages={21867, 1--85},
  
}


\bib{gr3}{article}{
   author={Ghoussoub, Nassif},
   author={Robert, Fr{\'e}d{\'e}ric},
   title={Elliptic equations with critical growth and a large set of
   boundary singularities},
   journal={Trans. Amer. Math. Soc.},
   volume={361},
   date={2009},
   number={9},
   pages={4843--4870},
   
}
\bib{gr4}{article}{
   author={Ghoussoub, Nassif},
   author={Robert, Fr{\'e}d{\'e}ric},
   title={On the Hardy-Schr\"odinger operator with a boundary singularity},
   journal={Submitted}, 
   volume={},
   date={2015},
   number={},
   pages={},
   
}
\bib{gr5}{article}{
   author={Ghoussoub, Nassif},
   author={Robert, Fr{\'e}d{\'e}ric},
   title={The remaining cases for a problem involving the Hardy-Schr\"odinger operator with interior singularity},
   journal={In preparation},
   volume={},
   date={2015},
   number={},
   pages={},
   
}


\bib{GY}{article}{
   author={Ghoussoub, Nassif},
   author={Yuan, C.},
   title={Multiple solutions for quasi-linear PDEs involving the critical Sobolev and Hardy exponents},
   journal={Trans. Amer. Math. Soc.},
   date={2000},
   number={12},
   pages={5703--5743},
   
}



\bib{gnn}{article}{
   author={Gidas, Basilis},
   author={Ni, Wei Ming},
   author={Nirenberg, Louis},
   title={Symmetry and related properties via the maximum principle},
   journal={Comm. Math. Phys.},
   volume={68},
   date={1979},
   number={3},
   pages={209--243},
   
}

\bib{gt}{book}{
   author={Gilbarg, David},
   author={Trudinger, Neil S.},
   title={Elliptic partial differential equations of second order},
   series={Classics in Mathematics},
   note={Reprint of the 1998 edition},
   publisher={Springer-Verlag, Berlin},
   date={2001},
   pages={xiv+517},
   
}

\bib{GmVe}{article}{
   author={Gmira, Abdelilah},
   author={V{\'e}ron, Laurent},
   title={Boundary singularities of solutions of some nonlinear elliptic
   equations},
   journal={Duke Math. J.},
   volume={64},
   date={1991},
   number={2},
   pages={271--324},
   
}


\bib{GV}{article}{
   author={Guerch, Bouchaib},
   author={V{\'e}ron, Laurent},
   title={Local properties of stationary solutions of some nonlinear
   singular Schr\"odinger equations},
   journal={Rev. Mat. Iberoamericana},
   volume={7},
   date={1991},
   number={1},
   pages={65--114},
   
}


\bib{HanLin}{book}{
   author={Han, Qing},
   author={Lin, Fanghua},
   title={Elliptic partial differential equations},
   series={Courant Lecture Notes in Mathematics},
   volume={1},
   publisher={New York University, Courant Institute of Mathematical
   Sciences, New York; American Mathematical Society, Providence, RI},
   date={1997},
   pages={x+144},
  
}

\bib{han}{article}{
   author={Han, Zheng-Chao},
   title={Asymptotic approach to singular solutions for nonlinear elliptic
   equations involving critical Sobolev exponent},
   language={English, with French summary},
   journal={Ann. Inst. H. Poincar\'e Anal. Non Lin\'eaire},
   volume={8},
   date={1991},
   number={2},
   pages={159--174},
   }


\bib{hebeybook1}{book}{
    title={Introduction \`a l'analyse non lin\'eaire sur les Vari\'et\'es},
    author={Hebey, Emmanuel},
    publisher={Diderot},
   date={1997},
   address={Paris}
}

\bib{hv}{article}{
   author={Hebey, Emmanuel},
   author={Vaugon, Michel},
   title={The best constant problem in the Sobolev embedding theorem for
   complete Riemannian manifolds},
   journal={Duke Math. J.},
   volume={79},
   date={1995},
   number={1},
   pages={235--279},
   
}



\bib{HLW}{article}{
   author={Hsia, Chun-Hsiung},
   author={Lin, Chang-Shou},
   author={Wadade, Hidemitsu},
   title={Revisiting an idea of Br\'ezis and Nirenberg},
   journal={J. Funct. Anal.},
   volume={259},
   date={2010},
   number={7},
   pages={1816--1849},
  
}


\bib{jaber}{article}{
   author={Jaber, Hassan},
   title={Hardy-Sobolev equations on compact Riemannian manifolds},
   journal={Nonlinear Anal.},
   volume={103},
   date={2014},
   pages={39--54},
   
}




\bib{Jan}{article}{
   author={Jannelli, Enrico},
   title={The Role played by Space Dimensions in Elliptic Critical Problems},
   journal={Jour. Diff. Eq.},
   volume={156},
   date={1999},
   pages={407Ð426},
   
}
\bib{JanLazz}{article}{
   author={Jannelli, Enrico},
    author={Lazzo, Monica},
   title={Fundamental Solutions and Nonlinear Elliptic Problems},
   journal={Ann. Univ. Ferrara 45 (Suppl.)},
   volume={45},
   date={1999},
   pages={155Ð171},
   
}


\bib{Ka}{article}{
   author={Kang, Dongsheng},
     title={On the elliptic problems with critical weighted
SobolevÐHardy exponents},
   journal={Nonlinear Analysis},
   volume={66},
   date={2007},
   number={},
   pages={1037Ð1050},
}

\bib{KP1}{article}{
   author={Kang, Dongsheng},
   author={Peng, Shuangjie},
   title={Positive solutions for singular critical elliptic problems},
   journal={Appl. Math. Lett.},
   volume={17},
   date={2004},
   number={},
   pages={411-416},

}

\bib{KP2}{article}{
   author={Kang, Dongsheng},
   author={Peng, Shuangjie},
   title={Solutions for semilinear elliptic problems with critical
   Sobolev-Hardy exponents and Hardy potential},
   journal={Appl. Math. Lett.},
   volume={18},
   date={2005},
   number={10},
   pages={1094--1100},
   issn={0893-9659},
   review={\MR{2161873 (2006b:35102)}},
   doi={10.1016/j.aml.2004.09.016},
}

\bib{KP3}{article}{
   author={Kang, Dongsheng},
   author={Peng, Shuangjie},
   title={Sign-changing solutions for elliptic problems with critical Sobolev-Hardy exponents},
   journal={J. Math. Anal. Appl.},
   volume={291},
   date={2004},
   number={},
   pages={488-499},

}
\bib{kms}{article}{
   author={Khuri, M. A.},
   author={Marques, F. C.},
   author={Schoen, R. M.},
   title={A compactness theorem for the Yamabe problem},
   journal={J. Differential Geom.},
   volume={81},
   date={2009},
   number={1},
   pages={143--196},
   }


\bib{LeeParker}{article}{
   author={Lee John M.},
    author={Parker Thomas H..},
   title={The Yamabe Problem},
   journal={Bull. A.M.S},
   volume={17},
   date={1987},
   number={1},
   pages={37--91},
   
}


\bib{Lieb}{article}{
   author={Lieb, Elliott H.},
   title={Sharp constants in the Hardy-Littlewood-Sobolev and related
   inequalities},
   journal={Ann. of Math. (2)},
   volume={118},
   date={1983},
   number={2},
   pages={349--374},
   
}



	
\bib{LW3}{article}{
   author={Lin, Chang-Shou},
   author={Wadade, Hidemitsu},
   title={Minimizing problems for the Hardy-Sobolev type inequality with the
   singularity on the boundary},
   journal={Tohoku Math. J. (2)},
   volume={64},
   date={2012},
   number={1},
   pages={79--103},
   
}

	
\bib{LW1}{article}{
   author={Lin, Chang-Shou},
   author={Wang, Zhi-Qiang},
   title={Symmetry of extremal functions for the Caffarrelli-Kohn-Nirenberg
   inequalities},
   journal={Proc. Amer. Math. Soc.},
   volume={132},
   date={2004},
   number={6},
   pages={1685--1691 (electronic)},
   
}

\bib{LW2}{article}{
   author={Lin, Chang-Shou},
   author={Wang, Zhi-Qiang},
   title={Least-energy solutions to a class of anisotropic elliptic equations in cones of $\mathbb{R}^N_+$},
   note={Preprint},
   
}


\bib{lions1}{article}{
   author={Lions, Pierre-Louis},
   title={The concentration-compactness principle in the calculus of
   variations. The limit case. I},
   journal={Rev. Mat. Iberoamericana},
   volume={1},
   date={1985},
   number={1},
   pages={145--201},
  
}

\bib{lions2}{article}{
   author={Lions, Pierre-Louis},
   title={The concentration-compactness principle in the calculus of
   variations. The limit case. II},
   journal={Rev. Mat. Iberoamericana},
   volume={1},
   date={1985},
   number={2},
   pages={45--121},
  
}

\bib{LZ}{article}{
   author={Li, Yanyan},
   author={Zhu, Meijun},
   title={Yamabe type equations on three-dimensional Riemannian manifolds},
   journal={Commun. Contemp. Math.},
   volume={1},
   date={1999},
   number={1},
   pages={1--50},
   
}	


\bib{MMP}{article}{
   author={Marcus, Moshe},
   author={Mizel, Victor J.},
   author={Pinchover, Yehuda},
   title={On the best constant for Hardy's inequality in $\bold R^n$},
   journal={Trans. Amer. Math. Soc.},
   volume={350},
   date={1998},
   number={8},
   pages={3237--3255},
   
}
		

\bib{PinchoverAIHP}{article}{
   author={Pinchover, Yehuda},
   title={On positive Liouville theorems and asymptotic behavior of
   solutions of Fuchsian type elliptic operators},
   journal={Ann. Inst. H. Poincar\'e Anal. Non Lin\'eaire},
   volume={11},
   date={1994},
   number={3},
   pages={313--341},
   
}

\bib{PT}{article}{
   author={Pinchover, Yehuda},
   author={Tintarev, Kyril},
   title={Existence of minimizers for Schr\"odinger operators under domain
   perturbations with application to Hardy's inequality},
   journal={Indiana Univ. Math. J.},
   volume={54},
   date={2005},
   number={4},
   pages={1061--1074},
   
}


	
	
\bib{PS}{article}{
   author={Pucci, Patrizia},
   author={Servadei, Raffaella},
   title={Existence, non-existence and regularity of radial ground states
   for $p$-Laplacain equations with singular weights},
   journal={Ann. Inst. H. Poincar\'e Anal. Non Lin\'eaire},
   volume={25},
   date={2008},
   number={3},
   pages={505--537},
   
}
	

 
 \bib{r1}{article}{
 author={Robert, Fr\'ed\'eric},
 title={Existence et asymptotiques optimales des fonctions de Green des op\'erateurs elliptiques d'ordre deux (Existence and optimal asymptotics of the Green's functions of second-order elliptic operators)},
note={Unpublished notes},
date={2010},
}


 \bib{rodemich}{article}{
 author={Rodemich, E.},
 title={The Sobolev inequalities with best possible constants},
note={Analysis Seminar at California Institute of Technology},
date={1966},
}


\bib{RW}{article}{
   author={Ruiz, David},
   author={Willem, Michel},
   title={Elliptic problems with critical exponents and Hardy potentials},
   journal={J. Differential Equations},
   volume={190},
   date={2003},
   number={2},
   pages={524--538},
   
}


\bib{schoen1}{article}{
   author={Schoen, Richard},
   title={Conformal deformation of a Riemannian metric to constant scalar
   curvature},
   journal={J. Differential Geom.},
   volume={20},
   date={1984},
   number={2},
   pages={479--495},
   
}


\bib{SY}{article}{
   author={Schoen, Richard},
   author={Yau, Shing-Tung},
   title={Conformally flat manifolds, Kleinian groups and scalar curvature},
   journal={Invent. Math.},
   volume={92},
   date={1988},
   number={1},
   pages={47--71},
   
}

\bib{sz}{article}{
   author={Schoen, Richard},
   author={Zhang, Dong},
   title={Prescribed scalar curvature on the $n$-sphere},
   journal={Calc. Var. Partial Differential Equations},
   volume={4},
   date={1996},
   number={1},
   pages={1--25},
  
}


\bib{SmetsTAMS}{article}{
   author={Smets, Didier},
   title={Nonlinear Schr\"odinger equations with Hardy potential and
   critical nonlinearities},
   journal={Trans. Amer. Math. Soc.},
   volume={357},
   date={2005},
   number={7},
   pages={2909--2938 (electronic)},
   
}

\bib{st}{book}{
   author={Struwe, Michael},
   title={Variational methods},
   series={Ergebnisse der Mathematik und ihrer Grenzgebiete. 3. Folge. A
   Series of Modern Surveys in Mathematics},
   volume={34},
   edition={4},
   note={Applications to nonlinear partial differential equations and
   Hamiltonian systems},
   publisher={Springer-Verlag, Berlin},
   date={2008},
   pages={xx+302},
   
}


\bib{Tal}{article}{
   author={Talenti, Giorgio},
   title={Best constant in Sobolev inequality},
   journal={Ann. Mat. Pura Appl. (4)},
   volume={110},
   date={1976},
   pages={353--372},
   
}
\bib{Tru}{article}{
   author={Trudinger, Neil},
   title={Remarks concerning the conformal deformation of Riemannian structures on compact manifolds},
   journal={Ann. Scuola Norm. Sup. Pisa},
   volume={22},
    date={1968},
   pages={},
   
}
\bib{YC}{article}{
   author={H. Yang},
    author={J. Chen},
   title={A result on Hardy-Sobolev critical elliptic equations with boundary singularities},
   journal={Commun. Pure Appl. Anal.},
   volume={6},
      number={1},
   date={2007},
   pages={191Ð201},
   
}



\end{thebibliography}
\end{document}